\documentclass{amsart}
\usepackage{amssymb, amsmath,amsthm}
\usepackage[dvipsnames]{xcolor}
\usepackage[colorlinks=true,hyperindex, linkcolor=NavyBlue, pagebackref=false, citecolor=magenta,pdfpagelabels, linktoc=all]{hyperref}
\usepackage{breakurl}
\usepackage{stmaryrd}
\usepackage{tikz-cd}
\usepackage[all]{xy}

\newtheorem{prop}{Proposition}[section]

\newtheorem*{lem*}{Lemma}
\newtheorem{lem}[prop]{Lemma}
\newtheorem{cor}[prop]{Corollary}
\newtheorem{thm}[prop]{Theorem}
\theoremstyle{definition}
\newtheorem{defi}[prop]{Definition}
\theoremstyle{remark}
\newtheorem{examp}[prop]{Example}
\newtheorem{warning}[prop]{Warning}
\newtheorem{remar}[prop]{Remark}

\DeclareMathAlphabet{\mathpzc}{OT1}{pzc}{m}{it}

\DeclareMathOperator{\End}{End}
\DeclareMathOperator{\Hom}{Hom}

\DeclareMathOperator{\Res}{Res}
\DeclareMathOperator{\Sym}{Sym}

\DeclareMathOperator{\GL}{GL}

\DeclareMathOperator{\mrs}{mrs}
\DeclareMathOperator{\WD}{WD}
\DeclareMathOperator{\Gal}{Gal}

\DeclareMathOperator{\tr}{tr}

\DeclareMathOperator{\Irr}{Irr}

\DeclareMathOperator{\Spec}{Spec}
\DeclareMathOperator{\MaxSpec}{m-Spec}

\DeclareMathOperator{\id}{id}

\DeclareMathOperator{\Rep}{Rep}
\DeclareMathOperator{\coker}{coker}

\DeclareMathOperator{\Ext}{Ext}

\DeclareMathOperator{\Art}{Art}

\DeclareMathOperator{\Sets}{Sets}

\newcommand{\Q}{\mathbb{Q}}
\newcommand{\Qp}{\mathbb {Q}_p}
\newcommand{\Zp}{\mathbb{Z}_p}
\newcommand{\Qpbar}{\overline{\mathbb{Q}}_p}

\newcommand{\Eins}{\mathbf 1}
\newcommand{\PP}{\mathcal P}

\newcommand{\FF}{\mathcal F}
\newcommand{\LL}{\mathrm{LL}}

\newcommand{\ZZ}{\mathbb Z}

\newcommand{\QQ}{\mathbb Q}

\newcommand{\ab}{\mathrm{ab}}
\newcommand{\Fp}{\mathbb F_p}

\newcommand{\mm}{\mathfrak m}

\newcommand{\wt}{\widetilde}

\newcommand{\OO}{\mathcal O}

\DeclareMathOperator{\ad}{ad}
\DeclareMathOperator{\adbar}{\overline{\ad}}
\DeclareMathOperator{\adz}{\ad^0}

\DeclareMathOperator{\wtimes}{\widehat{\otimes}}

\DeclareMathOperator{\Supp}{Supp}

\newcommand{\nn}{\mathfrak n}

\newcommand{\pp}{\mathfrak p}

\newcommand{\br}[1]{\llbracket #1\rrbracket}
\newcommand{\gl}{\mathfrak{g}}
\newcommand{\lv}{\mathfrak l}
\newcommand{\qq}{\mathfrak{q}}

\newcommand{\cont}{\mathrm{cont}}

\newcommand{\HT}{\mathrm{HT}}

\newcommand{\set}[1]{\{#1\}}

\newcommand{\rhobar}{\overline{\rho}}

\newcommand{\Std}{\mathrm{Std}}

\newcommand{\Gm}{\mathbb G_m}

\newcommand{\ps}{\mathrm{ps}}
\newcommand{\CH}{\mathrm{CH}}
\newcommand{\gen}{\mathrm{gen}}
\DeclareMathOperator{\PGL}{PGL}
\DeclareMathOperator{\Spf}{Spf}
\newcommand{\rig}{\mathrm{rig}}

\newcommand{\univ}{\mathrm{univ}}
\newcommand{\Xbar}{\overline{X}}
\newcommand{\Dbar}{\overline{D}}
\newcommand{\irr}{\mathrm{irr}}
\newcommand{\diag}{\mathrm{diag}}
\newcommand{\zl}{\mathfrak z_L}
\DeclareMathOperator{\spcl}{spcl}
\DeclareMathOperator{\nspcl}{n-spcl}
\newcommand{\cris}{\mathrm{cris}}
\newcommand{\prnc}{\mathrm{prnc}}
\newcommand{\spcd}{\mathrm{spcd}}

\newcommand{\cyc}{\mathrm{cyc}}
\newcommand{\cha}{\mathrm{char}}
\newcommand{\trdeg}{\mathrm{trdeg}}
\newcommand{\Kirr}{\mathrm{Kirr}}
\newcommand{\Kred}{\mathrm{Kred}}
\newcommand{\cred}{\mathrm{Cred}}

\title{On local Galois deformation rings}

\author{Gebhard B\"ockle}
\address{Ruprecht-Karls-Universit\"{a}t Heidelberg}
\email{gebhard.boeckle@iwr.uni-heidelberg.de}

\author{Ashwin Iyengar}
\address{Johns Hopkins University}
\email{iyengar@jhu.edu}

\author{Vytautas Pa\v{s}k\={u}nas}
\address{Universit\"{a}t Duisburg-Essen}
\email{paskunas@uni-due.de}

\date{\today.}
\begin{document} 
\maketitle

\begin{abstract} We show that framed deformation
rings of mod $p$ representations of the absolute Galois group of a $p$-adic local field are complete
intersections of expected dimension. We determine 
their irreducible components and show that they and 
their special fibres are normal and  complete
intersection. As an application 
we prove density results of loci with prescribed 
$p$-adic Hodge theoretic properties.
\end{abstract}

\tableofcontents

\section{Introduction}

Let $p$ denote any prime number, let $F$ be a finite extension of $\Qp$, and let $G_F$ denote its absolute Galois group. Let $L$ be a another finite extension of $\Qp$ with ring of integers $\OO$, uniformizer $\varpi$ and residue field $k=\OO/\varpi$. Fix a continuous representation $\rhobar:G_F\rightarrow \GL_d(k)$ and denote by $D^{\square}_{\rhobar}: \mathfrak A_{\OO}\rightarrow \mathrm{Sets}$ the functor from the category $\mathfrak A_{\OO}$ of local Artinian $\OO$-algebras with residue field $k$ to the category of sets, such that for $(A,\mm_A)\in \mathfrak A_{\OO}$, $D^{\square}_{\rhobar}(A)$ is the set of continuous representations $\rho_A: G_F\rightarrow \GL_d(A)$, such that 
$\rho_A \pmod{\mm_A}=\rhobar$. The 
functor $D^{\square}_{\rhobar}$ of framed deformations of $\rhobar$ is pro-represented by a complete local 
Noetherian $\OO$-algebra $R^{\square}_{\rhobar}$ (with residue field~$k$). 

Our first main result completely settles a folklore conjecture on ring-theoretic properties of $R^{\square}_{\rhobar}$ that  can be traced  back to the foundational work of Mazur~\cite[Section 1.10]{Mazur_GQ}:
\begin{thm}[Corollary \ref{ci}]\label{thm_intro-1}
The ring $R^{\square}_{\rhobar}$ is a local complete intersection, flat over $\OO$ and of relative dimension $d^2+d^2[F:\Qp]$. In particular, every continuous representation $\rhobar: G_F\rightarrow \GL_d(k)$ has a lift to characteristic zero.
\end{thm}

Obstruction theory provides a presentation $R^\square_{\rhobar} = \OO\br{x_1,\dots,x_r}/(f_1, \ldots, f_s)$ with $r$ equal to the dimension of the tangent space and $s$ equal to $\dim H^2(G_F, \ad\rhobar)$. The Euler--Poincar\'e characteristic formula from local class field theory gives
    \[ r-s=d^2+d^2[F:\Qp]. \]
Our theorem proves that $\dim R^\square_{\rhobar}/\varpi$ is given by this cohomological quantity, the \textit{expected dimension} in the spirit of the \textit{Dimension Conjecture} of Gouv\^{e}a from \cite[Lecture 4]{Gouvea}. Having the expected dimension implies that $\varpi,f_1, \ldots, f_s$ is a regular sequence and that $R^\square_{\rhobar}$ is a local complete intersection. It also implies
(see \cite[Lemma 7.5]{DDR_GV}) that the derived deformation ring of $\rhobar$ as introduced 
by Galatius and Venkatesh in \cite{DDR_GV} (see also \cite{DDR_Cai}) 
is homotopy discrete, which means the derived deformation theory 
of $\rhobar$ does not contain more information than
the usual deformation 
theory of $\rhobar$. 
Theorem \ref{thm_intro-1}  is  used in the  forthcoming work of 
Matthew Emerton, Toby Gee and Xinwen Zhu on derived stacks of global Galois representations. 

Our second main result completely describes the connected components of the space $\Spec R^{\square}_{\rhobar}[1/p]$ as envisioned in~\cite{bjpp}. Let $\mu:=\mu_{p^\infty}(F)\subset F^\times$ be the $p$-power torsion subgroup and suppose that $L$ is sufficiently large. Let $R_{\det\rhobar}$ denote the universal deformation ring of the one dimensional representation $\det\rhobar$.

\begin{thm}[Corollaries \ref{ci_chi},
\ref{sp_normal}, \ref{chi_normal}, \ref{BJ_conjecture}, Proposition \ref{prop_det_to_square}]\label{thm_intro-2}
The natural  map $R_{\det\rhobar}\to R^{\square}_{\rhobar}$, induced by sending 
a deformation of $\rhobar$ to its determinant, is flat and induces a bijection of connected components
\begin{equation}\label{eq_components}
\pi_0(\Spec R^\square_{\rhobar}[1/p])\to \pi_0(\Spec R_{\det\rhobar}[1/p] ).
\end{equation}

Labeling these components in a natural way by characters $\chi:\mu\to \OO^\times$,  the connected components of $\Spec R^\square_{\rhobar}[1/p] $ are in natural bijection with the irreducible components $\Spec R^{\square,\chi}_{\rhobar}$ of $\Spec R^{\square}_{\rhobar}$, and the rings $R^{\square,\chi}_{\rhobar}$ and  $R^{\square,\chi}_{\rhobar}/\varpi$ 
are normal domains and complete intersections.
\end{thm}

As a consequence we obtain the following useful Corollary. 

\begin{cor}[Corollary \ref{generic_fibre_normal}] $R^{\square}_{\rhobar}$ is reduced and $R^{\square}_{\rhobar}[1/p]$ is normal. 
\end{cor} 

We would like to highlight the following result for the amusement of the reader.
\begin{thm}[Corollary \ref{factorial_irr}] If $\rhobar$ is absolutely irreducible then $R^{\square, \chi}_{\rhobar}$ and $R^{\square, \chi}_{\rhobar}/\varpi$ are factorial, 
except in the case $d=2$, $F=\QQ_3$ and $\rhobar\cong \rhobar(1)$. 
\end{thm}
Let $\psi: G_F\rightarrow \OO^{\times}$ be a continuous character lifting $\det \rhobar$. Let $R^{\square,\psi}_{\rhobar}$ be the quotient of $R^{\square}_{\rhobar}$ parameterizing deformations with determinant equal to $\psi$. 

\begin{thm}[Corollary \ref{ci_psi}, Theorem \ref{psi_normal}]\label{thm_intro_fix_det} The rings 
$R^{\square, \psi}_{\rhobar}$, $R^{\square, \psi}_{\rhobar}/\varpi$ are normal domains and complete intersections of dimension $\dim R^\square_{\rhobar} - \dim R_{\det\rhobar}+1$ and $\dim R^\square_{\rhobar} - \dim R_{\det\rhobar}$, respectively. Moreover, $R^{\square, \psi}_{\rhobar}$ is $\OO$-flat. 
\end{thm}

Our work builds in an essential way on the work of 
GB--Juschka \cite{bj} on the special fibres of the deformation rings 
of pseudo-characters (i.e. pseudo-representations) of $G_F$.
The paper \cite{bj} draws its inspiration from the work of
Chenevier \cite{che_unpublished}, who studied rigid analytic generic 
fibres of these rings.
Our results in turn imply that the 
rigid analytic spaces appearing in \cite{che_unpublished} are normal 
(Corollaries \ref{Rps_normal}, \ref{Rpspsi_normal}).

The knowledge of irreducible 
components of 
$R^{\square}_{\rhobar}$ allows us to refine the existing results 
on the Zariski density of the locus with prescribed $p$-adic Hodge theoretic properties. 

\begin{thm}[Theorem \ref{density}]\label{thm_intro-3}
Suppose that $p$ does not divide $2d$.
Let $\Sigma$ be a subset of the maximal spectrum of  $R^{\square}_{\rhobar}[1/p]$
parameterizing any of the following sets of lifts of $\rhobar$ to characteristic 
zero: 
\begin{enumerate}
\item crystalline lifts with regular Hodge--Tate weights;
\item potentially crystabelline lifts with fixed regular Hodge--Tate weights;
\item potentially crystalline supercuspidal lifts with fixed regular Hodge--Tate
weights.
\end{enumerate} 
Then $\Sigma$  is Zariski dense in $\Spec R^\square_{\rhobar}[1/p]$.
\end{thm}

The assumption $p\nmid 2d$ enters
via our use of the patched module $M_{\infty}$ constructed in \cite{6auth}. The paper \cite{6auth} is applicable whenever $\rhobar$ has a potentially diagonalisable lift. It has been proved recently by  Emerton--Gee \cite{EG_stack}, using the Emerton--Gee stack, that this holds for all 
$\rhobar$. The rest of our paper is independent of \cite{EG_stack}.
We show that the action of $R^{\square}_{\rhobar}$ on $M_{\infty}$ is faithful (Theorem
\ref{faithful_square}), which allows us to deduce Theorem \ref{thm_intro-3} from \cite{EP}.

Partial results towards Theorem~\ref{thm_intro-1} and also towards the more recent question solved by Theorem~\ref{thm_intro-2} appear in many places, e.g.~\cite{babnik}, \cite{bjpp}, \cite{boeckle}, \cite{CDP2}, \cite{Iy}, \cite{nakamura}, in special cases.  However, these papers either compute with equations defining 
the rings, or impose assumptions on $\rhobar$ so that the deformation theory of $\rhobar$ is essentially unobstructed which leads to only one irreducible component. Although there is some overlap in ideas with \cite{Iy}, the argument 
in our paper is rather different as we don't compute with equations. We refer 
the reader to Section \ref{sec_dense} for a more detailed discussion of
the previous results on Zariski density of specific loci in $\Spec R^{\square}_{\rhobar}$ and 
to Remark \ref{rel_to_BIP2} for a detailed explanation of the relation between Theorem \ref{thm_intro-3} 
and our more recent results in \cite{BIP2}.

\begin{remar}
In the theorems above we work with framed deformation rings. Our results 
also carry over to the versal deformation rings (which coincide with the 
universal deformation rings if $\rhobar$ has only scalar endomorphisms), 
by exploiting the fact that framed deformation rings are formally smooth 
over versal deformation rings (see e.g.~\cite[Lemma 2.1]{Iy}) and using \cite[Theorem 2.3.6, Corollary 2.2.23 (a)]{BH}. 
\end{remar}

\subsection{Complete intersection}
We now give an overview of the proof of Theorem \ref{thm_intro-1}. To do so, we introduce two further key players. The first are determinant laws, which we refer to as \textit{pseudo-characters} throughout the paper, and their deformations. Let $\overline{D}: k[G_F] \rightarrow k$ be the pseudo-character attached to $\rhobar$ as defined in \cite{che_durham}. Let $D^{\ps}: \mathfrak A_{\OO}\rightarrow \mathrm{Sets}$ be the functor mapping $(A,\mm_A)\in \mathfrak A_{\OO}$ to  the set $D^{\ps}(A)$ of continuous $A$-valued $d$-dimensional pseudo-characters $D:A[G_F]\to A$ with $\Dbar= D\pmod{\mm_A}$. The functor $D^{\ps}$ is pro-representable by a complete local Noetherian $\OO$-algebra $(R^{\ps}, \mm_{R^{\ps}})$, see \cite[Section 3.1]{che_durham}. The ring $R^{\ps}$ has been well understood in the recent work of GB--Juschka \cite{bj}, who have determined the dimension of its special fiber and showed that the absolutely irreducible locus is dense in the special fibre. In particular, they show the following:
\begin{thm}[GB--Juschka {\cite[Theorem 5.5.1(a)]{bj}}] \label{thm_GB}
The ring $R^{\ps}/\varpi$ is equi-di\-men\-sio\-nal of dimension $1+d^2[F:\Qp]$.
\end{thm}

Mapping a lifting of $\rhobar$ to its associated pseudo-character induces a natural transformation $D^{\square}_{\rhobar}\rightarrow D^{\ps}$ and thus 
a map of local $\OO$-algebras $R^{\ps}\rightarrow R^{\square}_{\rhobar}$. Our basic idea is to study $R^{\square}_{\rhobar}$ by studying the fibres of this
map. Our initial observation was that the difference between the expected dimension of $R^{\square}_{\rhobar}/\varpi$ and the dimension computed in Theorem \ref{thm_GB} is $d^2-1$, which 
is the dimension of $\PGL_d$. On the other hand a fibre at a point corresponding to an 
absolutely irreducible pseudo-character can be shown to be isomorphic to $\PGL_d$. This led 
us naturally to study fibres at other points. 
In fact it is technically more convenient to introduce an intermediate ring $R^{\ps}\rightarrow A^{\gen}\rightarrow R^{\square}_{\rhobar}$, depending on $\Dbar$ and not on $\rhobar$ itself, such that $A^{\gen}$ is of finite type over $R^{\ps}$ and $R^{\square}_{\rhobar}$ is a completion of $A^{\gen}$ 
at a maximal ideal. This is our second key player. 

To describe $A^\gen$, let $D^u: R^{\ps}\br{G_F}\rightarrow R^{\ps}$ be the universal pseudo-character lifting $\overline{D}$ and let $\CH(D^u)$ be the closed two-sided ideal of $R^{\ps}\br{G_F}$ defined in \cite[Section 1.17]{che_durham}, so that
    \[ E:= R^{\ps}\br{G_F}/ \CH(D^u) \]
is the largest quotient of $R^{\ps}\br{G_F}$ for which the Cayley--Hamilton theorem for $D^u$ holds. Following \cite[Section 1.17]{che_durham} we will call such algebras Cayley--Hamilton $R^{\ps}$-algebras of degree $d$. By \cite[Proposition 3.6]{WE_alg} the ring $E$ is a finitely generated $R^{\ps}$-module. Now a construction of Procesi \cite{Pro87} gives a commutative $R^{\ps}$-algebra $A^{\gen}$ together with a homomorphism
    \[ j: E\rightarrow M_d(A^{\gen}) \]
of Cayley--Hamilton $R^{\ps}$-algebras satisfying the following \textit{universal property}: if $f: E\rightarrow M_d(B)$ is a map of Cayley--Hamilton $R^{\ps}$-algebras for a commutative $R^{\ps}$-algebra $B$, then there is a unique map $\tilde{f}: A^{\gen}\rightarrow B$ of $R^{\ps}$-algebras such that $f= M_d(\tilde{f})\circ j$. We give further details in Lemma \ref{existence_Agen} in the main text. The superscript \textit{gen} in $A^{\gen}$ stands for \textit{generic matrices}. 

Since $E$ is finitely generated as $R^{\ps}$-module, the construction of Procesi shows that $A^{\gen}$ is of finite type over $R^{\ps}$. Moreover one has an algebraic action of $\GL_d$ on $X^{\gen}:= \Spec A^{\gen}$, which for every a $R^{\ps}$-algebra $B$ and point $f: E\rightarrow M_d(B)$ in $X^{\gen}(B)$ is simply given by conjugation of matrices. Wang-Erickson has studied the quotient stack $[X^{\gen}/\GL_d]$ in his thesis 
\cite{WE_thesis}, \cite{WE_alg} and $X^{\gen}$ is isomorphic to $\mathrm{Rep}^\square_{\Dbar} = \mathrm{\Rep}^\square_{E,D^u}$ as defined in \cite[Theorem 3.8]{WE_alg}. It is an important observation that to $\pi:X^\gen\to X^\ps:=\Spec R^\ps $ we can apply geometric invariant theory (GIT). As shown in  \cite[Theorem 2.20]{WE_alg}, the induced morphism $  X^{\gen}\sslash G \rightarrow X^{\ps}$ is an adequate homeomorphism in the sense of \cite[Definition 3.3.1]{alper}. 

Our first important result on dimensions is for $ \Xbar^{\gen} :=\Spec A^\gen/\varpi$.
\begin{thm}[{Theorem~\ref{main}, Lemma~\ref{dim_gen_sp}}]\label{main-intro}
We have $$\dim X^{\gen}[1/p] \le \dim \Xbar^{\gen} \le d^2+ d^2[F:\Qp].$$
 \end{thm}

To prove the second inequality of Theorem~\ref{main-intro} we decompose the base of the finite type morphism 
$\bar\pi : \Xbar^\gen\to \Xbar^\ps=\Spec R^\ps/\varpi$ 
as a finite union $\Xbar^\ps=\bigcup_i U_i$ of locally closed subschemes $U_i$. 
The points of the $U_i$ correspond to semi-simple degree $d$ representations of $G_F$ with certain (degree) conditions on the irreducible constituents. The work \cite{bj} gives dimension estimates on the $U_i$. We  combine them with bounds on the  dimensions of the fibres at the closed points of $U_i$, obtained using GIT, and with results on $\bar\pi^{-1}(U_i)\to U_i$ from commutative algebra. In Subsection~\ref{sec_fibre} we  analyze in detail the dimensions of the fibres of $\pi$ at points $y$ of $X^\ps$ valued either in finite fields containing $k$ or local fields whose residue fields contain $k$. The analysis at such points suffices for all results in this paper. The commutative algebra results, used to analyze $\bar\pi^{-1}(U_i)\to U_i$ and to give the first inequality, are proved in Subsection~\ref{sec_dim_sp}. The key technical improvement 
working with $X^{\gen}$ instead of $\Spec R^{\square}_{\rhobar}$ directly, is that the fibres are of finite type over $\kappa(y)$. 

We apply the bounds from Theorem~\ref{main-intro} to the study of lifting rings of continuous residual representations $\rho_x:G_F\to \GL_d(\kappa(x))$ where $x$ is a point of $X^\gen$ whose residue field $\kappa(x)$ is a finite or a local field. We distinguish three cases:
\begin{enumerate}
    \item If $\kappa(x)$ is a finite extension $k'$ of $k$ then we set $\Lambda$ to be the ring of integers $\OO'$ of the unramified extension $L'$ of $L$ with residue field $k'$.
    \item If $\kappa(x)$ is a finite extension of $L$ then we set $\Lambda$ to be $\kappa(x)$.
    \item If $\kappa(x)$ is a local field that contains $k$ and if $k'$ denotes its residue field, then we take as $\Lambda$ a Cohen ring of $\kappa(x)$ (with the natural topology) tensored over the Witt vector ring $W(k')$ with $\OO'$.
\end{enumerate}

Let $\mathfrak A_{\Lambda}$ be the category of local Artinian $\Lambda$-algebras $(A,\mm_A)$ with residue field $\kappa(x)$. We equip the rings $A$ with a natural topology, and we consider the functor $D^{\square}_{\rho_x}:\mathfrak A_\Lambda\to\mathrm{Sets}$ such that  $D^{\square}_{\rho_x}(A)$ is the set of continuous group homomorphisms $\rho: G_F\rightarrow 
\GL_d(A)$, such that $\rho \pmod{\mm_A}=\rho_x$. 
In cases (1) and (2) such functors occur in work of Mazur and Kisin, respectively. The formulation in case (3) appears to be new. 
In all cases, the functor $D^{\square}_{\rho_x}$ is pro-represented by a complete local Noetherian $\Lambda$-algebra $R^{\square}_{\rho_x}$ 
with residue field $\kappa(x)$. The arguments of Mazur and Kisin 
carry over to the case when $\kappa(x)$ is a local field of characteristic 
$p$ and yield a presentation
\begin{equation}\label{intro_present}
R^{\square}_{\rho_x}\cong \Lambda\br{x_1,\ldots, x_r}/(f_1, \ldots, f_s) 
\end{equation}
with $r = \dim_{\kappa(x)} Z^1(G_F, \ad \rho_x)$ and $s = \dim_{\kappa(x)} H^2(G_F, \ad \rho_x)$; here $\ad \rho_x$ is the adjoint representation of $G_F$ on $\End_{\kappa(x)}(\rho_x)$ by conjugation. By a suitable version of Tate local duality results, one finds $r-s=d^2+d^2[F:\Qp]$. From this, Theorem~\ref{main-intro} and some commutative algebra results that relate the completion of $A^\gen$ at $x$ to the ring $R^{\square}_{\rho_x}$, we deduce the following result.
\begin{cor}[{Corollaries~\ref{ci} and~\ref{cor_RforLocalK}}]\label{intro-ci}
For $x$ as above the following hold:
\begin{enumerate}
\item $R^{\square}_{\rho_x}$ is a flat
$\Lambda$-algebra of relative  dimension $d^2+ d^2[F:\Qp]$ and is complete intersection;
\item if $\cha(\kappa(x))=p$ then $R^{\square}_{\rho_x}/\varpi$ is complete intersection of dimension $d^2+ d^2[F:\Qp]$.
\end{enumerate}
\end{cor}
At first glance one might expect that for closed points $x$ of $X^\gen$ the residue field $\kappa(x)$ is always finite. However, as we show in Example~\ref{rem:concrete-ex}, $\kappa(x)$ can also 
be a local field of characteristic $0$ or $p$. In Subsection \ref{sec_completions} we show that this exhausts all possibilities. 

Corollary \ref{intro-ci} gives us a handle on the completions of the local rings $\OO_{X^{\gen},x}$ 
(resp.~$\OO_{\Xbar^{\gen}, x}$) at closed points $x\in X^{\gen}$ (resp.~$x\in \Xbar^{\gen}$), which allows us to deduce the following result.
\begin{cor}[{Corollaries~\ref{ci_gen} and~\ref{char0_lift}}] The following hold:
\begin{enumerate}
\item $A^{\gen}$ is $\OO$-torsion free, equi-dimensional of dimension $1+ d^2+ d^2[F:\Qp]$ and is locally complete intersection;
\item $A^{\gen}/\varpi$ is  equi-dimensional of dimension $d^2+d^2[F:\Qp]$ and is locally complete intersection. 
\end{enumerate}
\end{cor}

We end Section~\ref{sec_Rsquare} with a result on the density of (certain) absolutely irreducible points in $R^\square_{\rhobar}$ and in $R^\square_{\rhobar}/\varpi$. This is motivated by and relies on similar results for $R^\ps$. A point $x$ in $X^\ps=\Spec R^\ps$ is called \textit{absolutely irreducible} if the associated semisimple representation $\rho_x:G_F\to \GL_d(\overline{\kappa(x)})$ (which is unique up to isomorphism) is irreducible. It follows from the main theorem of \cite{che_unpublished} that the locus of absolutely irreducible points is dense open in the generic fibre $X^\ps[1/p]=\Spec R^\ps[1/p]$, and this is extremely useful because such points are regular on $X^\ps[1/p]$. 

A key role in the study of the 
regular locus in the special fibre $\Xbar^\ps=\Spec R^\ps/\varpi$
in \cite{bj} is played by a class 
of absolutely irreducible points, 
which are called \textit{non-special}.
We  extend this notion slightly in Appendix \ref{appendix}. We say that an absolutely irreducible point $x$ in $\Xbar^\ps$ with finite or local residue field is \textit{Kummer-reducible} if there exists a degree $p$ Galois extension $F'$ of $F(\zeta_p)$ such that $\rho_x|_{G_{F'}}$ is reducible, and \textit{Kummer-irreducible} if not. If $\zeta_p \in F$ then $x \in \Xbar^{\ps}$ is Kummer-irreducible if and only if it is \textit{non-special} in the sense of \cite[Section 5]{bj}. We show that if $x$ is Kummer-irreducible then 
$H^2(G_F,\ad^0 \rho_x) = 0$ where $\ad^0\rho_x$ is the subrepresentation of $\ad \rho_x$ of trace zero matrices. 
Much more importantly for us, we also show that the locus of Kummer-irreducible $x\in \Xbar^\ps$ is dense open. At these points $\Xbar^\ps$ is not necessarily smooth but it is relatively smooth over $\Spec R_{\det\rhobar}$. Here we prove the following:
\begin{prop}[{Proposition~\ref{nspcl_formal} and Corollaries \ref{irr_gen_dense} and~\ref{cor_irrlift}}]
We have:
\begin{enumerate}
\item The set of absolutely irreducible points $x\in 
\Spec R^{\square}_{\rhobar}[1/p]$ with $\kappa(x)$ finite over $L$ is dense in 
$\Spec R^{\square}_{\rhobar}[1/p]$.
\item The set of Kummer-irreducible points $x\in 
\Spec R^{\square}_{\rhobar}/\varpi$ with $\kappa(x)$ a local field is dense in 
$\Spec R^{\square}_{\rhobar}/\varpi$.
\end{enumerate}
In particular, every continuous representation $\rhobar: G_F\rightarrow \GL_d(k)$ has an absolutely irreducible lift to characteristic zero.
\end{prop}

\subsection{Irreducible components}

From here on, we assume that $L$ contains $F$, so that in particular $L$ contains all roots of unity contained in $F$. We now give a more detailed overview of Theorem~\ref{thm_intro-2} on components of $R^\square_{\rhobar}$. The homomorphism $R_{\det\rhobar}\to R^\square_{\rhobar}$ from that theorem is induced by the natural transformation $D^\square_{\rhobar}\to D_{\det\rhobar}$ that to a representation assigns its determinant, and it induces the map \eqref{eq_components} on components. 

Via the Artin map $F^\times\to G_F^\ab$ from local class field theory, the inclusion $\mu\subset F^\times$ and the identification of $R_{\det\rhobar}$ with the completed group ring of 
the pro-$p$ completion of $G_F^\ab$, the ring $R_{\det\rhobar}$ becomes an $\OO[\mu]$-algebra. It is well-known that $R_{\det\rhobar}$ is a power series ring over $\OO[\mu]$ in $[F:\Qp]+1$ formal variables. The components of the \'etale $L$-algebra $\OO[\mu][1/p]=L[\mu]$ are in bijection with the characters $\chi: \mu\to \OO^\times$. Setting $R^{\square,\chi}_{\rhobar}=R^{\square}_{\rhobar}\otimes_{\OO[\mu],\chi}\OO$ we obtain a decomposition $\Spec R^{\square}_{\rhobar}[1/p]=\bigsqcup_\chi \Spec R^{\square,\chi}_{\rhobar}[1/p]$, where $\chi$ ranges over the characters $\mu\to\OO^\times$. 

The main step in the proof of the bijectivity of the map \eqref{eq_components} in Theorem \ref{thm_intro-2} is to show that
the rings   $R^{\square,\chi}_{\rhobar}$ are normal by verifying Serre's criterion for normality. We first present $R^{\square}_{\rhobar}$ over $R_{\det\rhobar}$
(Proposition \ref{Prop-RelCI})
by imitating Kisin's method of presenting global rings
over local rings. Since $R^{\chi}_{\det\rhobar}:=R_{\det \rhobar}\otimes_{\OO[\mu], \chi}\OO$ 
is formally smooth, by applying $\otimes_{\OO[\mu], \chi}\OO$ we obtain a presentation of $R^{\square,\chi}_{\rhobar}$ over $R^{\chi}_{\det\rhobar}$ analogous to the presentation 
\eqref{intro_present}. Since $R^{\square, \chi}_{\rhobar}$ has
the same dimension as $R^{\square}_{\rhobar}$, the
presentation yields that $R^{\square, \chi}_{\rhobar}$ 
is complete intersection of expected dimension, and hence satisfies Serre's
condition (S2). We then show that $X^{\gen, \chi}:=\Spec (A^{\gen}\otimes_{\OO[\mu], \chi}\OO)$ and its special fibre
$\Xbar^{\gen, \chi}$ are regular in codimension $1$ by
showing that the Kummer-irreducible locus 
in $\Xbar^{\gen, \chi}$ (resp.~absolutely irreducible 
locus in $X^{\gen, \chi}[1/p]$) is regular, and its complement 
has codimension at least $2$ if either $F\neq \Qp$, or 
$d>2$ or $\Dbar$ is absolutely irreducible. The case 
$F=\Qp$, $d=2$ and $\Dbar$ reducible requires 
an extra analysis of the reducible locus. Since 
$R^{\square, \chi}_{\rhobar}$ is a completion of
a local ring at 
a closed point of $X^{\gen,\chi}$, we deduce that 
$R^{\square, \chi}_{\rhobar}$ is regular in 
codimension $1$. We thus deduce that 
$R^{\square, \chi}_{\rhobar}$ is normal. 
Since $R^{\square, \chi}_{\rhobar}$ is 
a local ring it is an integral domain. 
A similar argument works for the special fibre. 

Theorem \ref{thm_intro_fix_det} on $R^{\square, \psi}_{\rhobar}$ is proved by reduction to the results on $R^{\square,\chi}_{\rhobar}$ where $\chi:\OO[\mu]\to\OO^\times$ is the restriction of $\psi$ to $\mu$ via the Artin map. To give an idea of the proof, let $\mathcal X : \mathfrak{A}_{\OO}\rightarrow \Sets$ be the functor, which sends $(A, \mm_A)$ to the group  $\mathcal X(A)$ of continuous characters $\theta: G_F\rightarrow 1+\mm_A$ such that $\theta$ restricted to $\mu$ is trivial, and let $\OO(\mathcal X)$ be the complete local Noetherian $\OO$-algebra pro-representing $\mathcal X$. Local class field theory gives an isomorphism 
$\OO(\mathcal X)\cong \OO\br{y_1, \ldots, y_{[F:\Qp]+1}}$. Let  $\varphi_d:\OO(\mathcal X)\to \OO(\mathcal X)$ be the morphism induced by the $d$-power map  $\mathcal X\to \mathcal X$, $\theta\mapsto \theta^d$. Our key technical result is Proposition~\ref{prop_TwistFunctors} which yields a natural isomorphism 
\[  R_{\rhobar}^{\square,\chi}\otimes_{\OO(\mathcal X), \varphi_d}\OO(\mathcal X) \cong   R_{\rhobar}^{\square,\psi} \widehat\otimes_\OO\OO(\mathcal X) \]
that comes from an analogous isomorphism of functors.  It allows us to compare the sets of points $x$ with $\kappa(x)$ finite or local at which $H^2(G_F,\adz\rho_x)$ is non-zero on both sides. We also show that the map $ \Spec R_{\rhobar}^{\square,\chi}\otimes_{\OO(\mathcal X), \varphi_d}\OO(\mathcal X) \to \Spec R_{\rhobar}^{\square,\chi}$ induces a homeomorphism on special fibres, and a finite covering on generic fibres. Then we use topological arguments to obtain the dimension of $R^{\square, \psi}_{\rhobar}$ and bound the codimension of its singular locus from the analogous results on $R^{\square,\chi}_{\rhobar}$.

We also prove analogs of Theorem \ref{thm_intro-2} 
(resp.~Theorem \ref{thm_intro-3}) for spaces $X^{\gen, \chi}$ and $\Xbar^{\gen, \chi}$ for characters $\chi: \mu\rightarrow \OO^{\times}$ (resp. $X^{\gen, \psi}$, $\Xbar^{\gen, \psi}$ for characters $\psi: G_F \rightarrow \OO^{\times}$ lifting $\det \rhobar$), see Corollaries \ref{ci_gen_chi}, \ref{Xgen_chi_normal}, \ref{Agen_chi_domain} (resp.~Corollaries \ref{ci_gen_psi},
\ref{Agen_psi_domain}). We expect that our results will
be useful in the study of geometry of  Emerton--Gee stack and its derived versions.

It is natural to ask whether our results generalize 
to deformations valued in reductive groups other than $\GL_d$. This question will be addressed in
the forthcoming joint work of VP and Julian Quast.

\subsection{Overview by section}
In Section~\ref{sec_GIT} we briefly review GIT. A key result that gets used later on is Lemma \ref{bound}. In Section~\ref{sec_Rsquare} we introduce 
$X^{\gen}$ and its special fibre $\Xbar^{\gen}$. In Subsection~\ref{sec_fibre} we bound the dimensions of the fibres of the 
map $X^{\gen}\rightarrow X^{\ps}$. In Subsection~\ref{sec_dim_sp} we combine this with results of \cite{bj} to 
bound the dimension of $X^{\gen}$ and $\Xbar^{\gen}$.  In Subsection~\ref{sec_completions} we relate the completions of local rings at closed points $x$ of $X^{\gen}$, $\Xbar^{\gen}$
to the deformation theory of Galois representations $\rho_x: G_F\rightarrow 
\GL_d(\kappa(x))$ and prove Theorem \ref{thm_intro-1}. In Section~\ref{bound_mrs} we bound the maximally reducible semi-simple locus in $X^{\gen}$ and $\Xbar^{\gen}$. This computation later
on gets used only in the case $d=2$, $F=\QQ_2$ and $\Dbar$ is reducible. In Subsection~\ref{sec_dens_irr} we prove the Zariski
density of the Kummer-irreducible locus in $\Xbar^{\gen}$ and 
absolutely irreducible locus in $X^{\gen}[1/p]$ and also establish lower 
bounds for the dimension of their complements. These bounds are used to establish normality later on. In Section~\ref{sec_irrcomp} we present $R^{\square}_{\rhobar}$ over $R_{\det \rhobar}$ and prove Theorem \ref{thm_intro-2}. In Section~\ref{fixed_det} we prove Theorem 
\ref{thm_intro_fix_det}. In Section~\ref{sec_dense} 
we prove Theorem \ref{thm_intro-3}. In Appendix \ref{appendix} we introduce the notion of Kummer-irreducible points in
$\Spec R^{\ps}/\varpi$, which slightly generalizes the notion of non-special points defined in  \cite{bj}. This technical improvement is needed 
in Section~\ref{fixed_det} when 
$\zeta_p\not\in F$.

\subsection{Notation}
Let $F$ be a finite extension of $\Qp$ and let $G_F$ be its absolute Galois group. Let $L$ be another finite extension of $\mathbb Q_p$, such that $\Hom_{\Qp\text{-}\mathrm{alg}}(F, L)$ has cardinality $[F:\Qp]$. Let $\OO$ be the ring of integers in $L$, $\varpi$ a uniformiser, and $k$ the residue field.  We will 
denote by $\zeta_p$ a primitive 
$p$-th root of unity in a fixed 
algebraic closure of $F$. 
For a commutative ring $R$ we let 
$P_1R=\{\pp \in \Spec R: \dim R/\pp =1\}$.

We fix a representation $\rhobar:G_F\rightarrow \GL_d(k)$ and assume that all its irreducible subquotients are absolutely irreducible. We note that we may always achieve that after enlarging $k$, since the image of $\rhobar$ is a finite group. Let $\ad \rhobar$ be the adjoint representation of $G_F$ and $\ad^0\rhobar$ the subspace of trace zero endomorphisms, so that $G_F$ acts on $\End_k(\rhobar)$ by conjugation. We will denote the dimension as a $k$-vector space of cohomology groups $H^i(G_F, \ad \rhobar)$ by $h^i$. 
\subsection{Acknowledgements} AI would like to thank Carl Wang-Erickson for a discussion of \cite{WE_alg}. VP would like to thank Daniel Greb for a discussion on Geometric Invariant Theory and Toby Gee for stimulating correspondence. The authors would like to thank Toby Gee, James Newton, Julian Quast and Carl Wang-Erickson 
for their comments on an earlier draft as well as 
Frank Calegari, S{\o}ren Galatius and Akshay Venkatesh for organizing Oberwolfach Arbeitsgemeinschaft \textit{Derived Galois Deformation Rings and Cohomology of Arithmetic Groups} in April 2021,
which served as an impetus for this collaboration.

We thank the anonymous referee for a very thorough and careful reading of this article, and for making a number of helpful suggestions to improve the exposition.

GB acknowledges support by Deutsche Forschungsgemeinschaft  (DFG) through CRC-TR 326 \textit{Geometry and Arithmetic of Uniformized Structures}, project number 444845124.

\section{Geometric invariant theory}\label{sec_GIT}

We first recall the set up of \cite{seshadri}. Let $R$ be a Noetherian ring and let $S=\Spec R$. Let $G$ be a reductive group scheme over $S$, so that $G$ is an affine group scheme over $S$, $G\rightarrow S$ is smooth and the geometric fibres are connected reductive groups. In the application $G= S\times_{\Spec \ZZ} \GL_d$ and $G= S\times_{\Spec \ZZ} \mathbb G_m^r$ so that these conditions hold.
 
Let $V$ be a free $R$-module of finite rank $r$ endowed with a $G$-module structure, let $\check{V}=\Hom_R(V, R)$ and let $\Sym(\check{V})$ be the symmetric algebra. The $G$-module structure on $V$ induces an action of $G$ on $\Spec(\Sym(\check{V})) = \mathbb A^r_S.$ Let $X$ be a closed $G$-invariant subscheme of $\Spec (\Sym(\check{V}))$. The $G$-action on $X$ induces an action on $B$, the ring of functions on $X$. The GIT quotient $X\sslash G$ is represented by the ring of invariants $B^G$.

\begin{lem}\label{IrredCompInvariant}
Every irreducible component $Z$ of $X$, equipped with its reduced subscheme structure, is $G$-invariant.
\end{lem}
\begin{proof} This fact is mentioned (and a proof is sketched) in \cite[Section 4]{seshadri}, but we give a full proof for the convenience of the reader. We have to show that $\varphi( G\times_S Z) \subset Z$, where $\varphi: G\times_S X \rightarrow X$ is the action map.
In terms of rings, this amounts to showing 
that the kernel of 
$\varphi^{\sharp}: 
B\rightarrow \mathcal O(G)\otimes_R B/\pp$ 
is equal to $\pp$, where $\mathcal O(G)$ is the ring of functions of $G$ and $\pp$ is 
a prime of $B$ such that $Z=V(\pp)$. Since the identity element in $G$ maps $Z$ to itself, $\ker \varphi^{\sharp}$
is contained in $\pp$. Since $Z$ is an irreducible component 
of $X$, it is enough to show that $\mathcal O(G)\otimes_R B/\pp$ is an integral domain, as then $\ker \varphi^{\sharp}$ 
is a prime of $B$ and therefore has to equal to $\pp$.

Since $G\rightarrow S$ is geometrically connected and smooth,
$G\times_S \eta$ is integral for every geometric point $\eta$ of $S$. Thus $\mathcal O(G)\otimes_R \overline{\kappa(\pp)}$ 
is an integral domain, where $\overline{\kappa(\pp)}$ is an algebraic closure of the fraction field of $B/\pp$. Since
$G\rightarrow S$ is smooth, it is also flat. Thus 
$\OO(G)\otimes_R B/\pp$ is a subring of $\mathcal O(G)\otimes_R \overline{\kappa(\pp)}$, and hence is an integral domain. 
\end{proof}

Let $y=\Spec \kappa$ be a geometric point of $X\sslash G$. We may identify the fibre $X_y$ with a closed 
$G$-invariant subscheme of $X$.

\begin{lem}\label{bound}
Let $x\in X_y(\kappa)$ be such that the orbit $G \cdot x$ is closed in $X_y$ then $$\dim X_y \le \dim_{\kappa} T_x(X_y).$$
\end{lem} 
\begin{proof}
Let $Z$ be an irreducible component of $X_y$ such that $\dim Z = \dim X_y$. By Lemma \ref{IrredCompInvariant} $Z$ is closed in $X_y$ and $G$-invariant. Then by \cite[Theorem 3]{seshadri} both $Z$ and $X_y$ have a unique closed $G$-orbit, hence those orbits must be equal. Therefore $x \in Z$ so since $Z$ is irreducible,
	\[ \dim X_y = \dim Z \leq \dim_\kappa T_x(Z) \leq \dim_\kappa T_x(X_y).\qedhere \]
\end{proof}

\section{\texorpdfstring{$R^{\square}_{\rhobar}$}{Rsquare} is complete intersection}\label{sec_Rsquare}

Let $\rhobar:G_F\rightarrow \GL_d(k)$
be a continuous representation. Let $D^{\square}_{\rhobar}: \mathfrak A_{\OO}\rightarrow \mathrm{Sets}$ be the functor from 
the category of local Artinian $\OO$-algebras with residue field $k$ to the category of sets, such that for $(A,\mm_A)\in \mathfrak A_{\OO}$, 
$D^{\square}_{\rhobar}(A)$ is the set 
of continuous representations $\rho_A: G_F\rightarrow \GL_d(A)$ such that 
$\rho_A \pmod{\mm_A}=\rhobar$. The 
functor $D^{\square}_{\rhobar}$ is pro-represented by a complete local 
Noetherian $\OO$-algebra $R^{\square}_{\rhobar}$. The main goal of this section is to establish inequalities
\begin{equation}\label{goal_bounds}
\dim R^\square_{\rhobar} \le 1 + d^2 + d^2[F:\Qp], \quad \dim R^\square_{\rhobar}/\varpi\le d^2 + d^2[F:\Qp]. 
\end{equation}

It is well known to the experts that these bounds imply that the rings $R^\square_{\rhobar}$ and $R^\square_{\rhobar}/\varpi$ are complete intersection. 
Namely, the proof of \cite[Proposition 21.1]{Mazur} shows that the tangent space to $D^{\square}_{\rhobar}$ 
is $Z^1(G_F, \ad\rhobar)$ and it follows from the proof of Proposition 2 in
\cite[Section 1.6]{Mazur_GQ}
that there
is a presentation
\begin{equation}\label{presentation}
    R^{\square}_{\rhobar}\cong \OO\br{x_1, \ldots, x_r}/(f_1,\ldots, f_s),
\end{equation}
where $r = \dim_k Z^1(G_F, \ad \rhobar) = d^2-h^0+h^1$ and $s = h^2$ and $h^i=\dim_k H^i(G_F, \ad \rhobar)$. The Euler--Poincar\'e characteristic formula implies that $r - s =  d^2 + d^2[F:\Qp]$. 
Thus $\varpi, f_1, \ldots, f_s$ can be extended to a system of parameters in a regular 
ring $\OO\br{x_1, \ldots, x_r}$ and hence form a regular sequence. 

Let $\overline{D}: k[G_F] \rightarrow k$ be the determinant law attached to $\rhobar$ in the sense of \cite{che_durham}, so that $\overline{D}$ is equal to the composition of the polynomial laws induced by $k[G_F]\overset{\rhobar}{\longrightarrow} M_d(k)$ and $M_d(k) \overset{\det}{\longrightarrow} k$.
In this paper, we will refer to 
determinant laws as pseudo-characters. 
Let $D^{\ps}: \mathfrak A_{\OO}\rightarrow \mathrm{Sets}$ be the functor sending an object $(A,\mm_A)\in \mathfrak A_{\OO}$ to the set $D^{\ps}(A)$ of continuous $A$-valued $d$-dimensional pseudo-characters of $A[G_F]$ which reduce to $\Dbar$ modulo $\mm_A$. 
The functor $D^{\ps}$ is pro-representable by a complete local Noetherian $\OO$-algebra $(R^{\ps}, \mm_{R^{\ps}})$ by \cite[Proposition 3.3]{che_durham}.

Mapping a deformation of $\rhobar$ to its determinant
induces a natural transformation $D^{\square}_{\rhobar}\rightarrow D^{\ps}$ and thus 
a map of local $\OO$-algebras $R^{\ps}\rightarrow R^{\square}_{\rhobar}$. The ring $R^{\ps}$ has been well understood in the recent work of GB--Juschka \cite{bj}, who have determined its dimension and showed that the absolutely irreducible locus is dense in the special fibre.  Our basic idea is to study $R^{\square}_{\rhobar}$ by studying the fibres of this 
map. In fact it is technically more convenient to introduce an intermediate ring $R^{\ps}\rightarrow A^{\gen}\rightarrow R^{\square}_{\rhobar}$ (see the next subsection), depending on $\Dbar$ and not on $\rhobar$ itself, such that 
$A^{\gen}$ is of finite type over $R^{\ps}$ and 
$R^{\square}_{\rhobar}$ is a completion of $A^{\gen}$ 
at a maximal ideal. Since $\dim R^{\square}_{\rhobar}\le \dim A^{\gen}$, it is enough to bound the dimension of  $A^{\gen}$. In fact we first bound the dimension of its special fibre (Theorem \ref{main}).

\subsection{Generic matrices} 
Let $D^u: R^{\ps}\br{G_F}\rightarrow R^{\ps}$ be the universal pseudo-character lifting $\overline{D}$. Let $\CH(D^u)$ be the Cayley--Hamilton ideal, which is a closed two-sided ideal of $R^{\ps}\br{G_F}$ defined in \cite[Section 1.17]{che_durham} in such a way that $$E:= R^{\ps}\br{G_F}/ \CH(D^u)$$ is the largest quotient of $R^{\ps}\br{G_F}$ for which the Cayley--Hamilton theorem for $D^u$ holds. Following \cite[Section 1.17]{che_durham} we will call such algebras \textit{Cayley--Hamilton $R^{\ps}$-algebras of degree $d$}. Then $E$ is a finitely generated $R^{\ps}$-module, \cite[Proposition 3.6]{WE_alg}. If $f: E\rightarrow M_d(B)$ is a homomorphism of $R^{\ps}$-algebras for a commutative $R^{\ps}$-algebra $B$ then we say $f$ is a
\textit{homomorphism of Cayley--Hamilton algebras} 
if $\det\circ f: E\rightarrow B$ is equal 
to the specialization of $D^u$ along $R^{\ps}\rightarrow B$. 

The superscript \textit{gen} in $A^{\gen}$ stands for 
\textit{generic matrices}, and the following construction appears in the work of Procesi \cite{Pro87}; Lemmas \ref{existence_Agen}, \ref{topology}, \ref{closed_pts} are contained in
\cite[Theorem 3.8]{WE_alg}, but one needs to translate from the language of groupoids and stacks used in \textit{op.~cit.}\ to access them.
\begin{lem}\label{existence_Agen} 
There is a finitely generated commutative $R^{\ps}$-algebra  $A^{\gen}$ together 
 with a homomorphism of Cayley--Hamilton $R^{\ps}$-algebras $j: E\rightarrow M_d(A^{\gen})$, satisfying the following \textit{universal property}: if $f: E\rightarrow M_d(B)$ is a map of Cayley--Hamilton $R^{\ps}$-algebras for a commutative $R^{\ps}$-algebra $B$ then there is a unique map $\tilde{f}: A^{\gen}\rightarrow B$ of $R^{\ps}$-algebras such that $f= M_d(\tilde{f})\circ j$. 
\end{lem}
\begin{proof}
By writing down a generic $d\times d$-matrix for  each $R^{\ps}$-generator of $E$ and quotienting out by the relations the generators satisfy in $E$, one obtains a commutative $R^{\ps}$-algebra $C$ and 
a homomorphism of $R^{\ps}$-algebras 
$j: E\rightarrow M_d(C)$. More formally, 
$C$ is a quotient of $R^{\ps}\otimes_{\ZZ} \Sym(W)$, where $W$ is a direct sum of $n$ copies of $\End(\Std)^*$, where $\Std$ is the standard representation of $\GL_d$ over $\ZZ$, $n$ is the size of a generating set of $E$ as an $R^{\ps}$-module and $\Sym(W)$ is the symmetric algebra over $\ZZ$. If we were to only require the maps to be $R^{\ps}$-algebras homomorphisms (i.e.\ if we did not impose the Cayley--Hamilton condition) then the map
$j: E\rightarrow M_d(C)$ would satisfy the required universal property. To ensure that the Cayley--Hamilton 
condition is satisfied we have to consider 
the quotient of $C$ constructed as follows. 
Let $\Lambda_i: E\rightarrow R^{\ps}$, $0\le i\le d$ be the coefficients of the characteristic polynomial of $D^u$; these are homogeneous polynomial laws satisfying 
$D^{u}(t-a)= \sum_{i=0}^n (-1)^i \Lambda_i(a) t^{d-i}$ in $R^{\ps}[t]$ as explained in \cite[Section 1.10]{che_durham}. For each 
$a\in E$ let $c_i(j(a))$ be the $i$-th coefficient of the characteristic polynomial of the matrix $j(a)\in M_d(C)$. Let $I$ be the ideal of $C$ generated by $ \Lambda_i(a)-c_i(j(a))$ for all $a\in E$ and $0\le i\le d$  and 
let $A^{\gen}:= C/I$. Since \cite[Corollary 1.14]{che_durham} and \cite[1.1.9.15]{WE_thesis} imply that the coefficients of the characteristic polynomial determine
pseudo-characters uniquely, the composition $E\rightarrow M_d(C)\rightarrow M_d(A^{\gen})$ 
is a map of Cayley--Hamilton algebras and the universal property of $j: E\rightarrow M_d(C)$ implies the universal property for 
$j:E\rightarrow M_d(A^{\gen})$.
Since $E$ is finitely generated as $R^{\ps}$-module, $C$ and hence $A^{\gen}$ are of finite type over $R^{\ps}$. 
\end{proof}

Let us make a connection to GIT as described in Section \ref{sec_GIT}.
If $E$ is generated by $n$ generators as an $R^{\ps}$-module then as explained in the proof of 
Lemma \ref{existence_Agen},
$A^{\gen}$ is a quotient of $R^{\ps}\otimes_{\ZZ} \Sym(W)$.
The group $G:=\GL_d$ acts on $W$ by conjugation and this induces an action of $\GL_d$ on $X^{\gen}:= \Spec A^{\gen}$. For every $R^{\ps}$-algebra $B$, a point in $X^{\gen}(B)$ corresponds to an $n$-tuple of $d\times d$-matrices with entries in $B$ satisfying certain relations, and $\GL_d(B)$ acts on $X^{\gen}(B)$ by conjugating the matrices. The scheme $X^{\gen}$ is isomorphic to $\mathrm{Rep}^\square_{\Dbar} = \mathrm{\Rep}^\square_{E,D^u}$ as defined in \cite[Theorem 3.8]{WE_alg}.

The GIT quotient $X^{\gen}\sslash G$ is represented by the ring of invariants $(A^{\gen})^{G}$. The map $R^{\ps}\rightarrow A^{\gen}$ is $G$-invariant and induces a homomorphism $R^{\ps}\rightarrow (A^{\gen})^{G}$.
It follows from  \cite[Theorem 2.20]{WE_alg} that the induced map
\begin{equation}\label{adequate0}
     X^{\gen}\sslash G \rightarrow X^{\ps}:= \Spec R^{\ps} 
\end{equation}
is an adequate homeomorphism in the sense of \cite[Definition 3.3.1]{alper}, i.e. an integral, universal homeomorphism which is a local isomorphism around points with characteristic zero residue field.
We denote by $\overline{X}^{\gen}$ and $\overline{X}^{\ps}$ the special fibres of $X^{\gen}$ and $X^{\ps}$, respectively.
The same argument shows that
    \[ \overline{X}^{\gen}\sslash G \rightarrow \overline{X}^{\ps} \]
is an adequate homeomorphism. 

We equip $R^{\ps} $ with the $\mm_{R^{\ps}}$-adic topology. Since the ring is Noetherian and the residue field is finite, $R^{\ps}$
is a compact ring with respect to this topology. 
\begin{lem}\label{topology} Let $B$ be a topological $R^{\ps}$-algebra. If $f: E\rightarrow M_d(B)$ 
is  any (not a priori continuous) homomorphism of $R^{\ps}$-algebras then 
the composition $G_F \rightarrow E^{\times}\overset{f}{\longrightarrow} \GL_d(B)$ defines a continuous 
representation of $G_F$. 
\end{lem}
\begin{proof} Since $R^{\ps}$ is a compact ring 
\cite[Corollary 1.10]{unique_top} implies that for every finitely generated $R^{\ps}$-module $M$
there is a unique Hausdorff topology on $M$ making $M$ 
into a topological $R^{\ps}$-module.

We equip $R^{\ps}\br{G_F}$ with its projective 
limit topology, $E$ with the quotient topology, and its group of units $E^{\times}$ with the subspace topology via the embedding $E^{\times}\hookrightarrow E\times E$, $x\mapsto (x, x^{-1})$. Since the map 
$G_F \rightarrow R^{\ps}\br{G_F}$ is continuous, the map $G_F\rightarrow E^{\times}$ is also continuous. Moreover, 
since $\CH(D^u)$ is a closed ideal, the topology on $E$ is 
Hausdorff.

Since $E$ is a finitely generated $R^{\ps}$-module, 
its topology coincides with $\mm_{R^{\ps}}$-adic 
topology (this is also proved in \cite[Proposition 3.6]{WE_alg}). Let $M:=f(E)\subset M_d(B)$, let $\tau_1$ 
be the subspace topology on $M$, and let $\tau_2$ be 
be the unique Hausdorff topology on $M$ such that
the action of $R^{\ps}$ is continuous. We claim that the 
identity map $(M, \tau_2)\rightarrow (M, \tau_1)$ is continuous. 
We will now prove the claim. Since $B$ is a topological 
$R^{\ps}$-algebra, the action of $R^{\ps}$ on $M_d(B)$, and 
hence on $M$, is continuous with respect to $\tau_1$. 
Since $M$ is a finitely generated 
$R^{\ps}$-module, we may pick a continuous surjection 
$\varphi: (R^{\ps})^n \twoheadrightarrow M$ for some $n\ge 1$.  Since $R^{\ps}$ is Noetherian the kernel of $\varphi$ is finitely generated and hence a closed submodule of $(R^{\ps})^n$. 
Thus the quotient topology on $M$ induced via $\varphi$ is 
Hausdorff and therefore must coincide with $\tau_2$, which proves the claim.

The same argument shows that $\tau_2$ coincides with the quotient topology via $E\twoheadrightarrow M$ and  the claim implies that 
the map  $f: E\rightarrow M_d(B)$ is continuous and hence
induces a continuous group homomorphism $E^{\times}\rightarrow M_d(B)^{\times}=\GL_d(B)$. 
\end{proof}

\begin{lem}\label{group_alg} The composition  $R^{\ps}[G_F]\rightarrow R^{\ps}\br{G_F}\twoheadrightarrow E$
is surjective. 
\end{lem}
\begin{proof} Since $R^{\ps}[G_F]$ 
is dense in $R^{\ps}\br{G_F}$ its 
image will be dense in $E$ for the 
topologies introduced in the proof of 
Lemma \ref{topology}. The image is also closed, as it is an $R^{\ps}$-submodule of $E$. Hence the map is surjective. 
\end{proof} 

The representation $\rhobar: G_F \rightarrow \GL_d(k)$ induces a map of $R^{\ps}$-algebras $E\rightarrow M_d(k)$ and thus a homomorphism of $R^{\ps}$-algebras $A^{\gen}\rightarrow k$. It follows from the universal property of $A^{\gen}$ that $R^{\square}_{\rhobar}$ is isomorphic to 
the completion of $A^{\gen}$ with respect to the kernel of this map; see Proposition \ref{main_complete} for a more precise statement. Conversely, 
we have the following Lemma. 
\begin{lem}\label{closed_pts} Let $x\in X^{\gen}$ be a closed point above 
the unique  closed point of $X^{\ps}$ and let $\rho_x: G_F\rightarrow \GL_d(\kappa(x))$ be the representation 
obtained by composing $$G_F\rightarrow R^{\ps}\br{G_F}\rightarrow E\overset{j}{\longrightarrow} 
M_d(A^{\gen})\rightarrow M_d(\kappa(x)).$$
Then the pseudo-character associated to $\rho_x$ is equal 
to $\Dbar\otimes_k \kappa(x)$. In particular, $\rho_x$ and $\rhobar\otimes_k \kappa(x)$ have the same semi-simplification. 
\end{lem} 
\begin{proof} Since $D^u\otimes_{R^{\ps}} k = \Dbar$
the first part follows immediately from the definition of $A^{\gen}$. 
The second part follows from 
\cite[Theorem 2.12]{che_durham}. 
Note that since we have assumed 
that all irreducible subquotients of $\rhobar$ are absolutely irreducible, it is enough to 
prove that $\rho_x$ and 
$\rhobar$ have the same semi-simplification after extending  scalars to the algebraic closure of $k$.
\end{proof} 
\begin{remar} We note that one needs to impose the Cayley--Hamilton condition in the definition of $A^{\gen}$ for Lemma \ref{closed_pts} to hold. For example, if $\Dbar=\chi_1+\chi_2$, where 
$\chi_1, \chi_2: G_F\rightarrow k^{\times}$ are distinct characters, then $E\otimes_{R^{\ps}} k\cong k\times k$ by Equation (8) in the proof of \cite[Lemma 1.4.3]{bel_che}, let $\pi_1: E\rightarrow k$ be 
the map obtained by projecting to the first component. Then the map $E\rightarrow M_2(k)$, $a\mapsto \diag( \pi_1(a), \pi_1(a))$ is a map of $R^{\ps}$-algebras, and hence induces a map of $R^{\ps}$-algebras $x: C\rightarrow k$, 
where $C$ is the algebra introduced in the 
proof of Lemma \ref{existence_Agen}.
The representation $\rho_x$ obtained 
by specializing $j: E\rightarrow M_2(C)$ 
at $x$ is isomorphic to $\chi_1+\chi_1$; hence
$\rho_x$ is not equal to $\chi_1+\chi_2$. 
\end{remar}

\subsection{Bounding the dimension of the fibres}\label{sec_fibre}
Let $\pp$ be a prime ideal of $R^{\ps}$ such that $\dim R^{\ps}/\pp\le 1$. 
Its residue field $\kappa(\pp)$ is either $k$ or a local field by Lemma \ref{P1R} below. 
Let $\kappa$ be an algebraic closure $\kappa(\pp)$ equipped with its natural topology and let $y: R^{\ps} \to \kappa$ denote the corresponding homomorphism. The goal of this subsection (Proposition \ref{bound_dim_fibre}) is to bound the dimension of the fibre
    \[ X^{\gen}_y:= X^{\gen}\times_{X^{\ps},y } \Spec \kappa. \]
 Let $D_y$ be the specialization of the universal pseudo-character along $y: R^{\ps}\rightarrow \kappa$ and let 
 \begin{equation}\label{def_E_y} 
E_y := E\otimes_{R^{\ps},y} \kappa \cong  (R^{\ps}\br{G_F}\otimes_{R^{\ps}, y} \kappa)/ \CH(D_y),
\end{equation}
where the last isomorphism follows from  \cite[Section 1.22]{che_durham} or  \cite[Lemma 1.1.8.6]{WE_thesis}. 

Since $E$ is a finitely generated $R^{\ps}$-module, $E_y$ is a finite dimensional $\kappa$-algebra. It follows from the proof of Lemma \ref{topology} that the natural 
map $G_F \rightarrow E_y^{\times}$ is continuous for the topology 
on $E_y$ induced by the topology on $\kappa$. Thus if $W$ is 
an $E_y$-module on a finite dimensional $\kappa$-vector space 
then the induced $G_F$-action on $W$ is continuous. 

Since $\kappa$ is algebraically closed we may write\footnote{We follow the convention of \cite{che_durham}, so that a pseudo-character of a direct sum of representations is a product of their pseudo-characters; the papers \cite{bj} and \cite{WE_alg}  refer to a direct sum instead.}
    \[ D_y= \prod_{i=1}^r D_i, \]
where each $D_i$ is an irreducible pseudo-character\footnote{We refer the reader to \cite[Section 4.1]{bj} for the fundamentals of pseudo-characters.} of dimension $d_i$. We define an equivalence relation on the set $\{ D_i: 1\le i \le r\}$ by $D_i \sim D_j$ if $D_i=D_j(m)$ for some $m\in \ZZ$. Let $k$ be the number of the equivalence classes and let $n_i$ be the number of elements in the $i$-th equivalence class.

Moreover, for $1\le i\le r$ we fix representations 
$\rho_i: G_F \rightarrow \GL_{d_i}(\kappa)$ such that 
$D_i$ is the pseudo-character associated to $\rho_i$.
These representations are uniquely determined up to an 
isomorphism by \cite[Theorem 2.12]{che_durham}, but by $\rho_i$ we 
really mean a group homomorphism into $\GL_{d_i}(\kappa)$ and not the equivalence 
class. 

If $V$ is a continuous representation of  $G_F$ on a finite dimensional $\kappa$-vector space such that its semi-simplification is isomorphic to $\oplus_{i=1}^r \rho_i$ then the pseudo-character associated to $V$ is equal to $D_y$ and the action 
of $G_F$ on $V$ extends to an action of $R^{\ps}\br{G_F}$ and then to an action of $R^{\ps}\br{G_F}\otimes_{R^{\ps}, y} \kappa$, which factors through the 
Cayley--Hamilton quotient. It follows from \eqref{def_E_y} that $V$ and any $G_F$-invariant subquotient  of $V$ is an $E_y$-module. 
In particular, we may apply this to $V= \oplus_{i=1}^r \rho_i$ to deduce that each $\rho_i$ is an $E_y$-module.

\begin{lem}\label{dim_ext} If $i\neq j$ then\footnote{We do not assume that $V$ is multiplicity free so that $\rho_i$ and $\rho_j$ might be isomorphic as $G_F$-representations even if $i\neq j$. Note
however that the statement of the lemma might not hold if $i=j$. For example, if $V$ is irreducible then $E_y=M_d(\kappa)$ is a semi-simple algebra, but $\Ext^1_{G_F}(\rho_1, \rho_1)$ is non-zero.
The proof uses that the pseudo-character associated to $\rho_i\oplus\rho_j$ is a factor of $D_y$.  }
$$\Hom_{E_y}(\rho_i, \rho_j)= \Hom_{G_F}(\rho_i, \rho_j) \quad \text{and} \quad \Ext^1_{E_y}(\rho_i, \rho_{j})= \Ext^1_{G_F}(\rho_i, \rho_{j}),$$
where $\Ext^1_{G_F}(\rho_i, \rho_j)$ is computed in the category of 
continuous representations of $G_F$ on finite dimensional $\kappa$-vector spaces. 
\end{lem}
\begin{proof} It follows from Lemma \ref{group_alg} that the natural map 
$\kappa[G_F]\rightarrow E_y$ is surjective. This implies the assertion 
about $\Hom$ spaces and gives an inclusion $\Ext^1_{E_y}(\rho_i, \rho_{j})\subset \Ext^1_{G_F}(\rho_i, \rho_{j})$. To prove the reverse 
inclusion consider an extension $0\rightarrow \rho_{j} \rightarrow W \rightarrow \rho_i \rightarrow 0$ of $G_F$-representations and  let 
$V= W\oplus \bigoplus_{l\neq i, {j}} \rho_l$. As explained above,
the $G_F$-action on $V$ will factor through the action of $E_y$. Hence, 
$W$ is a representation of $E_y$, which implies that $\Ext^1_{E_y}(\rho_i, \rho_{j})= \Ext^1_{G_F}(\rho_i, \rho_{j})$. 
\end{proof}

Since \eqref{adequate0} is an adequate homeomorphism there is a unique point $y'\in X^{\gen}\sslash G$ above $y$ and $X^{\gen}_{y'} \to X^{\gen}_y$ is a homeomorphism.
The group $G$ acts on $X^{\gen}_y$. Moreover, $X^{\gen}_y$ is of finite type over $\kappa$ and $X^{\gen}_y(\kappa)$ is in bijection with the set of continuous representations $\rho: G_F \rightarrow \GL_d(\kappa)$ such that the semi-simplification of $\rho$ is isomorphic to $\rho_1\oplus \ldots\oplus \rho_r$.

\begin{lem}\label{closed_orbit}
The fibre $X^{\gen}_y$ is connected and the unique closed $G$-orbit in  $X^{\gen}_y$ corresponds to the semi-simple representations. 
If the $\rho_i$ are pairwise non-isomorphic, then its dimension is equal to $d^2-r$. 
\end{lem} 
\begin{proof}
It follows from 
\cite[Theorem 3]{seshadri} that  $X^{\gen}_{y'}$
(and hence $X^{\gen}_y$, by the remark in the paragraph above) contains a unique closed $G$-orbit. Thus it is enough to show that the closure of every $G$-orbit will contain a semi-simple representation. If $x\in X^{\gen}_y(\kappa)$ then after conjugation we may assume that $x$ corresponds to a representation $\rho: G_F \rightarrow \GL_d(\kappa)$ such that the image of $\rho$ is block-upper-triangular, and the blocks on the diagonal are given by $\diag(\rho_{\sigma(1)}(g), \ldots, \rho_{\sigma(r)}(g))$ for some permutation $\sigma\in S_r$. By extending scalars to $\kappa[T]$, conjugating $\rho$ by $\diag(T^{r-1} \id_{d_{\sigma(1)}}, T^{r-2} \id_{d_{\sigma(2)}}, \dots, \id_{d_{\sigma(r)}})$ and specializing at $T=0$ we see that the closure of the $G$-orbit will contain a semi-simple representation. The action of $G$ on $X^{\gen}_y$ leaves the connected components invariant by Lemma \ref{IrredCompInvariant}. Hence, every connected component of $X^{\gen}_y$ will contain the closed point corresponding to the representation $g \mapsto \diag(\rho_{1}(g), \ldots, \rho_{r}(g))$. Thus $X^{\gen}_y$ is connected. 

The stabilizer of a semi-simple representation with distinct irreducible factors in $\GL_d$ is isomorphic to $\mathbb{G}_m^r$: a copy of  $\mathbb G_m$ is embedded as scalar matrices inside of each block. Hence, the dimension of the closed $G$-orbit is given by $\dim \GL_d - \dim \mathbb{G}_m^r= d^2-r$.
\end{proof}

In order to analyze $X^{\gen}_y$ we introduce the following notation. We fix a permutation $\sigma\in S_r$ and write $P$ for the block-upper-triangular parabolic subgroup of $\GL_d$ with the $i$-th diagonal block of size $d_{\sigma(i)}\times d_{\sigma(i)}$. We write $N$ for its unipotent radical and $L$ for its Levi subgroup consisting of block diagonal matrices. We let $Z_L \cong \mathbb{G}_m^r$ denote the centre of $L$. Finally, we denote their Lie algebras by $\mathfrak p$, $\mathfrak n$, $\mathfrak l$ and $\mathfrak z_L$, respectively, and write $\gl$ for the Lie algebra of $\GL_d$. We have 
\begin{equation}\label{formula1}
    \dim \gl =d^2, \quad \dim \lv= \sum_{i=1}^r d_{i}^2, 
    \quad \dim \zl= r, 
\end{equation}
\begin{equation}\label{formula2}
    \dim \nn=\frac{1}{2} (\dim \gl -\dim \lv)= \sum_{1\le i< j \le r} d_i d_j.
\end{equation}

\begin{remar}\label{nopara}
We note that although $\pp$, $\nn$, $\lv$ and $\zl$ depend on $\sigma$, their dimensions do not. 
\end{remar} 

Let $\rho_{\sigma}: G_F\rightarrow \GL_d(\kappa)$ be the representation $g\mapsto  \mathrm{diag}(\rho_{\sigma(1)}(g), \ldots,  \rho_{\sigma(r)}(g))$. 
It follows from a calculation with block-upper-triangular matrices that $\pp$ can be 
given an associative $\kappa$-algebra structure such that the
the inclusion $\pp\subset \gl=M_d(\kappa)$ is an inclusion of Cayley--Hamilton algebras. 

\begin{lem}\label{rhosig_univ}
There exists a closed subscheme $X^{\gen}_{y,\sigma} \subset X^{\gen}_y$
representing the functor sending a $\kappa$-algebra $B$ to the set of homomorphisms of Cayley--Hamilton $\kappa$-algebra  $\varphi: E_y \to \pp \otimes_\kappa B$ such that the projection onto the $i$th diagonal block is $\rho_{\sigma(i)}\otimes_{\kappa} B$ for $1 \leq i \leq r$.
\end{lem}
\begin{proof}
The universal map $j: E\rightarrow M_d(A^{\gen})$ induces a map
    \[ j_y: E_y \rightarrow M_d(A^{\gen} \otimes_{R^{\ps},y} \kappa). \]
Let $I_{\rho, \sigma}$ be the ideal of $A^{\gen} \otimes_{R^{\ps},y} \kappa$ generated by the matrix entries of $j_y(a)$ for all $a\in E_y$, which lie below the diagonal blocks of $P$, and by all the elements on the block diagonal of the matrices $(j_y(a)- \rho_{\sigma}(a))$ for all  $a\in E_y$. Let
    \[ X^{\gen}_{y,\sigma}:= \Spec ((A^{\gen} \otimes_{R^{\ps},y} \kappa)/ I_{\rho, \sigma}). \]
    Then $X^{\gen}_{y,\sigma}$ is a closed subscheme of $X^{\gen}_y$, and its defining ideal $I_{\rho,\sigma}$ was constructed precisely so that a $B$-point of $X^{\gen}_y$ factors through $X^{\gen}_{y,\sigma}$ if and only if it lands in $\pp \otimes_\kappa B$ and matches the $\rho_i$ on the diagonals for $1 \le i \le r$.
\end{proof}

The adjoint action (i.e.~via conjugation) of $Z_L N$ on $\pp$ induces an action of $Z_L N$ on $X^{\gen}_{y,\sigma}$.

\begin{lem}\label{closed_orbit_sigma}
The unique closed $Z_L$-orbit in $X^{\gen}_{y,\sigma}$ is the singleton $\{\rho_\sigma\}$.
\end{lem}
\begin{proof}
This is the same proof as in Lemma \ref{closed_orbit}, just use the same diagonal matrix trick to kill off the unipotent part.
\end{proof}

\begin{prop}\label{tangent}
Let $x \in X^{\gen}_{y,\sigma}$ be the point corresponding to the representation $\rho_{\sigma}$. Then
\begin{equation}\label{ext2}
    \begin{split}
        \dim T_x(X^{\gen}_{y,\sigma})&=  \dim \nn +(\dim \nn)[F:\Qp] +\sum_{1\le i<j\le r}  
        \dim \Hom_{G_F}( \rho_{\sigma(i)}, \rho_{\sigma(j)}(1))\\
        & \le  \dim \nn +(\dim \nn)[F:\Qp] + \sum_{i=1}^k \begin{pmatrix} n_i\\ 2\end{pmatrix}.
    \end{split}
\end{equation}
\end{prop}
\begin{proof}
Using Lemma \ref{rhosig_univ} and the decomposition $\pp = \mathfrak l \oplus \nn$ we may identify $T_x(X^{\gen}_{y,\sigma})$ with the space of $\kappa$-algebra homomorphisms $\varphi: E_y \rightarrow M_d(\kappa[\varepsilon])$, which can be written as $\varphi= \rho_{\sigma} + \varepsilon \beta$, where $\beta$ is a $\kappa$-linear map $\beta: E_y \rightarrow \nn$. If $\beta: E_y\rightarrow \nn$ is any $\kappa$-linear map then $\varphi:= \rho_{\sigma} + \varepsilon \beta$ is a homomorphism of $\kappa$-algebras if and only 
if
\begin{equation}\label{transf}
    \beta(aa') = \rho_{\sigma}(a) \beta(a')+ \beta(a) \rho_{\sigma}(a'), \quad \forall a, a'\in E_y.
\end{equation}
For $1\le i\le r$ we let $\Eins_i\in M_d(\kappa)$ be the block diagonal matrix with the identity matrix on the $i$-th block and zeros everywhere else. Since $\rho_{\sigma}(g)$ commutes with $\Eins_i$ for all $i$, we have an isomorphism 
    \[ T_x(X^{\gen}_{y,\sigma})\cong \bigoplus_{1\le i< j\le r} V_{ij}, \]
where $V_{ij}$ is the space of functions $\beta: E_y \rightarrow \Eins_i \nn \Eins_j$ satisfying \eqref{transf}. We may identify $\Eins_i \nn \Eins_j$ with $\Hom_{\kappa}(\rho_{\sigma(j)}, \rho_{\sigma(i)})$. Then $V_{ij}$ is precisely the space of $1$-cocycles for the Hochschild cohomology of $E_y$ with values in $\Hom_{\kappa}(\rho_{\sigma(j)}, \rho_{\sigma(i)})$. Thus 
\begin{equation}
    \begin{split}
        \dim_{\kappa} V_{ij} &= \dim_{\kappa} HH^1(E_y, \Hom_{\kappa}(\rho_{\sigma(j)}, \rho_{\sigma(i)})) + \dim_{\kappa} 
        \Hom_{\kappa}(\rho_{\sigma(j)}, \rho_{\sigma(i)})\\&- \dim_{\kappa} HH^0(E_y, \Hom_{\kappa}(\rho_{\sigma(j)}, \rho_{\sigma(i)}))\\&= \dim_{\kappa}\Ext^1_{E_y}(\rho_{\sigma(j)}, \rho_{\sigma(i)}) + d_i d_j - 
        \dim_{\kappa} \Hom_{E_y}(\rho_{\sigma(j)}, \rho_{\sigma(i)})\\
        &= d_i d_j + [F:\Qp] d_i d_j + \dim_{\kappa} \Ext^2_{G_F}(\rho_{\sigma(j)}, \rho_{\sigma(i)}),
    \end{split} 
\end{equation}
where the first equality follows from \cite[Proposition IX.4.4.1]{CE}, the second from \cite[Corollary IX.4.4.4]{CE} and the third from Lemma \ref{dim_ext} together with the local Euler--Poincar\'e characteristic formula
in this context\footnote{The results in \cite[Theorem 3.4.1]{bj} on local Tate duality and a local Euler--Poincar\'e characteristic formula, when the coefficient field $\kappa$ is a local field, are based on the work of Nekov\'{a}\v{r}~\cite{nekovar}.}
(see \cite[Theorem 3.4.1 (c)]{bj}). Thus 
    \[ \dim_{\kappa} T_x(X^{\gen}_{y,\sigma})= \dim \nn + (\dim \nn)[F:\Qp] +\sum_{1\le i<j\le r} \dim_{\kappa} \Ext^2_{G_F}(\rho_{\sigma(j)}, \rho_{\sigma(i)}). \]
It follows from the local duality, see \cite[Theorem 3.4.1 (b)]{bj}, that 
    \[ \dim_{\kappa}\Ext^2_{G_F}(\rho_{\sigma(j)}, \rho_{\sigma(i)}) = \dim_{\kappa}\Hom_{G_F}( \rho_{\sigma(i)}, \rho_{\sigma(j)}(1)). \]
Thus if this term is non-zero then it is equal to $1$ and $\rho_{\sigma(i)}$ and $\rho_{\sigma(j)}$ belong to the same equivalence class. 
\end{proof} 
 
\begin{remar}
If $\cha (\kappa)=p$ and $\zeta_p\in F$ then $D_i\sim D_j$ if and only if $D_i=D_j$ and the bound is sharp in this case. 
\end{remar}

\begin{cor}\label{bound_dim2}
$\dim X^{\gen}_{y,\sigma} \le \dim_{\kappa} T_x(X^{\gen}_{y,\sigma})\le  \dim \nn +(\dim \nn)[F:\Qp] + \sum_{i=1}^k \begin{pmatrix} n_i\\ 2\end{pmatrix}.$
\end{cor} 
\begin{proof}
This follows from Lemma \ref{closed_orbit_sigma} and Lemma \ref{bound} applied with $G = Z_L$ and $X = X^{\gen}_{y,\sigma}$, noting that $X^{\gen}_{y,\sigma} \sslash Z_L$ is a singleton.
\end{proof}

\begin{lem}\label{jacobson_dom}
If $f: X \to Y$ is a finite type and dominant morphism of Noetherian Jacobson universally catenary schemes, then $\dim Y \leq \dim X$.
\end{lem}
\begin{proof}
Passing to reduced subschemes does not affect Krull dimension, so we may assume that $X$ and $Y$ are both reduced.

First assume $X$ and $Y$ are irreducible. Pick dense open affines $U\subset Y$, $V\subset X$ such that $f(V)\subset U$. Since $f$ is dominant \cite[\href{https://stacks.math.columbia.edu/tag/0CC1}{Tag 0CC1}]{stacks-project} implies that $A:=\OO_Y(U)\hookrightarrow B:= \OO_X(V)$ is injective. Since $A$ is an integral domain, Noether normalization \cite[\href{https://stacks.math.columbia.edu/tag/07NA}{Tag 07NA}]{stacks-project} implies that the map factors as
    \[ A\hookrightarrow A[x_1,\ldots, x_m]\hookrightarrow B' \hookrightarrow B, \]
with $B'$ finite over $A[x_1, \ldots, x_m]$ and $B'_g \cong B_g$ for some non-zero $g \in A$. Then \cite[\href{https://stacks.math.columbia.edu/tag/0DRT}{Tag 0DRT}]{stacks-project} and \cite[13.C, Theorem 20]{matsumura_alg} imply that
    \[ \dim X = \dim B = \dim B_g = \dim B'_g = \dim B' = \dim A + m = \dim Y + m \]
so $\dim Y \leq \dim X$.

For the general case we argue as in the proof of \cite[\href{https://stacks.math.columbia.edu/tag/01RM}{Tag 01RM}]{stacks-project}. Write $X = \bigcup_j Z_j$ as the union of its irreducible components. Because $f$ is dominant, we have $Y=\bigcup_j \overline{f(Z_j)}$. Clearly the $\overline{f(Z_j)}$ have to be irreducible, and so the irreducible components of $Y$ have to be among them. The $Z_j$ and $\overline{f(Z_j)}$ are again Noetherian, Jacobson and universally catenary, and hence by the case already treated we have
 \[ \dim Y = \max_j \dim \overline{f(Z_j)} 
 \leq \max_j \dim Z_{j} = \dim X. \qedhere \]
\end{proof}

\begin{prop}\label{bound_dim_fibre} $\dim X^{\gen}_y \le \dim \gl -r + (\dim \nn)[F:\Qp] + \sum_{i=1}^k \begin{pmatrix} n_i \\ 2\end{pmatrix}.$
\end{prop}
\begin{proof} We want to apply Lemma \ref{jacobson_dom} to  
\begin{equation}\label{fibre_bundle}
\coprod_{\sigma\in S_r} G\times^{Z_{L_{\sigma}} N_{\sigma}} X^{\gen}_{y,\sigma}
\rightarrow X_y^{\gen},
\end{equation}
where the actions of $Z_{L_{\sigma}} N_{\sigma}$ on $X^{\gen}_{y,\sigma}$ and of $G$ on $X^{\gen}_{y}$ are given by conjugation.

If $x\in X_y^{\gen}(\kappa)$ and $\varphi: E_y \rightarrow M_d(\kappa)$ is the corresponding 
$\kappa$-algebra homomorphism then there will exist $\sigma \in S_r$ such that $\kappa^d$
will admit a filtration by subspaces   $0=V_0\subset V_1\subset \ldots \subset V_r=V$, which 
is invariant under the action of $E_y$ via $\varphi$, satisfying 
$V_i/ V_{i-1}\cong \rho_{\sigma(i)}$  for $1\le i\le r$. Thus there is $g\in G(\kappa)$ such that 
$g \varphi g^{-1}$ will lie in $X^{\gen}_{y,\sigma}(\kappa)$, and hence \eqref{fibre_bundle} induces a surjection on $\kappa$-points. But \eqref{fibre_bundle} is also a map of finite type $\kappa$-schemes, and therefore is a dominant map of Noetherian Jacobson universally catenary schemes, so we can apply Lemma \ref{jacobson_dom}.

The fibre bundles $G\times^{Z_{L_{\sigma}} N_{\sigma}} X^{\gen}_{y,\sigma}$ have dimension equal to
    \[ \dim G + \dim X^{\gen}_{y,\sigma} -\dim (Z_{L_{\sigma}} N_{\sigma}) =\dim \gl -r +\dim X^{\gen}_{y,\sigma} -\dim \nn. \]
The bound in Corollary \ref{bound_dim2} gives 
the required assertion.
\end{proof} 

\begin{cor}\label{fibre_irr} If $r=1$ then $X^{\gen}_y$ is smooth of dimension $\dim \gl-1$. 
\end{cor} 
\begin{proof} If $r=1$ then $E_y\cong M_d(\kappa)$ and thus has a unique irreducible 
representation $\rho$ (up to isomorphism). 
Thus all the points in $X^{\gen}_y(\kappa)$ lie in the same $G$-orbit. Fix such a point $x$. Since the $G$-stabiliser of $x$ is equal to $Z_G$ we obtain $\dim X^{\gen}_y= \dim G-\dim Z_G= \dim \gl -1$.

Since $E_y$ is semi-simple we have $\Ext^1_{E_y}(\rho, \rho)=0$ and thus an argument as in
the proof of Proposition \ref{tangent} gives us $$\dim_{\kappa} T_x(X^{\gen}_y)= \dim_{\kappa} \End_{\kappa}(\rho)-\dim_{\kappa} \End_{E_y}(\rho)=\dim X_y^{\gen}.$$ 
Thus  $x$ is a smooth point of $X^{\gen}_y$, and since
$G$ acts transitively on $X^{\gen}_y(\kappa)$ all the points in $X^{\gen}_y(\kappa)$ are smooth. Since $X_y^{\gen}$ is
of finite type over $\kappa$, we deduce that $X_y^{\gen}$ is smooth.
\end{proof}
\subsection{Commutative algebra preparations}\label{com_alg_prep}
Lemma \ref{jacobson} is the key result of this
section and it will be applied repeatedly  with $R=R^{\ps}$ and $S=A^{\gen}$ or their reductions modulo $\varpi$.

We will start with some
general commutative algebra lemmas.
For a ring $R$ we set $P_1R=\{\pp\in\Spec R : \dim R/\pp =1\}$.

\begin{lem}\label{P1R} Let $(R, \mm_R)$ be a complete 
local Noetherian $\OO$-algebra with finite residue field $k'$. If $\pp\in P_1R$ then 
$\kappa(\pp)$ is either a finite extension of $L$ or a local field of characteristic $p$. Moreover, $R/\pp$ is contained in the ring 
of integers $\OO_{\kappa(\pp)}$ of $\kappa(\pp)$ and the quotient topology on 
$R/\pp$ induced by the $\mm_R$-adic topology on $R$ coincides with the subspace topology induced by the topology on $\OO_{\kappa(\pp)}$.
\end{lem}
\begin{proof} It follows from Cohen's structure theorem that if $\cha (R/\pp)=0$ then $\OO\subset R/\pp$ and $R/\pp$ is a finitely generated 
$\OO$-module. Thus $\kappa(\pp)$ is a finite 
extension of $L$ and $R/\pp$ is contained in the integral closure of $\OO$ in $\kappa(\pp)$, which is equal to $\OO_{\kappa(\pp)}$. If $\cha(R/\pp)=p$ then $R/\pp$ is finite over 
a subring isomorphic to $k'\br{t}$ and the 
same argument carries over. Moreover, $\OO_{\kappa(\pp)}$ is a finitely generated $R/\pp$-module, and this implies that the topologies coincide. 
\end{proof}

\begin{lem}\label{jacobson} Let $(R, \mm_R)$ be a complete local Noetherian ring and $\varphi: R \rightarrow S$ a ring map of finite type. Let $U$ be a non-empty  open 
subscheme of $U_{\max}:=(\Spec R)\setminus \{\mm_R\}$, let $V$ (resp.~$V_{\max}$) be the  preimage of $U$ (resp.~$U_{\max}$) in $\Spec S$, let $Z$ (resp. $Z_{\max}$) be the closure of $V$ (resp. $V_{\max}$) in $\Spec S$ and let $Y$ be the preimage of $\{\mm_R\}$ in $\Spec S$. Then
\begin{enumerate}
\item $V$ is Jacobson; 
\item the set of closed points of $V$ is $V \cap \set{\text{closed points of } V_{\max}}$;
\item if $x$ is a closed point of $V$ then its image $y$ in $\Spec R$ is a closed point of $U$ and the field extension 
$\kappa(x)/ \kappa(y)$ is  finite; 
\item the set of closed points of $U$ is $U \cap P_1 R$;
\item if every irreducible component of $\Spec S$ meets $Y$ non-trivially then $\dim Z= \dim V +1$;
\item $ \dim V\le  \dim U +\max_{y \in U \cap P_1 R} \dim \varphi^{-1}(\{y\})$.
\end{enumerate}
\end{lem}
\begin{proof} We summarize the situation in the following diagram.
\begin{center}\begin{tikzcd}
	    Z_{(\max)} \ar[dr,hook,"\bullet"] \\
	    V_{(\max)} \dar \ar[u,hook,"\mathrm{cl}"] \ar[r,hook,swap,"\circ"] & \Spec S \dar{\varphi} & Y \ar[l,hook'] \dar \\
	    U_{(\max)} \ar[r,swap,hook,"\circ"] & \Spec R & \{\mathfrak{m}_R\} \lar[hook']
	\end{tikzcd}\end{center}
We will first prove parts (1),  (2) and (3). If $R=S$ and if $U=U_{\max}$ then (1) follows from \cite[\href{https://stacks.math.columbia.edu/tag/02IM}{Tag 02IM}]{stacks-project} and both (2) and (3) hold trivially. If $R=S$ and if $U$ is arbitrary then  $U=V$ 
and (1), (2)  follow from the previous case together with \cite[\href{https://stacks.math.columbia.edu/tag/005W}{Tag 005W}]{stacks-project} and (3) holds trivially. The case of general $\varphi$ now follows from \cite[\href{https://stacks.math.columbia.edu/tag/00GB}{Tag 00GB}]{stacks-project} together with \cite[\href{https://stacks.math.columbia.edu/tag/01P4}{Tag 01P4}]{stacks-project}, because the map $V\to U$ induced from $\varphi$ is of finite type.

Part (4)  follows from (2) applied with $S=R$, using that $\mm_R$ is the unique maximal ideal of $R$, so that the set of closed points of $U_{\max}$ is equal to $P_1R$. 

For (5) note first that since $V$ is open in $\Spec S$ the set of generic points of $V$ is a subset of the set of generic
points of $\Spec S$. Thus $Z$ is 
union of irreducible components of $\Spec S$. 
Let $Z'=\Spec S'$ be an irreducible component 
of $Z$ with the induced reduced subscheme structure so that $S'$ is a domain, let $V'= Z'\cap V$, let $R'$ be the image of 
$R$ in $S'$. The rings $R'$ and $S'$ are excellent and hence universally catenary by \cite[\href{https://stacks.math.columbia.edu/tag/07QW}{Tag 07QW}]{stacks-project}. If $\qq\in \Spec S'$ and $\pp=\qq \cap R'$ then 
\begin{equation}\label{nagata} 
\begin{split}
\dim S'_{\qq}&= \dim R'_{\pp} + \trdeg_{R'} S' - 
\trdeg_{\kappa(\pp)} \kappa(\qq)\\
&= 
\dim R'+ \trdeg_{R'} S' -\dim R'/\pp -\trdeg_{\kappa(\pp)} \kappa(\qq),
\end{split}
\end{equation}
where $\trdeg$ stands for transcendence degree, the
first equality is \cite[\href{https://stacks.math.columbia.edu/tag/02IJ}{Tag 02IJ}]{stacks-project}, and the second is \cite[Theorem 31.4]{matsumura}. It follows from 
\eqref{nagata} that 
\begin{equation}\label{nagata1}
\dim S'_{\qq} \le \dim R' + \trdeg_{R'} S'
\end{equation}
and the equality in \eqref{nagata1} holds if and only if $\qq$ maps to the maximal ideal of $R'$ and $\qq$ is a maximal ideal of $S'$. Since 
$Z'\cap Y$ is non-empty by assumption, such $\qq$
exists and so $$\dim S'= \dim R' + \trdeg_{R'} S'.$$
Let $\qq$ be a closed point of $V'$ and let $\pp=\qq \cap R'$. 
Since $V'$ is open in $Z'$ we have $\OO_{V',\qq}= S'_{\qq'}$. It 
follows from (3) that $\trdeg_{\kappa(\pp)}\kappa(\qq)=0$ and $\dim R/\pp=1$. Thus \eqref{nagata} gives us 
$$ \dim \OO_{V',\qq}= \dim R' + \trdeg_{R'} S'-1.$$
Since this holds for all closed points of $V'$ we
deduce that $$\dim V'= \dim R' + \trdeg_{R'} S'-1.$$ 
This implies part (5).

Let $x$ be a closed point of $V$ and let $y$ be its image in $U$. Then $y$ is also a closed point of $U$. We have 
$$\dim \OO_{V, x} \le \dim \OO_{U, y} + \dim (\OO_{V, x}\otimes_{\OO_{U, y}} \kappa(y))\le \dim U + \dim \varphi^{-1}(\{y\}),$$ 
where the first inequality is given by \cite[Theorem 15.1 (i)]{matsumura}. Since $$\dim V = \max_x \dim \OO_{V, x},$$ where the maximum 
is taken over all closed points $x$ of $V$, we get (6).
\end{proof}

\begin{remar} We caution the reader that the 
equality $\dim Z= \dim V +1$ might fail if one 
drops the assumption that $Y$ meets every irreducible component non-trivially. For example, 
if $R=\Zp$ and $S=\Zp[x]/(px-1)= \Qp$ 
then $Y$ is empty and $V_{\max}=Z_{\max}=\Spec S$. 
\end{remar} 

\begin{remar}\label{caution} Here is another cautionary example. 
If $R$ and $S$ are as in Lemma \ref{jacobson}, 
$\qq$ is a prime of $S$ and $S$ is a domain then 
it need not be true that $\dim S_{\qq}+\dim S/\qq = \dim S$. For example, if $R=\Zp$, $S=\Zp[x]$ and 
$\qq=(px-1)$ then $S/\qq=\Qp$ and $S_{\qq}$ is a DVR, so that $\dim S_{\qq}+\dim S/\qq =1$ and $\dim S=2$. 
We also note that $\qq$ is a closed point of $\Spec S$ but it does not map to a closed point of $\Spec R$. Further, if $\qq'=(p, x)$ then $S/\qq'=\Fp$ and 
$p, x$ is a regular sequence of parameters in $S_{\qq'}$, and 
thus $\dim S_{\qq'}=2$. Thus $\qq$ and $\qq'$ are 
closed points of an irreducible scheme, but their 
local rings have different dimensions. 
\end{remar}

\begin{lem}\label{fix} Let $Y$ be the preimage
of $\{\mm_{R^{\ps}}\}$ in $X^{\gen}$,
let $W$ be a closed non-empty $\GL_d$-invariant subscheme of  $X^{\gen}$  and
let $Z$ 
be an irreducible component of $W$. Then $Y\cap Z$ is non-empty. Moreover, if $x$ is a closed point of $Z$ then the following hold:
\begin{enumerate}
\item if $x\in Y$ then $\dim \OO_{Z, x}= \dim Z$;
\item if $x\not\in Y$ then $\dim \OO_{Z, x}= \dim Z-1$.
\end{enumerate}
\end{lem}
\begin{proof}
By Lemma \ref{IrredCompInvariant}, each irreducible component $Z$ of $W$ is $\GL_d$-invariant. The image of $Z$ in $X^{\ps}$ is closed by Corollary 2 (ii) to \cite[Proposition 9]{seshadri}, and nonempty and so must contain $\mm_{R^{\ps}}$ because $X^{\ps}$ has a unique closed point. Therefore $Y \cap Z$ is nonempty.

The claims about $\dim \OO_{Z,x}$ follows from the proof of part (5) in Lemma \ref{jacobson}.
\end{proof}

\begin{examp}\label{rem:concrete-ex} Let us illustrate Lemma \ref{fix} with a 
concrete example. Let $\Dbar$ be the pseudo-character 
of the $2$-dimensional trivial representation of the 
group $\Gamma:=\Zp$. It follows from 
\cite[Theorem 1.15]{che_durham}
that 
$R^{\ps}\cong\OO\br{t, d}$ and $$E\cong \frac{R^{\ps}\br{T}}{((1+T)^2- (2+t)(1+T) +1+d)},$$
where the map $\Gamma\rightarrow R^{\ps}\br{\Gamma}\twoheadrightarrow E$ sends a fixed topological generator $\gamma$ of $\Gamma$ to $1+T$. Then $E$ is a free $R^{\ps}$-module with basis $1+T, 1$ and so
$$A^{\gen}=\frac{R^{\ps}[x_{11}, x_{12}, x_{21}, x_{22}]}{(x_{11}+x_{22} - (2+t), x_{11}x_{22} - x_{12}x_{21} - (1+d))},$$
and $j: E\rightarrow M_2(A^{\gen})$ sends $1+T$ to the matrix $\bigl( \begin{smallmatrix} x_{11} & x_{12} \\ 
x_{21} & x_{22}\end{smallmatrix}\bigr)$.  
Let $x: A^{\gen}\rightarrow L$ be the homomorphism 
corresponding to the representation $\rho: E\rightarrow M_2(L)$, such that $\rho(\gamma)= \bigl( \begin{smallmatrix} 1 & p^{-1} \\ 
0 & 1 \end{smallmatrix}\bigr)$. Then $x$ is a closed point of $X^{\gen}$ with residue field $L$, thus it 
does not map to the closed point in $X^{\ps}$. Indeed, 
$A^{\gen}/(x_{11}-1, x_{21}, x_{22}-1)\cong \OO[x_{12}]$,
so we are in the situation considered in Remark \ref{caution}.
\end{examp}

\begin{lem}\label{dim_gen_sp} Let $W$ be a closed non-empty $\GL_d$-invariant subscheme of $X^{\gen}$ and write $W[1/p]$ and $\overline{W}$ for the generic and special fibre. 
Then $\dim W[1/p]\le \dim \overline{W}$. In particular, 
$\dim X^{\gen}[1/p] \le \dim \Xbar^{\gen}.$
\end{lem}
\begin{proof} We may assume that $W[1/p]$ is non-empty and using Lemma \ref{IrredCompInvariant} we may further assume that $W$ is irreducible. 
Lemma \ref{fix} implies that 
there is a closed point $x\in W$, which maps to the 
closed point in $X^{\ps}$. Lemma \ref{jacobson} (5) 
implies that $\dim W[1/p]=\dim W -1$. 

Since $W$ is irreducible and $W[1/p]\neq \emptyset$ the local ring $\OO_{W,x}$ is a domain and multiplication by $\varpi$ is injective.  Since 
$\cha(\kappa(x))=p$, $\varpi$ cannot be a unit in 
$\OO_{W,x}$. Thus $\dim \OO_{\overline{W}, x}= 
\dim \OO_{W,x}-1$. It follows from Lemma \ref{fix} that $\dim \overline{W}= \dim W-1$. 
\end{proof}

\subsection{Bounding the dimension of the space}\label{sec_dim_sp} The main result of this subsection is Theorem \ref{main}, which bounds the dimension of $\Xbar^{\gen}$. As explained earlier, this is an intermediate step in bounding the dimension of $R_{\rhobar}^\square$.

 Recall that $\Dbar: G_F \rightarrow k$ is the specialization of the universal pseudo-character $D^{u}: G_F \rightarrow R^{\ps}$ at the maximal 
 ideal of $R^{\ps}$. We may write $\Dbar= \prod_{i=1}^m \Dbar_i$, where $\Dbar_i$ are absolutely irreducible pseudo-characters. Let $\mathcal P$ be an (unordered)  
 partition of the set $\{1, \ldots, m\}$ into $r$ disjoint subsets $\Sigma_j$, and let $\underline{\Sigma}=(\Sigma_1,\ldots, \Sigma_r)$ be an ordering of the subsets in $\mathcal P$. For each $1\le j \le r$ let $\Dbar'_j= \prod_{i\in \Sigma_j} \Dbar_i$, 
 and let $d_j$ be the dimension of $\Dbar'_j$. We define an equivalence relation on the set 
 of pseudo-characters $\{\Dbar'_j: 1\le j\le r\}$ by $\Dbar'_j \sim \Dbar'_{j'}$ if 
$\Dbar'_j=\Dbar'_{j'}(t)$ for some $t\in \ZZ$. Let $k'$ be the number of the equivalence 
classes, $n_i'$ be the number of elements in the $i$-th equivalence class, $c_i$ be the dimension 
of the pseudo-characters in the $i$-th equivalence class. We have 
$$\sum_{i=1}^{k'} n_i' = r, \quad \sum_{i=1}^{k'} c_i n_i' = d.$$
We define 
\begin{equation}\label{formula3}
l_{\PP}:=\sum_{j=1}^r d_j^2=\sum_{i=1}^{k'}n_i' c_i^2, \quad p_{\PP}:= l_{\PP} + n_{\PP}= \sum_{j=1}^r d_j^2 +\sum_{1\le j<j' \le r} d_j d_{j'},
\end{equation}
where
\begin{equation}\label{formula4}
n_{\PP}= \frac{1}{2}(d^2- l_{\PP})= \sum_{1\le j<j' \le r} d_j d_{j'}= \sum_{1\le i < i' \le k'} c_i c_{i'} n_i' n_{i'}' + \sum_{i=1}^{k'} c_i^2 \begin{pmatrix} n_i' \\ 2\end{pmatrix}.
\end{equation} 
The notation is motivated by  \eqref{formula1} and \eqref{formula2}, see also Remark \ref{nopara}.
 
For each $1\le j\le r$ let $R^{\ps}_j$ be the universal deformation ring of $\Dbar'_j$ and let $X^{\ps}_j := R^{\ps}_j$.
The functor $\FF_{\underline{\Sigma}}$, which sends a local  Artinian $\OO$-algebra $(A, \mm_A)$ with residue field $k$ to the set of ordered $r$-tuples 
$(D_1, \ldots, D_r)$ of pseudo-characters with each $D_i$ a deformation of $\Dbar'_i$ to $A$ is represented by the completed tensor product
$$R^{\ps}_{\underline{\Sigma}}:= R^{\ps}_1\wtimes_{\OO} \ldots \wtimes_{\OO}R^{\ps}_r.$$
We let $X^{\ps}_{\underline{\Sigma}}:= \Spec R^{\ps}_{\underline{\Sigma}}$ and denote by $\Xbar^{\ps}_{\underline{\Sigma}}:= \Spec R^{\ps}_{\underline{\Sigma}}/\varpi$
to its special fibre. 
By mapping an $r$-tuple of pseudo-characters to their product we obtain a map 
 $$ \iota_{\underline{\Sigma}}: \Xbar^{\ps}_{\underline{\Sigma}} \rightarrow \Xbar^{\ps}.$$
 \begin{lem}\label{iotaP_finite} The map $R^{\ps}\rightarrow R^{\ps}_{\underline{\Sigma}}$  is finite. 
 \end{lem}
 \begin{proof}  By topological Nakayama's lemma it is enough to show that the fibre ring
 $C:=k\otimes_{R^{\ps}} R^{\ps}_{\underline{\Sigma}}$ is a finite dimensional $k$-vector 
 space. Let $\mathcal F$ be the closed subfunctor of $\mathcal F_{\underline{\Sigma}}$ defined by 
 $C$. If $(A, \mm_A)$ is a local Artinian $k$-algebra then $\mathcal F(A)$ is in bijection 
 with the set of $r$-tuples $(D_{1}, \ldots, D_{r})$, each $D_{i}$ lifting $\Dbar'_i$ to $A$ such 
 that $\prod_{i=1}^r D_i= (\prod_{i=1}^r \Dbar'_i)\otimes_k A$.
 
 Since $C$ is a complete local Noetherian ring, it is enough to show that its Krull dimension is $0$. 
 If this is not the case then there is $\pp\in \Spec C$ such that $\dim C/\pp=1$. Let $(D_{1, y}, \ldots, D_{r,y})$ be the specialization of the universal object of $\mathcal F_{\underline{\Sigma}}$ along 
 $y: R^{\ps}_{\underline{\Sigma}}\rightarrow \kappa(\pp)$. It follows from 
 \cite[Corollary 1.14]{che_durham} that the coefficients of the polynomials 
 $D_{i,y}(t - a)$, for all $a\in E$ and $1\le i\le r$ will generate a dense subring of $R^{\ps}_{\underline{\Sigma}}/\pp$. Since $R^{\ps}_{\underline{\Sigma}}/\pp$ is a complete local $k$-algebra
 of dimension $1$, there will exist $a\in E$ and index $i$  such that the coefficients of $D_{i,y}(t - a)$ will generate 
a transcendental extension of $k$ inside $\kappa(\pp)$. Since $\pp\in \Spec C$ we have 
$$ \prod_{i=1}^r D_{i,y}(t-a)= \prod_{i=1}^r \Dbar'_i(t-a).$$
Thus all the roots of $D_{i,y}(t-a)$ in the algebraic closure of $\kappa(\pp)$ are algebraic over $k$. Since $D_{i,y}(t-a)$ is a monic polynomial, 
we conclude that all the coefficients are also algebraic over $k$, giving a contradiction.
  \end{proof}

Let $\overline{X}^{\ps}_{\mathcal P}$ be the scheme theoretic image of $\iota_{\underline{\Sigma}}$. We note that $\overline{X}^{\ps}_{\mathcal P}$ depends only on $\mathcal P$ and not on the chosen ordering $\underline{\Sigma}$.  
It follows from  Lemma \ref{iotaP_finite} that 
\begin{equation}\label{XPps_bound}
\dim  \overline{X}^{\ps}_{\mathcal P}= \dim \Xbar^{\ps}_{\underline{\Sigma}}
 = \sum_{i=1}^r\dim \Xbar^{\ps}_i = r+ l_{\PP} [F:\Qp],
\end{equation}
 where the last equality is obtained by applying \cite[Theorem 5.4.1(a)]{bj} to each $\Xbar_i^{\ps}$.
  
 We define a partial order on the set of partitions of $\{1, \ldots, m\}$ by $\PP\le \PP'$ if $\PP'$ is a refinement of $\PP$. 
  The partition $\PP_{\min}$ consisting of $1$ part is the minimal element and the partition $\PP_{\max}$ consisting of $m$ 
  parts is the maximal element with respect to this partial ordering. If $ \PP\le \PP'$ then $\Xbar^{\ps}_{\PP'}$ is a closed subscheme 
  of $\Xbar^{\ps}_{\PP}$ and $\Xbar^{\ps}_{\PP_{\min}}= \Xbar^{\ps}$. Let 
  $$U_{\PP}:= \Xbar^{\ps}_{\PP} \setminus ( \{\mm_{R^{\ps}}\} \cup \bigcup_{\PP< \PP'} \Xbar^{\ps}_{\PP'})$$
and let $V_{\PP}$ be the preimage of $U_{\PP}$ in  $\Xbar^{\gen}$ and let $Z_{\PP}$ be the closure of $V_{\PP}$ in $\Xbar^{\gen}$.
Let $\Xbar^{\gen}_{\PP}$ be the preimage of $\Xbar^{\ps}_{\PP}$ in $\Xbar^{\gen}$. Then $\Xbar^{\gen}_{\PP}$ is closed in $\Xbar^{\gen}$ 
and contains $V_{\PP}$, hence we are in the situation of Lemma \ref{jacobson} with $\Spec R = \Xbar^{\ps}_{\PP}$ and $\Spec S = Z_\PP$.
Note that Lemma 
\ref{fix} implies that every irreducible component 
of $\Xbar^{\gen}_{\PP}$ contains a closed point mapping to $\mm_{R^{\ps}}$. Thus 
the condition in part (5) of Lemma \ref{jacobson}
is satisfied and hence $\dim Z_\PP = \dim V_\PP + 1$; the same conclusion applies to closures of various loci considered below.
Moreover, we have 
\begin{equation}\label{XgP}
 \Xbar^{\gen}_{\PP}= Y \cup \bigcup_{\PP\le \PP'} Z_{\PP'},
 \end{equation}
 where $Y$ is the preimage of 
 $\{\mm_{R^{\ps}}\}$ in $\Xbar^{\gen}$.

We will also need a variant of the situation above. Let us assume that $r>1$ and let $i$ and $j$ be distinct indices with $1\le i, j\le r$. 
Let $\mathcal F_{\underline{\Sigma}}^{ij}$ be a subfunctor of $\mathcal F_{\underline{\Sigma}}$ parameterizing the deformations 
$(D_1, \ldots, D_r)$ of the ordered $r$-tuple $(\Dbar'_1, \ldots, \Dbar'_r)$ such that $D_i= D_j(1)$. Then $\mathcal F_{\underline{\Sigma}}^{ij}$
is a closed subfunctor of $\mathcal F_{\underline{\Sigma}}$ and we let $R^{\ps, ij}_{\underline{\Sigma}}$ be the quotient of $R^{\ps}_{\underline{\Sigma}}$
representing it. If $\Dbar'_i\neq \Dbar'_j(1)$ then $R^{\ps, ij}_{\underline{\Sigma}}$ is the zero ring; otherwise it follows from Equation \eqref{XPps_bound} and another application of \cite[Theorem 5.4.1(a)]{bj} that
\begin{align*}
    \dim R^{\ps, ij}_{\underline{\Sigma}}/\varpi = \dim R^{\ps}_{\underline{\Sigma}}/\varpi - \dim R^{\ps}_i/\varpi &\leq r + l_{\PP} [F:\Qp] - ( 1+d_i^2[F:\Qp]) \\
    &\leq r + l_{\PP} [F:\Qp] - ( 1+[F:\Qp]).
\end{align*}
Let $\Xbar^{\ps, ij}_{\PP}$ be the scheme theoretic image of $\Spec  R^{\ps, ij}_{\underline{\Sigma}}$ in $\Xbar^{\ps}$ under $\iota_{\underline{\Sigma}}$. Then 
\begin{equation}\label{bound_Xij}
\dim \Xbar^{\ps,ij}_{\PP}\le r + l_{\PP} [F:\Qp] - ( 1+[F:\Qp]).
\end{equation}
Let $U_{\PP}^{ij}:= U_{\PP} \cap \Xbar^{\ps, ij}_{\PP}$, let $V_{\PP}^{ij}$ be the preimage of  $U_{\PP}^{ij}$ in $\Xbar^{\gen}$ and let 
$Z_{\PP}^{ij}$ be the closure of $V_{\PP}^{ij}$ in $\Xbar^{\gen}$. 

\begin{lem}\label{comb_fibre} If $y$ is a geometric closed point of $U_{\PP}$ then $$\dim X^{\gen}_y\le d^2-r+ n_{\PP} [F:\Qp] + \sum_{i=1}^{k'}  \begin{pmatrix} n_i'\\ 2\end{pmatrix}.$$
If we additionally assume that $y\not\in U_{\PP}^{ij}$ for any $i\neq j$ then 
$$\dim X^{\gen}_y\le d^2-r+ n_{\PP} [F:\Qp].$$
\end{lem} 
\begin{proof} We may write $D_y=D_1+ \ldots+ D_r$ with $D_i$ lifting 
$\Dbar_i'$. We note that all the $D_i$ are absolutely irreducible, since otherwise $y \in X_{\PP'}^{\ps}$ for some $\PP'> \PP$.  
Let $k$ and $n_i$ be the numbers defined in Section~\ref{sec_fibre}. Proposition \ref{bound_dim_fibre} implies that 
$$\dim X^{\gen}_y\le d^2-r+ n_{\PP} [F:\Qp] + \sum_{i=1}^{k}  \begin{pmatrix} n_i\\ 2\end{pmatrix}.$$
If $D_i= D_j(m)$ for some $m \in \ZZ$ then also $\Dbar_i'= \Dbar_j'(m)$. 
This implies that 
$$\sum_{i=1}^k \begin{pmatrix} n_i \\ 2\end{pmatrix} \le \sum_{i=1}^{k'} \begin{pmatrix} n_i' \\ 2\end{pmatrix},$$
which implies the first assertion. We note that if $a_i, \ldots, a_s$ are positive integers then $\sum_{i=1}^s \begin{pmatrix} a_i \\2\end{pmatrix} \le 
\begin{pmatrix} \sum_{i=1}^s a_i \\2\end{pmatrix}$. 

If $y\not \in U_{\PP}^{ij}$ for any $i\neq j$ then $D_i\neq D_j(1)$ for any $i\neq j$ and the $\Hom$ terms in \eqref{ext2} vanish. 
The assertion follows from Proposition \ref{bound_dim_fibre} using this improved bound. 
\end{proof} 

\begin{prop}\label{bound_again} $\dim Z_{\PP}^{ij} \le d^2 + p_{\PP} [F:\Qp]+ \sum_{i=1}^{k'} \begin{pmatrix} n_i'\\ 2\end{pmatrix} - (1 + [F:\Qp])$. 
\end{prop}
\begin{proof} It follows from Lemma \ref{jacobson} (5) that the closure of $U_{\PP}^{ij}$ has dimension $\dim U_{\PP}^{ij} +1$. 
Thus 
$$\dim U_{\PP}^{ij} +1 \le \dim X^{\ps, ij}_{\PP}\le  r+ l_{\PP}[F:\Qp] - (1+[F:\Qp]),$$
where the last inequality is \eqref{bound_Xij}.
Parts (5) and (6) of Lemma \ref{jacobson} together with Lemma \ref{comb_fibre} imply that 
$$\dim Z^{ij}_{\PP}\le  (r+ l_{\PP}[F:\Qp] - (1+[F:\Qp]))+ (d^2-r +n_{\PP} [F:\Qp] + \sum_{i=1}^{k'} \begin{pmatrix} n_i'\\ 2\end{pmatrix}),$$
which imply the assertion.
\end{proof}

 \begin{prop}\label{bound_delta} Let $\delta_{\PP}= \max \{ 0, \sum_{i=1}^{k'} \begin{pmatrix} n_i'\\ 2\end{pmatrix} - (1+[F:\Qp])\}$. Then 
$$\dim Z_{\PP} \le d^2+ p_{\PP} [F:\Qp] +\delta_{\PP}.$$
\end{prop}  
\begin{proof} Let $U'_{\PP}:= U_{\PP} \setminus \bigcup_{i\neq j} U_{\PP}^{ij}$, let $V'_{\PP}$ be the preimage of $U'_{\PP}$ in 
$\Xbar^{\gen}$ and let $Z'_P$ denote the closure of $V'_{\PP}$ in $\Xbar^{\gen}$. If $y$ is a closed point of $U'_{\PP}$ then 
$\dim \Xbar^{\gen}_y \le d^2-r+ n_{\PP} [F:\Qp]$ by Lemma \ref{comb_fibre}. Thus Lemma \ref{jacobson} implies that 
\begin{equation}\label{bound_Zprime}
\dim Z'_{\PP} \le \dim \Xbar^{\ps}_{\PP} + (d^2-r+ n_{\PP} [F:\Qp])= d^2+ p_{\PP} [F:\Qp].
\end{equation}
Since $Z_{\PP}= Z'_{\PP} \cup \bigcup_{i\neq j} Z^{ij}_{\PP}$ we have $\dim Z_{\PP}= \max_{i\neq j}\{ \dim Z'_{\PP}, \dim Z^{ij}_{\PP}\}$ and the 
assertion follows from Proposition \ref{bound_again}. 
\end{proof}
 
 \begin{prop}\label{dim_Pmin} $\dim Z_{\PP_{\min}}\le d^2 + d^2[F:\Qp]$.
 \end{prop} 
 \begin{proof} In this case $r=1$ so $Z_{\PP}=Z'_{\PP}$ and the assertion follows from \eqref{bound_Zprime}.
 \end{proof}
 
 \begin{lem}\label{bound_diff} Assume that  $\PP \neq \PP_{\min}$. If  $d=2$ then 
 $$ d^2+ d^2[F:\Qp] - \dim Z_{\PP} \ge [F:\Qp],$$  
 and $$d^2+d^2[F:\Qp] - \dim Z_{\PP}\ge 1+[F:\Qp],$$
 otherwise. 
 \end{lem} 
 \begin{proof} Proposition \ref{bound_delta} implies that 
 $$ d^2+ d^2[F:\Qp] - \dim Z_{\PP}\ge n_{\PP}[F:\Qp] - \delta_{\PP}.$$
If $d>2$ then $n_{\PP}\ge 2$ and if $d=2$ then $n_{\PP} =1$, which implies the assertion if $\delta_{\PP}=0$. 
Let us assume that $\delta_{\PP}\neq 0$. Then using \eqref{formula4} we may write 
$$n_{\PP}[F:\Qp] -\delta_{\PP}= \sum_{1\le i< j\le k'} c_i  c_j n_i 'n_j' [F:\Qp]  + \sum_{i=1}^{k'} ( c_i^2 [F:\Qp] -1) \begin{pmatrix} n_i'\\ 2\end{pmatrix} + 1 + [F:\Qp],$$
which implies the assertion. 
 \end{proof}
 
 \begin{lem}\label{bound_Y} Let $Y$ be the preimage of $\{\mm_{R^{\ps}}\}$ in $\Xbar^{\gen}$. Then 
 $$\dim Y\le d^2 + n_{\PP_{\max}} [F:\Qp] + n_{\PP_{\max}}-1 .$$
In particular, $d^2+ d^2[F:\Qp]- \dim Y \ge 1+ l_{\PP_{\max}} [F:\Qp]  \ge 1+2 [F:\Qp]$. 
\end{lem}  
\begin{proof}Proposition \ref{bound_dim_fibre} implies that 
$$\dim Y \le d^2 -m +  n_{\PP_{\max}} [F:\Qp]+ \sum_{i=1}^{k'}  \begin{pmatrix} n_i'\\ 2\end{pmatrix}.$$
As already explained in the proof of Lemma \ref{bound_diff} we have $\sum_{i=1}^{k'}  \begin{pmatrix} n_i'\\ 2\end{pmatrix}\le n_{\PP_{\max}}$. 
This implies the assertion. 
\end{proof} 

 \begin{thm}\label{main} $\dim \Xbar^{\gen} \le d^2+ d^2[F:\Qp]$.
 \end{thm}
 \begin{proof} Since $\Xbar^{\ps}= \{\mm_{R^{\ps}}\}\cup \bigcup_{\PP} U_{\PP}$ we have 
 $\Xbar^{\gen}= Y \cup \bigcup_{\PP} Z_{\PP}$. Since these are closed in $\Xbar^{\gen}$ we have 
 $$\dim \Xbar^{\gen} = \max_{\PP} \{ \dim Y, \dim Z_{\PP}\}\le d^2+d^2[F:\Qp],$$ 
 by Proposition \ref{dim_Pmin} and Lemmas \ref{bound_diff} and  \ref{bound_Y}.
 \end{proof} 

Theorem \ref{main} is the main input to Corollary \ref{ci}, which proves Theorem \ref{thm_intro-1}. The missing ingredient is a description of the relationship between $X^{\gen}$ and $R_{\rhobar}^\square$, which is the subject of the next subsection.

\subsection{Completions at maximal ideals and deformation problems}\label{sec_completions}  
Let $Y\subset X^{\gen}$ be the preimage of the closed point of $X^{\ps}$ and let 
$x$ be either a closed point of $Y$ or a closed point of $X^{\gen}\setminus Y$ and let $y$ be its image in $\Spec R^{\ps}$.
 It follows from Lemmas \ref{P1R} and \ref{jacobson} that 
 $\kappa(x)$ is a finite extension of $\kappa(y)$ and 
 there are the following possibilities:
 \begin{enumerate}
 \item if $x\in Y$ then $\kappa(x)$ is a finite extension of $k$;
 \item if $x\in X^{\gen}[1/p]$ then $\kappa(x)$ is a finite extension of $L$;
 \item if $x\in \Xbar^{\gen}\setminus Y$ then $\kappa(x)$ is a local field of characteristic $p$.
 \end{enumerate}

The universal property of $A^{\gen}$ gives 
us a continuous Galois representation $$\rho_x: G_F \rightarrow \GL_d(\kappa(x)).$$ In this section we 
want to relate the completion of the local ring $\OO_{X^{\gen}, x}$ 
to a deformation problem for $\rho_x$.

We will introduce some notation to formulate the deformation 
problem for $\rho_x$. More generally, let 
$\rho: G_F\rightarrow \GL_d(\kappa)$ be a continuous 
representation, where
$\kappa$ is either a finite extension of $k$, a finite extension of $L$ 
or a local field of characteristic $p$ containing $k$ equipped with natural topology.
We first define a ring of coefficients $\Lambda$ over which the deformation problem is defined.
\begin{enumerate}
    \item If $\kappa$ is a finite field then pick an unramified extension $L'$ of $L$ with residue field $\kappa$ and let $\Lambda := \OO_{L'}$ denote the ring of integers in $L'$.
    \item If $\kappa$ is a finite extension of $L$ then let $\Lambda:=\kappa$, let $\Lambda^0$ 
be the ring of integers in $\Lambda$ and let $t=\varpi$.
    \item If $\kappa$ is a local field of characteristic $p$ then let $\OO_{\kappa}$ be the ring of integers 
in $\kappa$ and let $k'$ be its residue field. Since $\cha(\kappa)=p$ by choosing a uniformizer we obtain an isomorphism $\OO_{\kappa}\cong k'\br{t}$. Let $L'$ be an unramified extension of $L$ with residue field $k'$,
let $\Lambda^0:= \OO_{L'} \br{t}$ and let $\Lambda$ 
be the $p$-adic completion of $\Lambda^0[1/t]$. Then 
$\Lambda$ is a complete DVR with uniformiser $\varpi$
and residue field $\kappa$. We equip $\Lambda^0$ 
with its $(\varpi, t)$-adic topology, this induces 
a topology on $\Lambda^0[1/t]$ and 
$\Lambda^0[1/t]/ p^n \Lambda^0[1/t]$ for all $n\ge 1$. We equip $\Lambda= \varprojlim_{n} \Lambda^0[1/t]/ p^n \Lambda^0[1/t]$ with the projective limit topology. 
\end{enumerate}

\begin{remar} In case (3), if $\Lambda'$ is an $\OO$-algebra, which is a complete DVR with uniformiser
$\varpi$ and residue field $\kappa$ then it follows
from \cite[Ch.~IX, \S 2.3, Prop.~4]{Bourbaki_ac8et9} that $\Lambda'$ is non-canonically 
isomorphic to $\Lambda$. 
We will refer to 
$\Lambda'$ (and $\Lambda$) as an $\OO$-Cohen ring of $\kappa$. 
\end{remar}

Let $\mathfrak A_{\Lambda}$ be the category of local
Artinian $\Lambda$-algebras with residue field $\kappa$. Let $(A, \mm_A)\in \mathfrak A_{\Lambda}$.
\begin{enumerate}
    \item In case (1) $A$ is a finite $\OO/\varpi^n$-module for some $n\ge 1$, and 
we just put  the discrete topology on $A$.
    \item In case (2) 
$A$ is a finite dimensional $L$-vector space and 
we put the $p$-adic topology on $A$. 
    \item In case (3) 
$A$ is a $\Lambda^0[1/t]/ \varpi^n \Lambda^0[1/t]$-module of finite length for some $n\ge 1$ and we
put the induced topology on $A$.
\end{enumerate}

Let $D^{\square}_{\rho}(A)$ be the set of 
continuous group homomorphisms $\rho_A: G_F\rightarrow 
\GL_d(A)$, such that $\rho_A \pmod{\mm_A}=\rho$. 

\begin{prop}\label{represent_kappa}
The functor $D^{\square}_{\rho}: \mathfrak A_{\Lambda}
\rightarrow \mathrm{Sets}$ is pro-represented 
by a complete local Noetherian 
$\Lambda$-algebra $R^{\square}_{\rho}$. Moreover, there is a presentation
\begin{equation}\label{present_square_x}
R^{\square}_{\rho}\cong \Lambda\br{x_1,\ldots, x_r}/(f_1, \ldots, f_s) 
\end{equation}
with $r = \dim_{\kappa} Z^1(G_F, \ad \rho)$ and $s = \dim_{\kappa} H^2(G_F, \ad \rho)$.
\end{prop} 
\begin{proof} If $\kappa$ is a finite field then this is a well known consequence of the obstruction theory due to Mazur, \cite[Section 1.6]{Mazur_GQ}. (We revisit the argument in the proof of Proposition \ref{Prop-RelCI}.) If $\kappa$ is a local field then essentially the same argument 
works, except that one has to work harder to justify why the $2$-cocycle constructed out of 
an obstruction to lifting is continuous. 
Lecture 6 in \cite{mod_lift} contains 
 a very nice exposition of the result if $\kappa$ is  a finite extension of $L$. The same argument works 
if $\kappa$ is a local field of characteristic $p$. 
\end{proof} 
If we let $h^i:=\dim_{\kappa} H^i(G_F, \ad \rho)$ then 
\begin{equation}\label{r_minus_s}
r-s= \dim_{\kappa} (\ad \rho) -h^0 + h^1-h^2= d^2+d^2[F:\Qp],
\end{equation}
where the last equality follows from Euler--Poincar\'e characteristic formula, which by \cite[Theorem 3.4.1(c)]{bj} holds in all of the three settings under consideration.

\begin{prop}\label{main_complete} Let $\qq$ be the kernel of the 
map $$\Lambda\otimes_{\OO} A^{\gen} \rightarrow \kappa(x), \quad \lambda \otimes a \mapsto \bar{\lambda}\bar{a},$$
where $\bar{\lambda}$ and $\bar{a}$ denote the images
of $\lambda$ and $a$ in $\kappa(x)$. Then
the completion of $(\Lambda\otimes_{\OO} A^{\gen})_{\qq}$ with respect to the maximal ideal is naturally isomorphic to $R^{\square}_{\rho_x}$.
\end{prop}
\begin{proof} We will prove the proposition, 
when $\kappa(x)$ is a local field of characteristic $p$. The other cases are similar and are left to the 
reader. 

Let $\widehat{B}$ be the completion of  $(\Lambda\otimes_{\OO} A^{\gen})_{\qq}$. It follows
from Lemma \ref{335} below that $\widehat{B}/\varpi \widehat{B}$ (and hence $\widehat{B}$) is Noetherian.
Thus $\widehat{B}/\qq^n \widehat{B} \in \mathfrak A_{\Lambda}$ for all $n\ge 1$. The composition
$$ \Lambda\otimes_{\OO} E\overset{\id\otimes j}{\longrightarrow} \Lambda\otimes_{\OO} M_d(A^{\gen})\rightarrow M_d( \widehat{B}/\qq^n \widehat{B})$$
induces a continuous representation
$G_F \rightarrow
\GL_d(\widehat{B}/\qq^n \widehat{B})$ by Lemma \ref{topology}, which is a 
deformation of $\rho_x$ to $\widehat{B}/\qq^n \widehat{B}$, and hence 
a map of local $\Lambda$-algebras $R^{\square}_{\rho_x}\rightarrow \widehat{B}/\qq^n \widehat{B}$. By passing to the projective limit over $n$ we obtain a continuous representation 
$\hat{\rho}: G_F\rightarrow \GL_d(\widehat{B})$ and a map of local $\Lambda$-algebras $R^{\square}_{\rho_x}\rightarrow \widehat{B}$. 

Let $(A, \mm_A)\in \mathfrak A_{\Lambda}$ and let 
$\rho: G_F\rightarrow \GL_d(A)$ be a continuous 
representation such that $\rho\pmod{\mm_A}=\rho_x$. 
We claim that there is a unique homomorphism of local $\Lambda$-algebras 
 $\varphi: \widehat{B}\rightarrow A$, 
such that $\rho$ is equal to the composition 
$\GL_d(\varphi)\circ \hat{\rho}$. The claim implies
that the map $R^{\square}_{\rho_x}\rightarrow \widehat{B}$ constructed above is an isomorphism.

The proof of the claim is based  
on \cite[Proposition 9.5]{kisin_over}. Following 
its proof, we may construct an ascending chain of 
local open $\Lambda^0$-subalgebras $A_n^0$ of $A$ for $n\ge 1$, such 
that for all $n$ the following hold: $A_n^0[1/t]=A$, the image of $A_n^0$ under the 
projection $b: A \rightarrow \kappa(x)$ 
is equal to $\OO_{\kappa(x)}$ and 
$\bigcup_{n\ge 1} A_n^0= b^{-1}(\OO_{\kappa(x)})$. 
Let $M\in \GL_d(\kappa(x))$ be a matrix such that the 
image of $G_F$ under $M \rho_x M^{-1}$ is contained 
in $\GL_d(\OO_{\kappa(x)})$. Let 
$x'\in X^{\gen}$ correspond to the representation $M \rho_x M^{-1}$. Then $\kappa(x')=\kappa(x)$ and the image of 
$x': A^{\gen}\rightarrow \kappa(x)$ is contained in $\OO_{\kappa(x)}$. 
Let $z\in X^{\gen}$ be the composition 
$z: A^{\gen}\overset{x'}\longrightarrow \OO_{\kappa(x)}\rightarrow k'$, where  $k'$ is the residue field of $\OO_{\kappa(x)}$, 
let $\widetilde{M}\in \GL_d(A)$ be a matrix lifting $M$ and let $\rho':= \widetilde{M} \rho \widetilde{M}^{-1}$. Since $G_F$ is compact  $\rho'(G_F)$ 
will be contained in some $\GL_d(A^0_n)$ for $n \gg 0$. 
We may consider $\rho': G_F \rightarrow \GL_d(A^0_n)$ as 
a deformation of $\rho_z$ to $A^0_n$. Since the pseudo-character of $\rho_z$ is equal to 
$\Dbar\otimes_k k'$ by Lemma \ref{closed_pts},
the pseudo-character of $\rho':G_F\rightarrow \GL_d(A^0_n)$ is a deformation 
of $\Dbar\otimes_k k'$ to $A^0_n$ and hence induces a map 
of local $\OO$-algebras $R^{\ps}\rightarrow A_n^0$. 
Thus  $\rho'$ factors through the map 
$ \Lambda^0\otimes_{\OO} R^{\ps}\br{G_F}\rightarrow M_d(A^0_n)$, which will factor through the Cayley--Hamilton quotient
$(\Lambda^0\otimes_{\OO} R^{\ps}\br{G_F})/ \CH(\Lambda^0\otimes_\OO D^u)\rightarrow M_d(A^0_n)$.
It follows from  \cite[Section 1.22]{che_durham} or  \cite[Lemma 1.1.8.6]{WE_thesis} that 
$$\Lambda^0\otimes_{\OO} E\cong (\Lambda^0\otimes_{\OO} R^{\ps}\br{G_F})/ \CH(\Lambda^0\otimes_\OO D^u).$$
After inverting $t$ and conjugating by $\widetilde{M}^{-1}$ we obtain a map of $\Lambda^0[1/t]$-algebras 
$\Lambda^0[1/t]\otimes_{\OO} E \rightarrow M_d(A)$, such that if we compose this map with the map induced by 
$G_F\rightarrow R^{\ps}\br{G_F}\rightarrow E$ then we get back $\rho$. Since $A$ is an Artinian $\Lambda$-algebra, $\varpi^n \Lambda^0[1/t]$ will be mapped to zero for $n\gg 0$, and thus the map extends
to a map of $\Lambda$-algebras 
$\alpha:\Lambda\otimes_{\OO} E\rightarrow M_d(A)$. 
The universal property of $j: E\rightarrow M_d(A^{\gen})$ implies that there is a unique map 
of $\Lambda$-algebras $\varphi: \widehat{B}\rightarrow A$, such that $M_d(\varphi)\circ (\id\otimes j) = \alpha$. 

It remains to show the uniqueness of the map $\varphi$,
which is equivalent to showing that there is at most one map of $\Lambda\otimes_{\OO} R^{\ps}$-algebras $\alpha: \Lambda\otimes_{\OO} E\rightarrow M_d(A)$ such that the composition
with $G_F\rightarrow \Lambda \otimes_\OO E$ gives $\rho$. 
It follows from the Cayley--Hamilton theorem in 
$M_d(A)$ and \cite[Corollary 1.14]{che_durham},  that the map $\Lambda\otimes_{\OO} R^{\ps}\rightarrow \Lambda\otimes_{\OO} E\overset{\alpha}{\longrightarrow} A$ is uniquely determined by $\rho$. Thus $\alpha$ 
is uniquely determined on the image 
of $\Lambda\otimes_{\OO} R^{\ps}[G_F]$ 
in $\Lambda\otimes_{\OO} E$.  The map 
$R^{\ps}[G_F]\rightarrow E$ is surjective, since the image is dense and closed as $E$ is a finitely generated $R^{\ps}$-module,  hence $\alpha$ is uniquely determined by $\rho$. 
\end{proof}

The following Lemma is a mild generalization of 
\cite[Lemma 3.3.5]{bj}.

\begin{lem}\label{333} Let $R$ be a complete local Noetherian 
$k$-algebra with residue field $k$, let $A$ be a finitely
generated $R$-algebra, let $\pp\in \Spec A$ such that 
its image in $\Spec R$ lies in $P_1R$, and  let $\qq$ be the
kernel of the map 
$$B:=\kappa(\pp)\otimes_k A \rightarrow \kappa(\pp), \quad 
x\otimes a \mapsto x (a+\pp).$$
Then $\hat{B}_{\qq}\cong \hat{A}_{\pp}\br{T}$. In particular, 
$A_{\pp}$ is regular (resp.~complete intersection) if and only  if $\hat{B}_{\qq}$ is.
\end{lem}
\begin{proof} Let $\pp'$ be the image of $\pp$ in $\Spec R$. 
Since by assumption $\pp'\in P_1R$, the residue field 
$\kappa(\pp')$ a local field of characteristic $p$. Since
$A$ is finitely generated over $R$, $\kappa(\pp)$ is a 
finite extension of $\kappa(\pp')$ and thus is also a local field of characteristic $p$. 
The proof of \cite[Lemma 3.3.4]{bj} goes through verbatim 
by replacing $R$ with $A$ everywhere. 
\end{proof}

\begin{lem}\label{335} Let $R$ be a complete local Noetherian 
$\OO$-algebra with residue field $k$, let $A$ be a finitely
generated $R$-algebra, let $\pp\in \Spec A$ such that 
$\kappa(\pp)$ is a local field of characteristic $p$, and  let $\qq$ be the
kernel of the map 
$$B:=\Lambda\otimes_\OO A \rightarrow \kappa(\pp), \quad 
\lambda\otimes a \mapsto \bar{\lambda} (a+\pp),$$
Then $\hat{B}_{\qq}\cong \hat{A}_{\pp}\br{T}$. In particular, 
$A_{\pp}$ is regular (resp.~complete intersection) if and only  if $\hat{B}_{\qq}$ is.
\end{lem}
\begin{proof} We first observe that $\hat{B}_{\qq}$ is flat 
over $\hat{A}_{\pp}$.  This can be seen as follows. Since $\Lambda$ is $\OO$-flat, $B$ is $A$-flat. Since $B_{\qq}$ is $B$-flat and 
$\hat{B}_{\qq}$ is $B_{\qq}$-flat, we conclude that $\hat{B}_{\qq}$ 
is $A$-flat. Thus $\hat{B}_{\qq}\otimes_A \hat{A}_{\pp}$ is 
$\hat{A}_{\pp}$-flat. This ring is isomorphic to $\hat{B}_{\qq}\otimes_{A_{\pp}} \hat{A}_{\pp}$. Since the map $A_{\pp}/ \pp^n A_{\pp}\rightarrow 
\hat{A}_{\pp}/ \pp^n \hat{A}_{\pp}$ is an isomorphism for all $n\ge 1$, 
$\hat{B}_{\qq}$ is a completion of $\hat{B}_{\qq}\otimes_{A_{\pp}} \hat{A}_{\pp}$ at $\qq$, which implies the claim. 

It follows from Lemma \ref{333} that the map $A\rightarrow B$, $a\mapsto 1\otimes a$
induces a map of local rings $\hat{A}_{\pp}\rightarrow \hat{B}_\qq$, 
such that $\hat{B}_{\qq}/\varpi\cong (\hat{A}_{\pp}/\varpi)\br{T}$. 
By choosing $b\in \hat{B}_{\qq}$, which maps to $T$ under this isomorphism, 
we obtain a map $\varphi: \hat{A}_{\pp}\br{T} \rightarrow  \hat{B}_{\qq}$, which induces an isomorphism modulo $\varpi$. Thus $\varphi$  is a homomorphism 
of pseudo-compact $\hat{A}_{\pp}$-modules, which induces an isomorphism after applying $\otimes_{\hat{A}_{\pp}} \kappa$. Thus 
$(\coker \varphi)\otimes_{\hat{A}_{\pp}} \kappa=0$, and since 
$\hat{B}_{\qq}$ is $\hat{A}_{\pp}$-flat, $(\ker \varphi)\otimes_{\hat{A}_{\pp}} \kappa=0$. Topological Nakayama's lemma\footnote{Lemma 0.3.3 in Expos\'e $VII_B$ in SGA3.} for pseudo-compact modules implies that $\coker \varphi$
and $\ker \varphi$ are both zero. 
\end{proof} 

\begin{lem}\label{334} Let $R$ be a complete local Noetherian 
$\OO$-algebra with residue field $k$, let $A$ be a finitely
generated $R$-algebra, let $\pp\in \Spec A$ such that $\kappa(\pp)$ is either a finite extension of $L$ or a finite extension of $k$.
Let $\qq$ be the
kernel of the map 
$$B:=\Lambda\otimes_{\OO} A \rightarrow \kappa(\pp), \quad 
\lambda\otimes a \mapsto \bar{\lambda} (a+\pp).$$
Then $\hat{B}_{\qq}\cong \hat{A}_{\pp}$. 
\end{lem}
\begin{proof} The completion of $\Lambda\otimes_{\OO} \Lambda$ with respect to the kernel of $\Lambda\otimes_{\OO} \Lambda \rightarrow \Lambda$, $x\otimes y\mapsto xy$ is just $\Lambda$ (and that is 
why we don't get an extra variable $T$ like in Lemma \ref{333}, see \cite[Lemma 3.3.5]{bj}.) The rest of the proof is the same as the 
proof of Lemma \ref{333}.
\end{proof}

\begin{cor}\label{ci} Let $x$ be either a closed point of $Y$ or a closed point of  $X^{\gen}\setminus Y$. Then the following hold:
\begin{enumerate}
\item $R^{\square}_{\rho_x}$ is a flat
$\Lambda$-algebra of relative  dimension $d^2+ d^2[F:\Qp]$ and is complete intersection;
\item if $\cha(\kappa(x))=p$ then $R^{\square}_{\rho_x}/\varpi$ is complete intersection of dimension $d^2+ d^2[F:\Qp]$.
\end{enumerate}
\end{cor}
\begin{proof} Let us assume that $\kappa(x)$ is a
finite extension of $k$. It follows from Proposition \ref{main_complete} and 
Lemma \ref{334} that $R^{\square}_{\rho_x}/\varpi\cong \widehat{\OO}_{\Xbar^{\gen}, x}$, the 
completion of the local ring of $\Xbar^{\gen}$ at $x$ with respect to the maximal ideal. We have 
$\dim \widehat{\OO}_{\Xbar^{\gen}, x}=\dim \OO_{\Xbar^{\gen}, x}\le \dim \Xbar^{\gen}$, 
and thus by Theorem \ref{main} we obtain the bound
$$\dim R^{\square}_{\rho_x}/\varpi \le \dim \Xbar^{\gen} \le d^2+ d^2[F:\Qp]= r-s,$$
where the last equality is \eqref{r_minus_s}. 
It follows from \eqref{present_square_x} 
that $\dim R^{\square}_{\rho_x}/\varpi \ge r-s$ and $\dim R^{\square}_{\rho_x} \ge 1+r-s$. Thus the lower bounds
of the dimensions are equalities, and $\varpi, f_1, \ldots, f_s$ are a part of system of parameters
in $\Lambda\br{x_1,\ldots, x_r}$. Thus they form a regular sequence in $\Lambda\br{x_1,\ldots, x_r}$ and so $R^{\square}_{\rho_x}$ and $R^{\square}_{\rho_x}/\varpi$ are complete intersections of the claimed dimensions. Moreover, 
since $\Lambda$ is a DVR with uniformiser $\varpi$, 
flatness is equivalent to $\varpi$-torsion equal to zero, and hence $R^{\square}_{\rho_x}$ is flat over $\Lambda$.   

Let us assume that $\kappa(x)$ is a local field of
characteristic $p$. Proposition \ref{main_complete} and Lemma \ref{333} imply that 
$R^{\square}_{\rho_x}/\varpi\cong \widehat{\OO}_{\Xbar^{\gen}, x}\br{T}$ and Lemma 
\ref{fix} applied with $W=\Xbar^{\gen}$ implies that $\dim \widehat{\OO}_{\Xbar^{\gen}, x}\le \dim \Xbar^{\gen}-1$. Thus  $\dim R^{\square}_{\rho_x}/\varpi \le \dim \Xbar^{\gen}$ and the same argument as above goes through. 

If $\kappa(x)$ is a finite extension of $L$ then Proposition \ref{main_complete} and 
Lemma \ref{334} imply that $$R^{\square}_{\rho_x}\cong \widehat{\OO}_{X^{\gen}, x}=\widehat{\OO}_{X^{\gen}[1/p],x}.$$ Corollary \ref{dim_gen_sp} implies that $\dim R^{\square}_{\rho_x}\le \dim X^{\gen}[1/p]\le \dim \Xbar^{\gen}$. Then the same argument goes through. 
\end{proof} 

\begin{cor}\label{local_rings} Let $x$ be either a closed point in $Y$ or a closed point in $X^{\gen}\setminus Y$
and let $\widehat{\OO}_{X^{\gen}, x}$  be the completion with respect to the maximal ideal of the local ring at $x$. If $\kappa(x)$ is a finite extension of $k$ or $L$ then $\widehat{\OO}_{X^{\gen}, x}\cong R^{\square}_{\rho_x}$. If $\kappa(x)$ is a local 
field of characteristic $p$ then $R^{\square}_{\rho_x}\cong \widehat{\OO}_{X^{\gen}, x}\br{T}$.  
\end{cor}
\begin{proof} If $\kappa(x)$ is a finite extension 
of $k$ or $L$ then the assertion follows from Proposition \ref{main_complete} and Lemma \ref{334}. If $\kappa(x)$ is a local 
field of characteristic $p$ then the assertion follows from Proposition \ref{main_complete} and Lemma \ref{335}.
\end{proof}

 \begin{cor}\label{ci_gen} 
 The following hold:
 \begin{enumerate}
 \item $A^{\gen}$ is $\OO$-torsion free,
 equi-dimensional of dimension $1+ d^2+ d^2[F:\Qp]$ and is locally complete intersection;
 \item $A^{\gen}/\varpi$ is  equi-dimensional of dimension $d^2+d^2[F:\Qp]$ and is locally complete intersection. 
 \end{enumerate}
 \end{cor}
\begin{proof} Let us prove (1) as the proof of (2) is identical. Corollary \ref{local_rings} together 
with Corollary \ref{ci} implies that the local rings
at closed points of $X^{\gen}$ are $\OO$-torsion free and complete intersection.  This implies that $A^{\gen}$ is $\OO$-torsion free and $A^{\gen}$ is 
locally complete intersection by \cite[\href{https://stacks.math.columbia.edu/tag/09Q5}{Tag 09Q5}]{stacks-project}.

Let $Z$ be an irreducible component of $X^{\gen}$. Lemma \ref{fix} implies that there is a closed point 
$x\in Z$ such that $x$ maps to the closed point of $X^{\ps}$. Moreover, $\dim Z= \dim \OO_{Z, x}$. 
Since $\OO_{X^{\gen}, x}$ is complete intersection, 
it is equi-dimensional and thus $\dim \OO_{Z,x}=\dim
\OO_{X^{\gen}, x}= d^2+d^2[F:\Qp]+1$, where the last equality follows from Corollaries \ref{ci} and \ref{local_rings}. 
\end{proof} 

\begin{prop}\label{331} Let $x\in P_1 R^{\square}_{\rhobar}$, where
$R^{\square}_{\rhobar}$ is the framed deformation ring 
of $\rhobar: G_F \rightarrow \GL_d(k')$, where $k'$ is
finite extension of $k$. Let $\rho_x: G_F \rightarrow \GL_d(\kappa(x))$ be the representation obtained by specializing the universal framed deformation of $\rhobar$ at $x$. Let $\qq$ be the kernel of the 
map $$\Lambda\otimes_{\OO} R^{\square}_{\rhobar} \rightarrow \kappa(x), \quad \lambda \otimes a \mapsto \bar{\lambda}\bar{a},$$
where $\Lambda$ is the ring defined at the beginning of the subsection.  Then
the completion of $(\Lambda\otimes_{\OO} R^{\square}_{\rhobar})_{\qq}$ 
with respect to the maximal ideal is naturally isomorphic to $R^{\square}_{\rho_x}$.
\end{prop}
\begin{proof} The proof is similar to the proof
of Proposition \ref{main_complete}, but easier, since the setting is much closer to the setting of \cite[Proposition 9.5]{kisin_over} or \cite[Theorem 3.3.1]{bj}, where an analogous result is proved for versal deformation rings.
We leave the details to the reader.
\end{proof}

Let $x$ be a closed point of $X^{\gen}\setminus Y$, so that 
 $\kappa(x)$ is a local field. Since $G_F$ is compact there is 
a matrix $M\in \GL_d(\kappa(x))$, such 
that the image of $M \rho_x M^{-1}$ 
is contained in $\GL_d(\OO_{\kappa(x)})$. Let 
$x':A^{\gen}\rightarrow \OO_{\kappa(x)}$ be the $R^{\ps}$-algebra homomorphism corresponding to the 
representation $E\rightarrow M_d(\OO_{\kappa(x)})$, $a\mapsto M \rho_x(a) M^{-1}$. We will denote the 
corresponding Galois representation by 
$\rho_{x'}^0: G_F\rightarrow \GL_d(\OO_{\kappa(x)})$ and let $\rho_{x'}$
be the composition  
$\rho_{x'}:G_F \overset{\rho_{x'}^0}{\longrightarrow} \GL_d(\OO_{\kappa(x)})\rightarrow  \GL_d(\kappa(x))$. We note that $\kappa(x')=\kappa(x)$ and let $\Lambda$ be the coefficient ring defined at the beginning of the subsection. 
Let $k'$ be the residue field
of $\OO_{\kappa(x)}$ and let $\rho_z: G_F\rightarrow \GL_d(k')$ be the
representation corresponding to 
$z: A^{\gen}\overset{x'}{\longrightarrow} \OO_{\kappa(x)}\rightarrow k'$. 
Then $\rho^0_{x'}$ is a deformation of 
$\rho_z$ to $\OO_{\kappa(x)}$, thus the 
map $x': A^{\gen}\rightarrow \OO_{\kappa(x)}$ factors through 
$x': R^{\square}_{\rho_z}\rightarrow \OO_{\kappa(x)}$.

\begin{cor}\label{the_same_comp} There is an isomorphism of 
local $\Lambda$-algebras between $R^{\square}_{\rho_x}$, $R^{\square}_{\rho_{x'}}$ and the completion of $(\Lambda\otimes_{\OO} R^{\square}_{\rho_z})_{\qq}$ with respect to the maximal ideal, where $\qq$ is as in Proposition \ref{331}
with respect to $x': R^{\square}_{\rho_z}\rightarrow \OO_{\kappa(x)}.$
\end{cor}
\begin{proof} Let $\widetilde{M}$ be 
any lift of $M$ to $M_d(\Lambda)$. Since
$\Lambda$ is a local ring $\det \widetilde M$ is a unit in $\Lambda$ and hence $\widetilde{M}\in \GL_d(\Lambda)$. Conjugation by $\widetilde{M}$ induces 
an isomorphism between the deformation
problems for $\rho_x$ and $\rho_{x'}$ and hence between the deformation rings. Proposition \ref{331} implies that these
rings are also isomorphic to the completion of $(\Lambda\otimes_{\OO} R^{\square}_{\rho_z})_{\qq}$. 
\end{proof}

\begin{remar}\label{strategy} Corollary \ref{the_same_comp} 
enables us  to study local properties 
of $X^{\gen}$, by  studying
the completions of local rings at closed points above $\mm_{R^{\ps}}$. 
For example, if we could show that
$R^{\square}_{\rho_z}$ is regular, we could conclude that the local ring at 
$x'$, $(R^{\square}_{\rho_z})_{x'}$ is regular, and hence that the completion 
$\widehat{(R^{\square}_{\rho_z})}_{x'}$
is regular. If $\kappa(x)$ is a local field 
of characteristic $p$ then Proposition \ref{main_complete}, Corollary \ref{the_same_comp} and Lemma \ref{335} imply that 
$$ \widehat{\OO}_{X^{\gen}, x}\br{T}\cong R^{\square}_{\rho_x}\cong R^{\square}_{\rho_{x'}}\cong \widehat{(R^{\square}_{\rho_z})}_{x'}\br{T}.$$
If $\kappa(x)$ is a finite extension of $L$
then Proposition \ref{main_complete}, Corollary \ref{the_same_comp} and
Lemma \ref{334} imply that 
$$ \widehat{\OO}_{X^{\gen}, x}\cong R^{\square}_{\rho_x}\cong R^{\square}_{\rho_{x'}}\cong \widehat{(R^{\square}_{\rho_z})}_{x'}.$$
Thus in both cases we can deduce that  
$\widehat{\OO}_{X^{\gen}, x}$ and hence $\OO_{X^{\gen}, x}$ are regular. Thus if 
we can show that $R^{\square}_{\rho_z}$
is regular for all closed points $z\in X^{\gen}$ above $\mm_{R^{\ps}}$ then we can conclude that $\OO_{X^{\gen}, x}$ is regular for all closed points $x\in X^{\gen}$ and thus $X^{\gen}$ is regular. 

Of course, one may also reverse the logic of this argument: if $X^{\gen}$ is regular then all its local rings and their completions are regular 
and hence $R^{\square}_{\rho_z}$ is regular for all closed points $z\in X^{\gen}$ above $\mm_{R^{\ps}}$. 
\end{remar}

\begin{cor}\label{cor_RforLocalK} Let $\rho: G_F\rightarrow \GL_d(\kappa)$ be a continuous representation with $\kappa$ a local field. Then the conclusion of Corollary \ref{ci} holds for
$R^{\square}_{\rho}$.
\end{cor} 
\begin{proof} After conjugation we may assume that 
$\rho(G_F)\subset \GL_d(\OO_{\kappa})$. Let $\rhobar$ 
be the representation obtained by reducing the matrix 
entries modulo a uniformizer of $\OO_{\kappa}$ and 
let $\Dbar$ be the associated pseudo-character. Corollary \ref{ci} 
applies to $R^{\square}_{\rhobar}$. 
Since $\rho$ corresponds to an $x\in P_1 R^{\square}_{\rhobar}$, 
Proposition \ref{331} together with Lemmas \ref{334}, \ref{335} allows us to bound 
the dimension of $R^{\square}_{\rho}$ from above. Then the proof of Corollary \ref{ci} 
carries over. 
\end{proof}

\begin{cor}\label{char0_lift} Every
representation 
$\rhobar: G_F\rightarrow \GL_d(k)$ 
can be lifted to characteristic zero.
\end{cor} 
 \begin{proof} It follows from Corollary \ref{ci} that $R^{\square}_{\rhobar}[1/p]$ is non-zero. 
We may obtain a lift by specializing the universal framed deformation along any 
 $\OO$-algebra homomorphism $x: R^{\square}\rightarrow \Qpbar$. 
 \end{proof}

\subsection{Bounding the maximally reducible semi-simple locus}\label{bound_mrs}

Writing $\Dbar = \prod_{i=1}^m \Dbar_i$ with $\Dbar_i$ absolutely irreducible pseudo-characters, we now take $\PP = \PP_{\max}$ and consider the finite (by Lemma \ref{iotaP_finite}) $R^{\ps}$-algebra $R^{\ps}_{\underline{\Sigma}}$, where 
$\underline{\Sigma}$ amounts to some choice of ordering of $\set{1,\dots,m}$. Note that if $\rhobar_i:G_F\rightarrow \GL_d(k)$ is an (absolutely irreducible) representation with pseudo-character $\Dbar_i$ then
    \[ R^{\ps}_{\underline{\Sigma}} \cong R_{\rhobar_1} \wtimes_{\OO} \cdots \wtimes_{\OO} R_{\rhobar_m} \]
where $R_{\rhobar_i}$ denotes the universal deformation ring of $\rhobar_i$. So let $\rho_i^{\univ}: G_F \to \GL_{d_i}(R_{\rhobar_i})$ denote a representative of the strict equivalence class of the universal representation for each $i = 1,\dots,m$. If we let $M$ denote the universal invertible matrix in $\GL_d(\OO_{\GL_d}(\GL_d))$, then the representation
    \[ M \times \diag(\rho_1^{\univ}, \dots, \rho_m^{\univ}) \times M^{-1}: G_F \to \GL_d(R^{\ps}_{\underline{\Sigma}} \otimes_\OO \OO_{\GL_d}(\GL_d)) \]
gives rise to a map of Cayley--Hamilton algebras $E \to M_d(R^{\ps}_{\underline{\Sigma}} \otimes_\OO \OO_{\GL_d}(\GL_d))$ which satisfies the universal property of $A^{\gen}$ and so defines a map of $R^{\ps}$-schemes
    \[ \GL_d \times_\OO X^{\ps}_{\underline{\Sigma}} \to X^{\gen} \]
which descends to a map of $R^{\ps}$-schemes
    \[ \eta_{\underline{\Sigma}}: \GL_d/Z_L \times_\OO X^{\ps}_{\underline{\Sigma}} \to X^{\gen} \]
where $L:=L_{\underline{\Sigma}}$ denotes the standard Levi subgroup of $\GL_d$ with blocks corresponding to $\underline{\Sigma}$ and $Z_L$ denotes its center.

\begin{defi}
The \textit{maximally reducible semi-simple locus} $X^{\mrs} \subset X^{\gen}$ is the scheme-theoretic image of $\eta_{\underline{\Sigma}}: \GL_d/Z_L \times_\OO X^{\ps}_{\underline{\Sigma}} \to X^{\gen}$.
\end{defi}

\begin{lem}\label{mrs_points}
Let  $x \in X^{\gen}$ and let $y$ be the image of $x$ in $X^{\ps}$. If $y$ lies in $X^{\ps}_{\PP_{\max}}$ and $\rho_x$ is semi-simple then $x\in X^{\mrs}$. Moreover, such points are 
dense in $X^{\mrs}$.
\end{lem}
\begin{proof} We first note that if
$x\in X^{\gen}$ maps to $X^{\ps}_{\PP_{\max}}$ 
and $\rho_x$ is semi-simple then 
$\rho_x\cong \rho_1\oplus\ldots\oplus \rho_m$, 
with each $\rho_i$ an irreducible representation of $G_F$ lifting $\rhobar_i$.
By conjugating by an element of $h\in \GL_d(\kappa(x))$ we may ensure that  
$h^{-1} \rho_x(g) h= \diag(\rho_1(g), \ldots, \rho_m(g))$ for all $g\in G_F$ and this implies
that $x\in X^{\mrs}$. 

Since $\eta_{\underline{\Sigma}}$ is a map
of affine schemes, it is affine and hence quasi-compact, see 
\cite[\href{https://stacks.math.columbia.edu/tag/01S5}{Tag 01S5}]{stacks-project}. It follows from \cite[\href{https://stacks.math.columbia.edu/tag/01R8}{Tag 01R8}]{stacks-project} that the set theoretic image of 
$\eta_{\underline{\Sigma}}$ is dense in $X^{\mrs}$.
\end{proof}

\begin{prop}\label{mrs_bound}
$\dim X^{\mrs} \leq 1+d^2 +[F:\Qp]  \sum_{i=1}^m d_i^2$.
\end{prop}
\begin{proof}
The open subscheme $U_{\max} = X^{\ps} \setminus \set{\mm_{R^{\ps}}} \subset \Xbar^{\ps}$ is Jacobson by Lemma \ref{jacobson}, as is $V_{\max} := X^{\mrs} \times_{X^{\ps}} U_{\max}$. Let $Z_{\max}$ denote the closure of $V_{\max}$ in $X^{\mrs}$. The formation of scheme-theoretic images commutes with restriction to opens, so the map
    \[ (\GL_d/Z_L \times_\OO X^{\ps}_{\underline{\Sigma}}) \times_{X^{\ps}} U_{\max} \to V_{\max} \]
is a dominant map of Jacobson Noetherian excellent schemes. Applying Lemma \ref{jacobson_dom} we see that
    \[ \dim V_{\max} \leq \dim ((\GL_d/Z_L\times_{\OO} X^{\ps}_{\underline{\Sigma}}) \times_{X^{\ps}} U_{\max}). \]
Since $X^{\mrs}$ is by definition a non-empty closed $\GL_d$-invariant subscheme of $X^{\gen}$, Lemma \ref{fix} implies that every irreducible component of $X^{\mrs}$ has a point in common with the preimage of $\mm_{R^{\ps}}$ in $X^{\mrs}$. Therefore, Lemma \ref{jacobson} (5) implies that
    \[ \dim Z_{\max} = \dim V_{\max} + 1 \]
Furthermore, $\GL_d/Z_L$ is flat over $\Spec \OO$ with geometrically irreducible fibres, 
so the projection $\GL_d/Z_L\times_{\OO} X^{\ps}_{\underline{\Sigma}}\rightarrow X^{\ps}_{\underline{\Sigma}}$ is a flat (and hence open) map with irreducible fibres. It follows from \cite[\href{https://stacks.math.columbia.edu/tag/037A}{Tag 037A}]{stacks-project} that this 
map induces a bijection between the sets of irreducible components. Since $R^{\ps}_{\underline{\Sigma}}$ is a local ring, we deduce that $\GL_d/Z_L\times_{\OO} X^{\ps}_{\underline{\Sigma}}$ satisfies the 
assumptions of Lemma \ref{jacobson} (5) and thus Lemma \ref{jacobson} (5) implies that
    \[ \dim \GL_d/Z_L\times_\OO X^{\ps}_{\underline{\Sigma}} = \dim ((\GL_d/Z_L \times_\OO X^{\ps}_{\underline{\Sigma}}) \times_{X^{\ps}} U_{\max}) + 1 \]
Since $\dim X^{\ps}_{\underline{\Sigma}}= 
1+\sum_{i=1}^m (1+d_i^2[F:\Qp])$ and 
the relative dimension of $\GL_d/Z_L$ over $\OO$ is $d^2-m$ we get that 
$$\dim Z_{\max}\le \dim \GL_d/Z_L\times_\OO X^{\ps}_{\underline{\Sigma}}=1+d^2 + [F:\Qp]\sum_{i=1}^m d_i^2.$$

Let $Y^{\mrs}$ be the scheme theoretic image of 
$\GL_d/ Z_L \times_{\OO} \{\mm_{R^{\ps}}\}\rightarrow Y$. 
Since $Y$ is of finite type over $k$ the same argument as above shows that 
$$\dim Y^{\mrs} \le \dim (\GL_d/ Z_L \times_{\OO} \{\mm_{R^{\ps}}\})= d^2-m.$$ 

Now $Z_{\max}\cup Y^{\mrs}$ is a closed subscheme of $X^{\gen}$ containing the image of $\eta_{\underline{\Sigma}}$. It follows from Lemma \ref{mrs_points} that  $Z_{\max}\cup Y^{\mrs}$ will contain 
$X^{\mrs}$. Hence, $$\dim X^{\mrs}\le \max\{ \dim Z_{\max}, \dim Y^{\mrs}\}= \dim Z_{\max}.$$
\end{proof}

\begin{cor}\label{ss_bound}$ \dim \Xbar^{\mrs}
=\dim X^{\mrs}-1\leq d^2  + [F:\Qp]\sum_{i=1}^m d_i^2$.
\end{cor}
\begin{proof} It follows from Corollary \ref{ci} that $R^{\ps}_{\underline{\Sigma}}$ 
is $\OO$-torsion free, which implies that 
$\GL_d/Z_L\times_{\OO} X^{\ps}_{\underline{\Sigma}}$ is flat over 
$\Spec \OO$ and the same applies for $X^{\mrs}$. (Here we are simply saying that a subring of $\OO$-torsion free ring is $\OO$-torsion free.) Thus for all $x\in \Xbar^{\mrs}$, $\varpi$ is a regular element 
in $\OO_{X^{\mrs}, x}$ and so 
$\dim \OO_{\Xbar^{\mrs}, x}= \dim \OO_{X^{\mrs}, x}-1$. This implies 
$\dim \Xbar^{\mrs}= \dim X^{\mrs}-1$ and the 
inequality  follows from the Proposition \ref{mrs_bound}. 
\end{proof}

\begin{remar} One could study the closure of the reducible semi-simple locus corresponding to more general 
partitions using a similar argument. We don't pursue this here, since we need the bound only for $d = 2$ and $F=\mathbb{Q}_2$ when we apply it to Case 3 in the proof of Proposition \ref{nice_open} below.
\end{remar}

\subsection{Density of the irreducible locus}\label{sec_dens_irr} Let us first unravel the definitions of $U_{\PP_{\min}}$ 
 and $V_{\PP_{\min}}$ in Section~\ref{sec_dim_sp}. We have that $U_{\PP_{\min}}$ is 
 an open subscheme of $\Xbar^{\ps}$ such that 
 the closed points of $U_{\PP_{\min}}$ are in bijection with $\pp\in P_1(R^{\ps}/\varpi)$, 
 such that the specialization of the universal pseudo-character along $R^{\ps}\rightarrow \kappa(\pp)$ 
 is absolutely irreducible. Now $V_{\PP_{\min}}$ is the preimage of $U_{\PP_{\min}}$ in 
 $\Xbar^{\gen}$, so that it is an open subscheme of $\Xbar^{\gen}$ and its closed points are in bijection 
 $\qq \in \Xbar^{\gen}$, which map to $P_1(R^{\ps}/\varpi)$ in $\Xbar^{\ps}$,  such that the representation 
 $$ E\overset{j}{\rightarrow} M_d(A^{\gen})\rightarrow M_d(\kappa(\qq))$$
 is absolutely irreducible. 
 
 \begin{prop}\label{VPmin_dense} $V_{\PP_{\min}}$ is dense in $\Xbar^{\gen}$. 
 \end{prop}
 \begin{proof} We have
 $$ \Xbar^{\gen}\setminus V_{\PP_{\min}}= Y \cup \bigcup_{\PP_{\min}< \PP} Z_{\PP}$$
 and it follows from Lemmas \ref{bound_diff}, \ref{bound_Y} that $\Xbar^{\gen}\setminus V_{\PP_{\min}}$ has positive 
 codimension in $\Xbar^{\gen}$. 
 Since $\Xbar$ is equi-dimensional by Corollary \ref{ci_gen} 
 we conclude that $V_{\PP_{\min}}$ is dense in $\Xbar^{\gen}$. In particular, the inequality 
 in Proposition \ref{dim_Pmin} is an equality. 
 \end{proof} 
 
 We will now prove a stronger version of the above result. Following \cite[Definition 5.1.2]{bj} 
 we call $y\in U_{\PP_{\min}}$ \textit{special} if either $\zeta_p\not\in F$ and $D_y= D_y(1)$ or 
 $\zeta_p\in F$ and the restriction $D_y$ to $G_{F'}$ is reducible for some degree $p$ Galois 
 extension $F'$ of $F$. Otherwise, $y$ is called \textit{non-special}. According to 
 \cite[Lemma 5.1.3]{bj} there is a closed subscheme $U^{\spcl}$ of $U_{\PP_{\min}}$ such that 
 the closed points of $U^{\spcl}$ are precisely the closed special points of $U_{\PP_{\min}}$. 
 Let $V^{\spcl}$ denote the preimage of $U^{\spcl}$ in $\Xbar^{\gen}$ and let $Z^{\spcl}$ denote the 
 closure of $V^{\spcl}$.
 
 Similarly 
 let $U^{\Kirr}\subset U_{\PP_{\min}}$ 
 be the Kummer-irreducible locus 
 defined in Appendix \ref{appendix}. Let $U^{\Kred}$ denote its complement in $U_{\PP_{\min}}$, let 
 $V^{\Kred}$ be the preimage of 
 $U^{\Kred}$ in $\Xbar^{\gen}$ and 
 let $Z^{\Kred}$ denote the closure
 of $V^{\Kred}$. We have $V^{\spcl} \subseteq V^{\Kred}$ with equality if $\zeta_p\in F$ and thus $Z^{\spcl}\subseteq Z^{\Kred}$. 
 
 \begin{lem}\label{special} We have 
 $$\dim \Xbar^{\gen}- \dim Z^{\spcl}   \ge \frac{1}{2}[F:\Qp]d^2, \quad \dim \Xbar^{\gen} - \dim Z^{\Kred}   \ge [F:\Qp]d.$$
 \end{lem}
 \begin{proof} It follows from \cite[Theorem 5.4.1 (a)]{bj} that the dimension of the Zariski closure of $U^{\spcl}$ in $\Xbar^{\ps}$ is at most $1+ \frac{1}{2}[F:\Qp]d^2$. If $y \in U^{\spcl}$ then 
 its fibre $X_y^{\gen}$ has dimension $d^2-1$ by Corollary \ref{fibre_irr}. Thus Lemma \ref{jacobson} 
 implies that $$\dim Z^{\spcl} \le d^2+ \frac{1}{2}[F:\Qp]d^2.$$
  Since $\dim \Xbar^{\gen}= d^2+d^2[F:\Qp]$ 
 by Corollary \ref{ci_gen} the assertion follows. Similarly Proposition \ref{existence_Kirr} implies that the dimension of the closure of $U^{\Kred}$ in $\Xbar^{\ps}$ is at most $1+ (d^2-d)[F:\Qp]$. The same argument gives the required bound 
 for the codimension of $Z^{\Kred}$.
 \end{proof} 
 
 Let $U^{\nspcl}:= U_{\PP_{\min}}\setminus U^{\spcl}$ and let $V^{\nspcl}$ the preimage 
 of $U^{\nspcl}$ in $\Xbar^{\gen}$. Let $V^{\Kirr}$ be the preimage
 of $U^{\Kirr}$ in $\Xbar^{\gen}$. We have an inclusion 
 $V^{\Kirr}\subset V^{\nspcl}$ and the subschemes coincide
 if $\zeta_p\in F$. 
 
 \begin{prop}\label{nspcl}
 $V^{\Kirr}$ is Zariski dense in
 $\Xbar^{\gen}$. Moreover, the following hold:
 \begin{enumerate}
 \item if $d=2$ then $\dim \Xbar^{\gen} - \dim (\Xbar^{\gen}\setminus V^{\Kirr})\ge [F:\Qp]$; 
 \item if $d>2$ then $\dim \Xbar^{\gen} - \dim (\Xbar^{\gen}\setminus V^{\Kirr})\ge 1+[F:\Qp]$.
 \item if $d > 1$ is arbitrary but $\Dbar$ is absolutely irreducible (i.e.~$m = 1$) then $\dim \Xbar^{\gen} - \dim (\Xbar^{\gen}\setminus V^{\Kirr}) \ge d[F:\Qp]$.
 \end{enumerate}
 \end{prop} 
 \begin{proof} 
 Since $V_{\PP_{\min}}$ is dense in $\Xbar^{\gen}$ by Proposition \ref{VPmin_dense}, we have $\Xbar^{\gen}= Z_{\PP_{\min}}= Z^{\Kred}\cup Z^{\Kirr}$, where 
  $Z^{\Kirr}$ is the closure of $V^{\Kirr}$.  Since 
 $\dim Z^{\Kred}< \dim \Xbar^{\gen}$ by Lemma \ref{special} and $\Xbar^{\gen}$ is equi-dimensional 
 we get that $\Xbar^{\gen}= Z^{\Kirr}$. Moreover, 
 $$ \Xbar^{\gen}\setminus V^{\Kirr}= Y\cup Z^{\Kred} \cup \bigcup_{\PP_{\min}< \PP} Z_{\PP}$$
 and claims (1) and (2)  follow from the dimension estimates in Lemmas \ref{bound_Y}, \ref{bound_diff}, \ref{special}. 
 If $\Dbar$ is absolutely irreducible then $\set{\PP : \PP_{\min} < \PP} = \emptyset$
 and claim (3) follows from Lemmas \ref{bound_Y} and \ref{special}.
 \end{proof} 
 
 We now want to transfer the density results from $\Xbar^{\gen}$ to $R^{\square}_{\rhobar}/\varpi$. 
 
 \begin{lem}\label{going_down} Let $A\rightarrow B$ be a flat ring homomorphism and let $U$ be an open subscheme
 of $\Spec A$ and let $V$ be the preimage of $U$ in $\Spec B$. If $U$ is dense in $\Spec A$ then 
 $V$ is dense in $\Spec B$.
 \end{lem} 
 \begin{proof} Let $\qq$ be a minimal prime of $B$ and let $\pp$ be its image in $\Spec A$. 
 Since the map is flat, it satisfies going down, and so $\pp$ is a minimal prime of $A$. Since
 $U$ is dense, it will contain $\pp$ and hence $V$ will contain $\qq$. Thus $V$ contains all the 
 minimal primes of $B$ and so is dense in $\Spec B$.
  \end{proof} 

  \begin{prop}\label{nspcl_formal} Let $(\Spec (R^{\square}_{\rhobar}/\varpi))^{\Kirr}$ be the preimage of $V^{\Kirr}$ in $\Spec (R^{\square}_{\rhobar}/\varpi)$. Then $(\Spec (R^{\square}_{\rhobar}/\varpi))^{\Kirr}$ is dense in 
  $\Spec (R^{\square}_{\rhobar}/\varpi)$.
  \end{prop}
  \begin{proof} The map $A^{\gen}/\varpi \rightarrow  R^{\square}_{\rhobar}/\varpi$ is flat, since it is a localization followed by a completion. The assertion follows from Lemma \ref{going_down} and Proposition 
  \ref{nspcl}.
  \end{proof}
  
  \begin{remar}\label{spell_out} Since $(\Spec (R^{\square}_{\rhobar}/\varpi))^{\Kirr}$ is also the preimage of $U^{\Kirr}$ in 
 $\Spec R^{\square}_{\rhobar}/\varpi$ we may characterise it as an open subscheme of 
 $\Spec R^{\square}_{\rhobar}/\varpi$, such that its closed points are in bijection with 
 $x\in P_1 (R^{\square}_{\rhobar}/\varpi)$, which map to $P_1 (R^{\ps}/\varpi)$ in $\Spec R^{\ps}$ and  for which the representation 
 $$\rho_{x}: G_F \rightarrow \GL_d(R^{\square}_{\rhobar}/\varpi)\rightarrow \GL_d(\kappa(x))$$
 remains  absolutely irreducible after restriction 
 to $G_{F'}$ for all degree $p$ Galois extensions 
 $F'$ of $F(\zeta_p)$.  Lemma \ref{adz0} implies that
 $H^2(G_F, \adz\rho_x)=0$ for such $x$.
\end{remar}

 We will now prove similar results for the generic fibres. For each partition $\mathcal P$ as in Section~\ref{sec_dim_sp} let $X^{\ps}_{\PP}$ be the scheme theoretic image of $X^{\ps}$ inside 
 $X^{\ps}_{\underline{\Sigma}}$ and let $X^{\gen}_{\PP}$ be the preimage of $X^{\ps}_{\PP}$ in $X^{\gen}$.
 We warn the reader that contrary to our usual notational conventions it is not clear that 
 $\Xbar^{\ps}_{\PP}$ considered in Section~\ref{sec_dim_sp} is the special fibre of $X^{\ps}_{\PP}$. 
 However, the following still holds.

 \begin{lem}\label{dim_generic} $\dim X^{\gen}_{\PP}[1/p] \le \dim \Xbar^{\gen}_{\PP}$.
 \end{lem} 
 \begin{proof} Let $\mathfrak a_{\PP}$ be the $R^{\ps}$-annihilator of $R^{\ps}_{\underline{\Sigma}}$, 
 and let $\mathfrak b_{\PP}$ be the $R^{\ps}$-annihilator of $R^{\ps}_{\underline{\Sigma}}/\varpi$.
 We may write 
$$X^{\gen}_{\PP}= \Spec A^{\gen}/ \mathfrak a_{\PP} A^{\gen}, \quad \Xbar^{\gen}_{\PP}=\Spec A^{\gen}/\mathfrak b_{\PP} A^{\gen}.$$ 
Since $R^{\ps}_{\underline{\Sigma}}$ is a finite $R^{\ps}$-module by Lemma \ref{iotaP_finite}, we have 
 $\sqrt{\mathfrak b_{\PP}}= \sqrt{(\mathfrak a_{\PP}, \varpi)}$. In particular, the special fibre of 
 $X^{\gen}_{\PP}$ has dimension equal to $\dim \Xbar^{\gen}_{\PP}$. 
 The assertion follows from Lemma \ref{dim_gen_sp}.
\end{proof} 
 
 \begin{prop}\label{Virr} Let 
 $$V^{\irr}:= X^{\gen}[1/p]\setminus \bigcup_{\PP_{\min}< \PP} X^{\gen}_{\PP}[1/p].$$ 
 Then $V^{\irr}$ is an open dense subset of $X^{\gen}[1/p]$. Moreover, the following hold:
 \begin{enumerate}
 \item if $d=2$ then $\dim X^{\gen}[1/p] - \dim (X^{\gen}[1/p]\setminus V^{\irr})\ge [F:\Qp]$; 
 \item if $d>2$ then $\dim X^{\gen}[1/p] - \dim (X^{\gen}[1/p]\setminus V^{\irr})\ge 1+[F:\Qp]$;
 \item if $d > 1$ is arbitrary but $\Dbar$ is absolutely irreducible (i.e.~$m = 1$) then $X^{\gen}[1/p]=V^{\irr}$.
 \end{enumerate}
 \end{prop}
 \begin{proof} It follows from Corollary \ref{ci_gen} that $\dim X^{\gen}[1/p]= d^2+d^2[F:\Qp]=\dim \Xbar^{\gen}$. 
 Lemmas \ref{dim_generic} and \ref{bound_diff} 
 together with \eqref{XgP} imply that for $\PP> \PP_{\min}$ we have 
 \begin{equation}\label{inequality}
 \dim X^{\gen}[1/p]- \dim X^{\gen}_{\PP}[1/p]\ge \dim \Xbar^{\gen} - \dim \Xbar^{\gen}_{\PP}.
 \end{equation}
 It follows from \eqref{XgP} that $\Xbar^{\gen}\setminus V^{\Kirr} = Y \cup Z^{\Kred} \cup \bigcup_{\PP_{\min}< \PP} 
 \Xbar^{\gen}_{\PP}$. Thus it follows from \eqref{inequality} and the definition of $V^{\irr}$ that 
 \begin{equation} 
 \dim X^{\gen}[1/p] -\dim (X^{\gen}[1/p]\setminus V^{\irr}) \ge \dim \Xbar^{\gen}- \dim ( \Xbar^{\gen}\setminus V^{\Kirr})
 \end{equation} 
 and the lower bounds for the codimension of 
 $X^{\gen}[1/p]\setminus V^{\irr}$ follow from Proposition \ref{nspcl}. 
 
 Thus the dimension of the closure of $V^{\irr}$ is equal to $\dim X^{\gen}[1/p]$. Since $A^{\gen}$ 
 is $\OO$-torsion free and equi-dimensional by 
 Corollary \ref{ci_gen} (1),  
 $X^{\gen}[1/p]$ is equi-dimensional, and so 
 $V^{\irr}$ is dense in $X^{\gen}[1/p]$.
 
 If $\Dbar$ is absolutely irreducible then $\rho_x$ is absolutely irreducible 
 for all closed points $x\in X^{\gen}[1/p]$ and 
 so $X^{\gen}[1/p]= V^{\irr}$. 
 \end{proof}

  \begin{cor}\label{irr_gen_dense} Let $(\Spec R^{\square}_{\rhobar}[1/p])^{\irr}$ be the preimage of $V^{\irr}$ in 
  $\Spec R^{\square}_{\rhobar}[1/p]$. Then $(\Spec R^{\square}_{\rhobar}[1/p])^{\irr}$ is dense in $\Spec R^{\square}_{\rhobar}[1/p]$.
  \end{cor} 
  \begin{proof} As explained in the proof of Proposition \ref{nspcl_formal} 
  the map $A^{\gen} \rightarrow R^{\square}_{\rhobar}$ is flat. Hence, the localization 
  $A^{\gen}[1/p]\rightarrow R^{\square}_{\rhobar}[1/p]$ is also flat. The assertion follows from 
  Lemma \ref{going_down} and Proposition \ref{Virr}.
  \end{proof}
  
  \begin{remar} Similarly to Remark \ref{spell_out} we may characterize 
  $(\Spec R^{\square}_{\rhobar}[1/p])^{\irr}$ as an open subscheme of $\Spec R^{\square}_{\rhobar}[1/p]$ such that 
  its closed points correspond to maximal ideals $\pp$ of $R^{\square}_{\rhobar}[1/p]$ for which the representation
  $$\rho_{\pp}: G_F\rightarrow \GL_d(R^{\square}_{\rhobar}[1/p])\rightarrow \GL_d(\kappa(\pp))$$
  is absolutely irreducible.
  \end{remar}
  
  \begin{cor}\label{cor_irrlift} The characteristic zero lift of $\rhobar$ in Corollary \ref{char0_lift} may be chosen to be absolutely irreducible.
  \end{cor} 
  \begin{proof} It follows from Corollary \ref{char0_lift} that $\Spec R^{\square}_{\rhobar}[1/p]$ is non-empty, 
  and Corollary \ref{irr_gen_dense} implies that $(\Spec R^{\square}_{\rhobar}[1/p])^{\irr}$ is non-empty. 
  A closed point in $(\Spec R^{\square}_{\rhobar}[1/p])^{\irr}$ gives the desired lift of $\rhobar$ to characteristic zero.
  \end{proof}
  
  \begin{cor}\label{irr_sub_dense} Let $\Sigma \subset \MaxSpec R^{\square}_{\rhobar}[1/p]$ be dense in $\Spec R^{\square}_{\rhobar}[1/p]$. Then 
  $$\Sigma^{\irr}:=\Sigma \cap  (\Spec R^{\square}_{\rhobar}[1/p])^{\irr}$$
  is also dense in $\Spec R^{\square}_{\rhobar}[1/p]$.
  \end{cor}
  \begin{proof} It follows from the proof of Proposition \ref{Virr} that $\Sigma \setminus \Sigma^{\irr}$ is 
  contained in a closed subset of $\Spec R^{\square}_{\rhobar}[1/p]$ of positive codimension. Since $\Spec R^{\square}_{\rhobar}[1/p]$ is equi-dimensional $\Sigma^{\irr}$ is dense. 
  \end{proof} 

\section{Irreducible components}\label{sec_irrcomp}

The aim of this section is to determine the irreducible components of $\Spec R_{\rhobar}^\square$ for any $\rhobar: G_F \to \GL_d(k)$ and study their geometry. 
It is instructive to consider first the one dimensional case.  Let $\bar{\psi}: G_F\rightarrow k^{\times}$ denote any continuous character and write $\psi^{\univ}: G_F\rightarrow \GL_1(R_{\bar{\psi}})$ for its universal deformation. Local class field theory gives a group homomorphism
    \[ \mu\rightarrow F^{\times} \overset{\Art_F}{\longrightarrow} G_F^{\mathrm{ab}}\overset{\psi^{\univ}}{\longrightarrow} \GL_1(R_{\bar{\psi}}), \]
where $\mu:=\mu_{p^{\infty}}(F)$ is the subgroup of $p$-power roots of unity in $F$. 
We note that $\mu$ is a finite cyclic $p$-group. The map induces a homomorphism of $\OO$-algebras $\OO[\mu]\rightarrow R_{\bar{\psi}}$, where
$\OO[\mu]$ is the group algebra of $\mu$ over $\OO$.
\begin{lem}\label{detmapmu}
    $R_{\bar{\psi}} \cong \OO[\mu]\br{y_1,\dots,y_{[F:\Qp]+1}}.$
\end{lem}
\begin{proof} It follows from local class field theory that the pro-$p$ 
completion of $G_F^{\ab}$ is isomorphic to $\mu_{p^{\infty}}(F)\times \Zp^{[F:\Qp]+1}$ 
and the assertion follows from \cite[Proposition 3.13]{Gouvea}.
\end{proof}

It follows immediately from Lemma \ref{detmapmu} that the set of irreducible components of
$\Spec R_{\bar{\psi}}$ is in bijection with the 
group of characters $\chi: \mu \rightarrow \OO^{\times}$ and the irreducible component 
corresponding to $\chi$ is given by $R_{\bar{\psi}}\otimes_{\OO[\mu], \chi} \OO$, which is formally smooth over $\OO$.

Let us return to the general case $\rhobar: G_F \to \GL_d(k)$. Mapping a deformation of $\rhobar$ to its determinant induces  a natural map $R_{\det\rhobar} \to R^{\square}_{\rhobar}$, 
which makes $R^{\square}_{\rhobar}$ into an 
$\OO[\mu]$-algebra by applying the above discussion to $\bar{\psi}=\det \rhobar$.
The algebra $\OO[\mu][1/p]$ is semi-simple
and its maximal ideals are in bijection with characters $\chi: \mu \rightarrow \OO^{\times}$. 
We thus have 
\begin{equation}\label{product}
R^{\square}_{\rhobar}[1/p] \cong \prod_{\chi: \mu \rightarrow \OO^{\times}} R^{\square, \chi}_{\rhobar}[1/p],
\end{equation}
where $R^{\square, \chi}_{\rhobar}:= R^{\square}_{\rhobar}\otimes_{\OO[\mu], \chi} \OO$. So our goal is to show that 
the rings $R^{\square, \chi}_{\rhobar}$ are $\OO$-torsion free integral 
domains, which we do by showing in Corollary \ref{chi_normal} that they are normal. Since we already know 
that $R^{\square}_{\rhobar}$ is $\OO$-torsion 
free by Corollary \ref{ci}, this implies 
that the map $R_{\det \rhobar}\rightarrow
R^{\square}_{\rhobar}$ induces a bijection 
between the sets of irreducible components, 
which answers affirmatively a question raised by GB--Juschka in \cite{bjpp}. Along the way we will also determine the irreducible components of $A^{\gen}$ and $R^{\ps}$. 

\begin{warning}
We emphasize that $R_{\rhobar}^{\square,\chi}$ is \textit{not} a ``fixed determinant deformation ring'' in the usual sense, but is rather constructed by fixing the value of the determinant only on the subgroup $\Art_F(\mu)\subset G_F^{\ab}$: the ring $R_{\rhobar}^{\square,\chi}$ represents the closed subfunctor $D_{\rhobar}^{\square, \chi}\subset D^{\square}_{\rhobar}$ given by 
$$ D_{\rhobar}^{\square, \chi}(A)=\{\rho_A\in D^{\square}_{\rhobar}(A): \det \rho_A(\Art_F(x))= \chi(x), \forall x\in \mu\}.$$
\end{warning}

\begin{prop}\label{Prop-RelCI}
There is an isomorphism
    \[ R_{\det\rhobar}\br{x_1,\dots,x_r}/(f_1,\ldots,f_t)\xrightarrow{\simeq} R^{\square}_{\rhobar} \]
where $r := \dim_k Z^1(G_F, \ad^0\rhobar)$ and $t := \dim_k H^2(G_F,\adz{\rhobar})$ such that the elements $f_1,\ldots,f_t$ form a regular sequence in $R_{\det\rhobar}\br{x_1,\dots,x_r}$. Moreover, 
$$r-t= (d^2-1)([F:\Qp]+1).$$  
\end{prop}

\begin{proof} 
This argument is a modification of Kisin's method of presenting global deformation rings over local ones in \cite[Section 4]{KisinCurrentDevelopments}.
Kisin's argument is an important refinement of Mazur's obstruction theory in \cite[Section 1.6]{Mazur_GQ}.

The exact sequence $0\rightarrow \ad^0\rhobar \rightarrow \ad \rhobar \overset{\tr}{\rightarrow} k \rightarrow 0$ of Galois representations induces 
an exact sequence of abelian groups:  
$$0 \rightarrow Z^1(G_F, \ad^0\rhobar)\rightarrow Z^1(G_F, \ad\rhobar) \overset{Z^1(\tr)}{\longrightarrow} Z^1(G_F, k)$$
and hence $r=\dim_k \ker (Z^1(\tr))$. 
The map $Z^1(\tr): Z^1(G_F, \ad\rhobar) \to Z^1(G_F, k)$ is the induced map on Zariski tangent spaces of the map of deformation rings $R_{\det\rhobar} \to R^{\square}_{\rhobar}$, and thus lifts to a surjection
    \[ \wt\phi: \wt R := R_{\det\rhobar}\br{x_1,\ldots,x_r} \twoheadrightarrow R^{\square}_{\rhobar}. \]
We set $J := \ker \wt\phi$.  By Nakayama's lemma, we need to show that $\dim_k J/\wt\mm J \le t$.

The module $J/\wt\mm J$ appears as the kernel in the sequence
\begin{equation}\label{sesforlift}
    0 \to J/\wt\mm J \to \wt R/\wt\mm J \to \wt R/J \cong R^{\square}_{\rhobar} \to 0.
\end{equation}
In view of the above sequence, we shall construct a homomorphism
    \[ \alpha: \Hom_k(J/\wt\mm J,k) \to \ker (H^2(\tr)\colon H^2(G_F,\ad \rhobar) \to H^2(G_F,k)) \]
and show that the kernel of $\alpha$ injects into $\coker(H^1(\tr))$. This will imply the existence of the presentation in the statement of the Proposition, since then
\begin{equation}\label{bound_Jm}
 \dim_k J/\wt\mm J\le  \dim_k \ker (H^2(\tr) )+ \dim_k \coker (H^1(\tr) ) = \dim_k H^2(G_F,\adz \rhobar), 
\end{equation}
where the last equality comes from the long exact cohomology sequence that arises from $0 \to \adz \rhobar \to \ad \rhobar \to k \to 0$.

Fix $u \in \Hom_k(J/\wt\mm J, k)$. The pushout under $u$ of the sequence (\ref{sesforlift}) yields 
\begin{equation*}
    0\to I_u\to R_u\stackrel{\phi_u}\to  R^{\square}_{\rhobar}\to 0,
\end{equation*}
where $I_u = k$. The surjection of profinite groups $\GL_d(R_u)\twoheadrightarrow \GL_d(R^{\square}_{\rhobar})$ has a continuous section by \cite[Proposition 2.2.2]{RZ} (which is not necessarily a group homomorphism). Thus there is a continuous function $\wt\rho_u: G_F\rightarrow \GL_d(R_u)$ such that the diagram of sets
\[
    \xymatrix@C+2em{ G_F \ar[r]^-{\wt\rho_u} \ar[dr]_-{\rho^\square}&\GL_d(R_u)\ar@{>>}[d]^{\GL_d(\phi_u)}\\&\GL_d(R^{\square}_{\rhobar}) \rlap{.}}
\]
commutes. The kernel $1+M_d(I_u)$ of $\GL_d(\phi_u)$ can be identified with $\ad \rhobar\otimes_k I_u$, and so the set-theoretic lift yields a continuous $2$-cocycle
    \[ c_u \in Z^2(G_F,\ad \rhobar) \otimes_k I_u \]
given by
$1 + c_u(g_1,g_2) = \wt \rho_u(g_1g_2)\wt \rho_u(g_2)^{-1}\wt \rho_u(g_1)^{-1}$. The class 
$$[c_u] \in H^2(G_F,\ad \rhobar)\otimes_k I_u$$ is independent of the chosen lifting since any other lift $\wt\rho_u'$ gives rise to a class $c_u' \in Z^2(G_F, \ad\rhobar) \otimes_k I_u$ which differs from $c_u$ by a coboundary in $B^2(G_F, \ad\rhobar) \otimes_k I_u$, so the representation $\rho^\square$ can be lifted to a homomorphism $G_F \to \GL_d(R_u)$ if and only if $[c_u]=0$. The existence of the homomorphisms $R_{\det\rhobar}\to R_u \to R^{\square}_{\rhobar}$ together with the universality of $R_{\det\rhobar}$ imply that the image of $[c_u]$ in $H^2(G_F, k)$ is zero. We define $\alpha$ as the homomorphism $u\mapsto [c_u]$. 

To analyze the kernel of $\alpha$, let $u$ be such that $[c_u]=0$, so that $\rho^\square$ can be lifted to $R_u$. By the universality of $R^{\square}_{\rhobar}$ we obtain a splitting $s_u$ of $\phi_u$. One deduces that the map from $I_u$ to the kernel of the surjective map
    \[ t_u: \mm_{R_u}/(\mm_{R_u}^2 + \varpi R_u)\to \mm^\square/((\mm^\square)^2 + \varpi R^{\square}_{\rhobar}) \]
of mod $\varpi$ cotangent spaces is an isomorphism.

The map $t_u$ in turn is induced from the homomorphism $\wt R/\wt\mm J \to R^{\square}_{\rhobar}$ by pushout and from the analogous surjection
    \[ \wt t : \wt\mm/(\wt \mm^2+\varpi \wt R)\to \mm^\square/((\mm^\square)^2+\varpi R^{\square}_{\rhobar}). \]
Via our identification $I_u\cong \ker t_u$, the pushout along $u$ induces a surjective homomorphism $\gamma_u: \ker(\wt t)\to I_u \cong k$ of $k$-vector spaces. One easily verifies that $u \mapsto \gamma_u$ induces an injective $k$-linear map
    \[ \ker(\alpha) \hookrightarrow \Hom_k(\ker(\wt t),k) \]
Upon identifying $\ker(\wt t)^*$ with $\coker(H^1(\tr))$, the proof of the bound \eqref{bound_Jm} is complete.

It remains to show that $f_1,\ldots,f_t$ is a regular sequence. We may write $\OO[\mu]=\OO\br{z}/((1+z)^{m}-1)$, where $m$ is the order of $\mu$. By Lemma \ref{detmapmu}, we get a presentation
    \[ \frac{\OO\br{z,y_1,\ldots,y_{[F:\Q_p]+1},x_1,\ldots,x_r}}{((1+z)^{m}-1,f_1,\ldots,f_t)} \xrightarrow{\cong} R^{\square}_{\rhobar}. \]
By the same argument as in \eqref{r_minus_s} the Euler--Poincar\'e characteristic formula implies that $r-t= \dim_k(\ad^0\rhobar)  ([F:\Qp]+1)= 
(d^2-1)([F:\Qp]+1)$ and thus it follows from Corollary \ref{ci} that 
\begin{equation}\label{dimRrt}
\dim R^{\square}_{\rhobar}= [F:\Qp]+2+r-t.
\end{equation}
This  implies that $(1+z)^{m}-1,f_1,\ldots,f_t$ can 
be extended to a system  of parameters in a regular ring
$S:=\OO\br{z,y_1,\ldots,y_{[F:\Q_p]+1},x_1,\ldots,x_r}$. 
Thus $(1+z)^{m}-1,f_1,\ldots,f_t$ is a regular sequence
in $S$ 
and so $f_1,\ldots,f_t$ is a regular sequence in 
$R_{\det\rhobar}\br{x_1,\dots,x_r}=S/((1+z)^m-1)$.
\end{proof}

\begin{remar}\label{present_kappa} The Proposition also holds for continuous representations $\rhobar: G_F \rightarrow \GL_d(\kappa)$, where $\kappa$ is a local field, with essentially the same proof. 
The only difference is that one has to work harder to show the
existence of the continuous section $\tilde{\rho}_u$, as the groups 
$\GL_d(R_u)$ and $\GL_d(R^{\square}_{\rhobar})$ are not profinite anymore. The existence of such a section is well explained in \cite[Lecture 6]{mod_lift}.
\end{remar} 
\begin{cor}\label{ci_chi} For each character $\chi: \mu_{p^{\infty}}(F) \to \OO^\times$ and each closed point $x\in X^{\gen}$ above $\mm_{R^{\ps}}$ the following hold:
\begin{enumerate}
\item $R^{\square, \chi}_{\rho_x}$ is $\OO$-torsion free  of dimension
$1+d^2+ d^2[F:\Qp]$ and is complete intersection;
\item $R^{\square,\chi}_{\rho_x}/\varpi$ is complete intersection of dimension $d^2+ d^2[F:\Qp]$.
\end{enumerate}
\end{cor}
\begin{proof} Without loss of generality we may assume that the residue field of $x$ is equal to $k$. 
Proposition \ref{Prop-RelCI} gives the presentation
    \[ R_{\det\rho_x}^\chi\br{x_1,\dots,x_r}/(f_1,\dots,f_t)
    \xrightarrow{\cong} R^{\square,\chi}_{\rho_x}, \]
where $R_{\det\rho_x}^\chi:=
R_{\det\rho_x}\otimes_{\OO[\mu],\chi} \OO$. Since 
 $R_{\det\rho_x}^\chi$ is formally smooth over $\OO$ of dimension $[F:\Qp]+2$
 by Lemma \ref{detmapmu}, it is enough to show that 
 $$\dim R^{\square,\chi}_{\rho_x}/\varpi \le [F:\Qp]+1+ r-t.$$ 
 Then the same argument as in 
 the proof of Proposition \ref{Prop-RelCI}
 shows that the sequence 
 $\varpi, f_1, \ldots, f_t$ is regular
 in $R_{\det\rho_x}^\chi\br{x_1,\dots,x_r}$. Since
 $R^{\square, \chi}_{\rho_x}$ is a quotient of $R^{\square}_{\rho_x}$
 and $R^{\square}_{\rho_x}$ is $\OO$-torsion free by Corollary \ref{ci}, 
 we have $\dim  R^{\square, \chi}_{\rho_x}/\varpi \le \dim R^{\square}_{\rho_x}/\varpi=\dim R^{\square}_{\rho_x}-1$ 
 and the desired inequality follows from \eqref{dimRrt}.
 \end{proof}

The restriction of a pseudo-character $D: A[G_F] \rightarrow A$ to $G$ defines
a continuous group homomorphism $\det D: G_F \rightarrow A^{\times}$, see \cite[Definition 4.1.5]{bj}. Moreover, if $D$ is associated to a representation $\rho: G_F \rightarrow \GL_d(A)$ then $\det D=\det \rho$. This induces a map 
of deformation rings $R_{\det \Dbar}\rightarrow R^{\ps}$ and makes $R^{\ps}$ into an $\OO[\mu]$-algebra. 

Since $A^{\gen}$ is an $R^{\ps}$-algebra, we may define $$A^{\gen,\chi} := A^{\gen} \otimes_{\OO[\mu],\chi} \OO, \quad  X^{\gen,\chi}:=\Spec A^{\gen, \chi}$$ and we let $\Xbar^{\gen,\chi}$ denote its special fibre. Note that since a character of $G_F^{\ab}$ valued in a characteristic $p$ field is trivial after pulling back to $\mu_{p^\infty}(F)$, we have that $\Xbar^{\gen,\chi} = \Xbar^{\gen, \Eins}$ for all $\chi$, where $\Eins$ is the trivial character. Moreover, the reduced subschemes 
of $\Xbar^{\gen}$ and $\Xbar^{\gen, \chi}$ coincide and so
$$\dim \Xbar^{\gen, \chi}=\dim \Xbar^{\gen} =d^2+d^2[F:\Qp],$$
where the last equality is given by Corollary \ref{ci_gen}.

\begin{cor}\label{ci_gen_chi} For each character $\chi: \mu_{p^{\infty}}(F) \to \OO^\times$ the following hold:
 \begin{enumerate}
 \item $A^{\gen, \chi}$ is $\OO$-torsion free, equi-dimensional of dimension $1+ d^2+ d^2[F:\Qp]$ and is locally complete intersection;
 \item $A^{\gen, \chi}/\varpi$ is  equi-dimensional of dimension $d^2+d^2[F:\Qp]$ and is locally complete intersection. 
 \end{enumerate}
 \end{cor}
 \begin{proof} We claim that the local 
 rings at closed points of $X^{\gen, \chi}$ 
 are $\OO$-torsion free and complete intersection. Given the claim the proof is the same as in Corollary \ref{ci_gen}. 
 
 We will prove the claim using the strategy outlined in Remark \ref{strategy}. We already know from Corollary \ref{ci_chi} that $R^{\square, \chi}_{\rho_x}$ is $\OO$-torsion free and complete intersection of 
 dimension $d^2+d^2[F:\Qp]+1$ whenever
 $x\in X^{\gen, \chi}$ is a closed point with $\kappa(x)/k$ a finite extension. By applying $\otimes_{\OO[\mu], \chi} \OO$ we obtain the
 $\chi$-versions of Propositions \ref{main_complete}
 and \ref{331} and Corollary \ref{the_same_comp}. 
 
 Let $x$ be a closed point of $X^{\gen, \chi}$. If $\kappa(x)$ is a finite extension of $k$ then $\widehat{\OO}_{X^{\gen,\chi},x}\cong 
 R^{\square, \chi}_{\rho_x}$ by Proposition \ref{main_complete} and hence $\OO_{X^{\gen,\chi},x}$ is complete intersection. Otherwise, let $x'$ and $z$ be as in Corollary \ref{the_same_comp}. In particular, 
 $z$ is a closed point of $X^{\gen, \chi}$ and 
 $\kappa(z)$ is a finite extension of $k$. It
 follows from the argument explained in Remark 
 \ref{strategy} that 
 if $\kappa(x)$ is a local field 
of characteristic $p$ then 
$$ \widehat{\OO}_{X^{\gen,\chi}, x}\br{T}\cong R^{\square,\chi}_{\rho_x}\cong R^{\square,\chi}_{\rho_{x'}}\cong \widehat{(R^{\square,\chi}_{\rho_z})}_{x'}\br{T},$$
and if $\kappa(x)$ is a finite extension of $L$
then 
$$ \widehat{\OO}_{X^{\gen,\chi}, x}\cong R^{\square,\chi}_{\rho_x}\cong R^{\square,\chi}_{\rho_{x'}}\cong \widehat{(R^{\square,\chi}_{\rho_z})}_{x'}.$$
Since $R^{\square,\chi}_{\rho_z}$ is complete 
intersection, it follows from \cite[\href{https://stacks.math.columbia.edu/tag/09Q4}{Tag 09Q4}]{stacks-project} that the
local ring $(R^{\square,\chi}_{\rho_z})_{x'}$
(and hence its completion) is also complete intersection. The isomorphisms
above imply that 
$\widehat{\OO}_{X^{\gen,\chi}, x}$ is complete 
intersection. Hence, $\OO_{X^{\gen,\chi}, x}$
is complete intersection, see \cite[\href{https://stacks.math.columbia.edu/tag/09Q3}{Tag 09Q3}]{stacks-project}.
 \end{proof}
 
 \begin{remar} Alternatively, one could 
 first prove a version of Proposition \ref{Prop-RelCI} for deformation rings  of 
$\rhobar: G_F\rightarrow \GL_d(\kappa(x))$ 
to Artinian $\Lambda$-algebra as in Section~\ref{sec_completions} for any closed point of $x\in X^{\gen}$, by changing $\OO$ to $\Lambda$
and $k$ to $\kappa(x)$ everywhere. The Euler--Poincar\'e characteristic formula still holds in this setting,  see \cite[Theorem 3.4.1(c)]{bj}. Then deduce Corollary \ref{ci_chi} in this more general setting using the same proof and 
then obtain Corollary \ref{ci_gen_chi} by repeating verbatim the proof of Corollary 
\ref{ci_gen}. 
\end{remar}

In the Lemmas below, $\kappa$ is either a finite extension of $k$, a finite extension of $L$ or a local field of characteristic $p$ containing $k$. The ring $\Lambda$ is 
defined exactly as in the beginning of Section \ref{sec_completions}. If 
$\cha(\kappa)=0$ then $\Lambda=\kappa$, if 
$\cha(\kappa)=p$ then $\Lambda$ is an $\OO$-algebra, which is a complete DVR with 
uniformiser $\varpi$ and residue field $\kappa$. As in Section~\ref{sec_completions} 
we consider deformation problems of $\rho: G_F\rightarrow \GL_d(\kappa)$ to local Artinian $\Lambda$-algebras with residue field $\kappa$. 

\begin{lem}\label{formal_smooth} Let $\rho: G_F \to \GL_d(\kappa)$ be a continuous representation, where $\kappa$ is either a finite extension of $k$, a local field of characteristic $p$, or a finite extension of $L$. If $H^2(G_F, \adz\rho)=0$ then for all characters $\chi: \mu_{p^{\infty}}(F)\rightarrow \OO^{\times}$ the ring $R^{\square, \chi}_{\rho}$ is formally smooth over $\Lambda$.
\end{lem}
\begin{proof} 
It follows from the proof of \cite[Lemma 3.4.2]{bj}, where an analogous statement 
is proved for the deformation functors without the framing and for Artinian $\kappa$-algebras, that the map 
$$R_{\det \rho}\rightarrow R^{\square}_{\rho},$$ 
induced by sending a deformation of $\rho$ to an Artinian $\Lambda$-algebra to its determinant, 
is formally smooth. By applying $\otimes_{\OO[\mu], \chi} \OO$ we deduce that the map 
$$R_{\det \rho}^{\chi}\rightarrow R^{\square, \chi}_{\rho}$$
is formally smooth. 
 
Since the 
group 
$G_F^{\ab}/\Art_F(\mu_{p^{\infty}}(F))$ is $p$-torsion free, the ring 
$R_{\det \rho}^{\chi}$ is formally smooth 
over $\Lambda$. Hence, $R^{\square, \chi}_{\rho}$ is formally smooth over $\Lambda$.
(Alternatively, one could prove Proposition \ref{Prop-RelCI} for $\rho$, 
see Remark \ref{present_kappa}, and then obtain the Lemma as a Corollary.)
\end{proof}
 
Recall that in Section~\ref{sec_dens_irr} we have defined an open subscheme $U^{\nspcl}$ of $\Xbar^{\ps}\setminus \{\mm_{R^{\ps}}\}$ and
defined $V^{\nspcl}$ to be a preimage of $U^{\nspcl}$ 
in $\Xbar^{\gen}$. We will refer to $V^{\nspcl}$ as the 
\textit{absolutely irreducible non-special locus}.

\begin{prop}\label{Vnspcl_regular} For each character 
$\chi: \mu_{p^{\infty}}(F)
\rightarrow \OO^\times$ the 
absolutely irreducible non-special locus in $\Xbar^{\gen, \chi}$ 
is regular. 
\end{prop}
\begin{proof} 
It is enough to show that localization 
of $A^{\gen, \chi}/\varpi$ at $x$ is a regular ring for every closed point $x$
in $V^{\nspcl}\cap \Xbar^{\gen, \chi}$.
It follows from Lemma \ref{333} applied with $R=R^{\ps, \chi}/\varpi$ and $A=A^{\gen, \chi}/\varpi$
that it is enough to show that 
the completion of 
$\kappa(x)\otimes_{\OO} A^{\gen, \chi}$ 
at the kernel of the 
map of $\kappa(x)$-algebras 
$\kappa(x)\otimes_{\OO} A^{\gen, \chi}\rightarrow \kappa(x)$ is regular. 
Proposition \ref{main_complete} implies that 
we may identify this ring with 
deformation ring 
$R^{\square, \chi}_{\rho_x}/\varpi$. 
If $\zeta_p\in F$ then since $x$ is non-special $H^2(G_F, \adz\rho_x)=0$, see \cite[Lemma 5.1.1]{bj}, Lemma \ref{formal_smooth} 
implies that $R^{\square, \chi}_{\rho_x}/\varpi$
is formally smooth over $\kappa(x)$. If $\zeta_p\not\in F$ 
then $\mu$ is trivial, so that $R^{\square, \chi}_{\rho_x}= R^{\square}_{\rho_x}$, and $H^2(G_F, \ad\rho_x)=0$, see \cite[Lemma 5.1.1]{bj}. It follows from \eqref{present_square_x}  that $R^{\square}_{\rho_x}/\varpi$ is formally smooth over  $\kappa(x)$. 
\end{proof}

\begin{prop}\label{Virr_reg} For each character 
$\chi: \mu_{p^{\infty}}(F)
\rightarrow \OO^\times$ the 
absolutely irreducible locus in $X^{\gen, \chi}[1/p]$ 
is regular.
\end{prop}
\begin{proof} 
Let $x$ be a closed point in $X^{\gen, \chi}[1/p]$ and 
let $\rho_x: G_F\rightarrow \GL_d(\kappa(x))$ be
the corresponding Galois representation. We claim  that if $\rho_x$ is absolutely irreducible then $H^2(G_F, \adz \rho_x)=0$. Since $\kappa(x)$ is a finite extension of 
$L$, $\adz \rho_x$ is a direct summand of $\ad \rho_x$, 
and thus it is enough to show that $H^2(G_F, \ad \rho_x)=0$. 
By local Tate duality, it is enough to show that
$H^0(G_F, \ad \rho_x(1))=0$. Since $\rho_x$ is absolutely 
irreducible, non-vanishing of this group is equivalent to 
$\rho_x\cong \rho_x(1)$. By considering determinants we 
would obtain that the $d$-th power of the cyclotomic
character is trivial, yielding a contradiction.

Given the claim the rest of the proof is the same as the
proof of Proposition \ref{Vnspcl_regular}, since  Lemma \ref{334} implies that $\widehat{\OO}_{X^{\gen,\chi}, x}\cong R^{\square, \chi}_{\rho_x}$. 
\end{proof} 

\begin{lem}\label{pnot2_smooth} Assume that  $F=\Qp$ and $d=2$. Let $\kappa$ be either a finite extension of 
$L$ or a finite or local field of characteristic $p$. If $\cha(\kappa)=p$ then we further assume that $p>2$. Let $\rho: G_{\Qp}\rightarrow \GL_2(\kappa)$ be a
continuous representation with semi-simplification isomorphic to a direct sum of characters $\psi_1 \oplus \psi_2$ satisfying $\psi_1\neq \psi_2(1)$ and $\psi_2\neq \psi_1(1)$. Then 
$$H^2(G_{\Qp}, \ad \rho)=H^2(G_F, \adz\rho)=0.$$
In particular,
$R^{\square, \chi}_{\rho}$ is formally smooth over $\Lambda$.
\end{lem}
\begin{proof} Since $\cha(\kappa)\neq 2$, $\adz\rho$ is a direct summand $\ad \rho$, and thus it is enough 
to show that 
$H^2(G_{\Qp}, \ad \rho)=0$. By local 
Tate duality, see \cite[Theorem 3.4.1(b)]{bj}, it is enough to show that
$H^0(G_{\Qp}, \ad \rho(1))=0$. Non-vanishing of this group would 
imply that $\psi_i\psi_j^{-1}(1)$ 
is a trivial character for some
$i, j\in \{1,2\}$. If $i=j$ then this would imply 
$\chi_{\cyc}\otimes_{\Zp} \kappa$ is 
 trivial, which is not the case as $\cha(\kappa)\neq 2$. If $i\neq j$ then this 
does not hold by assumption.

The last assertion follows from Lemma \ref{formal_smooth}.
\end{proof}

\begin{lem}\label{peq2_smooth} Assume that $p=2$, $F=\QQ_2$ and $d=2$. Let $\kappa$ be a
finite or
local field of characteristic $2$ and 
let $\rho: G_{\QQ_2}\rightarrow \GL_2(\kappa)$ be a continuous representation, which is a non-split extension of distinct characters.

Then $H^2(G_{\QQ_2}, \adz \rho)=0$. In particular
$R^{\square, \chi}_{\rho}$ is formally smooth over $\Lambda$.
\end{lem}
\begin{proof} After twisting we may assume that we can choose a basis of the underlying vector space of $\rho$, such that with respect to that basis 
$$\rho(g)=\begin{pmatrix} 1 & b(g)\\ 0 & \psi(g)\end{pmatrix}, \quad \forall g\in G_{\QQ_2},$$
where $\psi: G_{\QQ_2}\rightarrow \kappa^{\times}$ is a non-trivial character. 
We use the same basis to identify $\ad \rho$ with $M_2(k)$ with the $G_{\QQ_2}$-action given by 
$$g\cdot M:= \rho(g) M \rho(g)^{-1}.$$
For $i,j\in \{1,2\}$ let $e_{ij}\in M_2(k)$ be the matrix with the $ij$-entry equal to $1$ and all the other entries equal to zero. Let 
$\adbar\rho$ be the quotient 
$\ad \rho$ by the scalar matrices and 
let $\overline{e}_{ij}$ be the image 
of $e_{ij}$ in $\adbar \rho$. A
direct computation shows that 
$$g\cdot \overline{e}_{12}= \psi(g)^{-1} \overline{e}_{12}, \quad 
g\cdot \overline{e}_{11} =
\overline{e}_{11}- 
\psi(g)^{-1}b(g) \overline{e}_{12},
\quad g\cdot \overline{e}_{21}= 
\psi(g) \overline{e}_{21}-\psi(g)^{-1} b(g)^2 \overline{e}_{12}.$$ 
Since $\rho$ is non-split, $b(g)\neq 0$
for some $g\in G_{\QQ_2}$. Thus
$\kappa \overline{e}_{12}$ is the 
unique irreducible subrepresentation 
of $\adbar \rho$. 
Since $G_{\QQ_2}$ acts on 
$\overline{e}_{12}$ 
by a non-trivial character, we deduce
that $H^0(G_{\QQ_2}, \adbar \rho)=0$. 

It follows from local Tate duality, 
see \cite[Theorem 3.4.1(b)]{bj}, that 
$H^2(G_{\QQ_2}, \adz \rho)=0$. Note 
that the cyclotomic character is
trivial modulo $2$.

The last assertion follows from 
Lemma \ref{formal_smooth}.
\end{proof}

\begin{prop}\label{nice_open} 
There is an open subscheme 
$V^{0,\chi} \subset \Xbar^{\gen, \chi}$ such that 
\begin{enumerate}
\item $H^2(G_F, \adz \rho_x)=0$ for all 
closed points $x\in V^{0, \chi}$;
\item $\dim \Xbar^{\gen, \chi} - \dim (\Xbar^{\gen, \chi}\setminus V^{0, \chi})\ge 2$. 
\end{enumerate}
In particular, $\Xbar^{\gen, \chi}$ is regular in 
codimension $1$. 
\end{prop}
\begin{proof} We first note that if $V\subset \Xbar^{\gen, \chi}$ is open and satisfies part (1) 
then $V$ is regular by the argument explained in the proof of Proposition \ref{Vnspcl_regular}. 
Thus if (1) and (2) hold then $\Xbar^{\gen,\chi}$ is regular in codimension $1$. 
We also note that Lemma \ref{adz0} implies that 
part (1) holds for $V^{\Kirr, \chi}:=V^{\Kirr}\cap 
\Xbar^{\gen, \chi}$.  We consider three separate cases. 

\textbf{Case 1: $d > 2$ or $F \neq \Qp$ or $\Dbar$ is (absolutely) irreducible.} These three conditions correspond to parts (1), (2), and (3) of Proposition \ref{nspcl} respectively, and indeed Proposition \ref{nspcl} implies that the complement of $V^{\Kirr, \chi}$ in $\Xbar^{\gen,\chi}$ has dimension at most $\dim \Xbar^{\gen,\chi} - 2$.
Hence, we may take $V^{0, \chi}=V^{\Kirr, \chi}$. 

\textbf{Case 2: $d = 2$ and $F = \Qp$ and $p>2$ and $\Dbar$ is reducible.} In this case, $\mu = \set{1}$ so $\chi = \Eins$ and thus $\Xbar^{\gen,\Eins} = \Xbar^{\gen}$. 
It follows from  Proposition \ref{bound_again}, Lemma \ref{bound_Y}, Lemma \ref{special} that 
$$V^{0, \chi}:= \Xbar^{\gen}\setminus( Y \cup Z_{\PP_{\max}}^{12} \cup Z_{\PP_{\max}}^{21} \cup Z^{\Kred})$$
satisfies part (2). We may also write $V^{0, \chi}=
V^{\Kirr}\cup V_{\PP_{\max}}'$, where we use the notation introduced in the proof of Proposition \ref{bound_delta}. Since part (1) holds for $V^{\Kirr}$ it is enough to consider closed points 
$x\in V_{\PP_{\max}}'$. The definition of $V_{\PP_{\max}}'$ implies firstly that $\rho_x$ is reducible and secondly that if we let $\psi_1$ and $\psi_2$ denote its irreducible Jordan-H\"older constituents then $\psi_1 \neq \psi_2(1)$ and $\psi_2 \neq \psi_1(1)$. Therefore,
$H^2(G_F, \adz \rho_x)=0$ by
Lemma \ref{pnot2_smooth}.

\textbf{Case 3: $d = 2$ and $F = \Q_2$ and $\Dbar$ is reducible.}
The proof is the same as in Case 2, using 
Lemma \ref{peq2_smooth} instead of Lemma \ref{pnot2_smooth}. However, one additionally 
has to remove the reducible semi-simple locus 
in $\Xbar^{\gen,\chi}$. Its dimension 
is at most $4+2=6$ by Corollary \ref{ss_bound} and the dimension 
of $\Xbar^{\gen,\chi}$ is $8$. Thus 
the codimension is at least $2$.
\end{proof}

\begin{prop}\label{Xbar_normal}
$\Xbar^{\gen,\chi}$ is normal.
\end{prop}
\begin{proof}
Since $\Xbar^{\gen,\chi}$ is a local complete intersection
by Corollary \ref{ci_gen_chi}, it is Cohen--Macaulay and satisfies property (S2), 
and Proposition \ref{nice_open} says that it satisfies property (R1). Hence, $\Xbar^{\gen,\chi}$
is normal by Serre's criterion for normality. 
\end{proof}

\begin{cor}\label{sp_normal} For each character $\chi: \mu_{p^{\infty}}(F)\rightarrow \OO^{\times}$ and $\rhobar: G_F\rightarrow \GL_d(k)$ the ring 
$R^{\square,\chi}_{\rhobar}/\varpi$ is a normal integral domain.
\end{cor}
\begin{proof}
Since $\Xbar^{\gen,\chi}$ is normal and excellent 
the completions of its local rings are normal by
\cite[Theorem 32.2 (i)]{matsumura}.
So after formally completing along the maximal ideal corresponding to $\rhobar$, Proposition \ref{main_complete} (after applying $\otimes_{\OO[\mu],\chi} \OO$) and Proposition \ref{334} tell us that  $R^{\square,\chi}_{\rhobar}/\varpi$ is normal, and thus an integral domain since it is a local ring. 
\end{proof}

\begin{lem}\label{Kummer_fix} Let $\hat{Y}$ be the preimage of $\mm_{R^{\ps}}$ in 
$\Spec R^{\square, \chi}_{\rhobar}/\varpi$. 
Let $W$ be a closed subscheme of $\Spec R^{\square, \chi}_{\rhobar}/\varpi$ such that 
$H^2(G_F, \adz \rho_x)\neq 0$ for all closed points 
$x\in W\setminus \hat{Y}$. Then 
$\dim R^{\square, \chi}_{\rhobar}/\varpi - \dim W\ge 2$. 
\end{lem}
\begin{proof} The assumptions imply that 
$W$ is contained in $\hat{Z}\cup \hat{Y}$, where $Z= \Xbar^{\gen, \chi}\setminus V^{0, \chi}$ and $\hat{Z}$ is a
formal completion of 
$Z$ at the point corresponding to $\rhobar$. In terms of commutative algebra the ring of functions 
of $\hat{Z}$ corresponds to the completion of the ring 
of functions of $Z$ with respect to the maximal ideal corresponding to $\rhobar$. 
Hence, $\dim \hat{Z}\le \dim Z$, and Proposition \ref{nice_open} implies that $\hat{Z}$ has codimension at least $2$ in $\Spec R^{\square, \chi}_{\rhobar}/\varpi$. Similarly 
$\hat{Y}$ is a formal completion of  $Y$ (the preimage
of $\{\mm_{R^{\ps}}\}$ in $X^{\gen}$)
at the point corresponding to $\rhobar$, and 
using Lemma \ref{bound_Y} we conclude that $\hat{Y}$ also has 
codimension of at least $2$ in 
$\Spec R^{\square, \chi}_{\rhobar}/\varpi$.
\end{proof}

\begin{prop}\label{X1p_normal} $X^{\gen, \chi}[1/p]$ is normal. 
\end{prop} 
\begin{proof} The proof is essentially the same as 
the proof of Proposition \ref{Xbar_normal}. It follows 
from Corollary \ref{ci_gen_chi} that 
$X^{\gen, \chi}[1/p]$ is Cohen--Macaulay and 
we have to check that the codimension of the 
singular locus is at least $2$. Since 
$X^{\gen, \chi}[1/p]$ is a preimage of $\Spec R^{\ps, \chi}[1/p]$ in $X^{\gen}$, Lemma \ref{jacobson} implies 
that $X^{\gen, \chi}[1/p]$ is Jacobson and we may argue 
with closed points. 

We have already shown in Proposition \ref{Virr_reg} that 
the absolutely irreducible locus $V^{\irr, \chi}$ in 
$X^{\gen, \chi}[1/p]$ is regular. Thus the singular 
locus is contained in $$\bigcup_{\PP_{\min}< \PP} 
X^{\gen, \chi}_{\PP}[1/p],$$ where $X^{\gen, \chi}_{\PP}:= X^{\gen, \chi}\cap X^{\gen}_{\PP}$. 

If either $\rhobar$ is 
absolutely irreducible, $F\neq \Qp$ or $d>2$ then it follows from Proposition \ref{Virr} that  $X^{\gen, \chi}[1/p]$ is regular in codimension $1$. 

If $\rhobar$ is reducible, $F=\Qp$ and $d=2$ then there 
are two partitions $\PP_{\min}$ and $\PP_{\max}$ and
$\dim X^{\gen, \chi}[1/p] -\dim X^{\gen, \chi}_{\PP_{\max}}[1/p]=1$, so the previous argument does not work. If $x\in X^{\gen, \chi}[1/p]$ is a closed singular point then 
it follows from Proposition \ref{Virr_reg} and Lemma \ref{pnot2_smooth} that 
$\rho_x$ is reducible and its semi-simplification has the form $\psi\oplus \psi(1)$ for some character $\psi: G_F\rightarrow \kappa(x)^{\times}$. Thus we may assume that 
the pseudo-character $\Dbar$ associated to 
$\rhobar$ is equal to $\bar{\psi} + \bar{\psi}(1)$.
We will now recall a construction, carried out after Lemma \ref{iotaP_finite}, in this special case. 
Let $R_{\bar{\psi}}$ be the universal deformation ring of $\bar{\psi}$. Mapping  a deformation $\psi_A$ to the pseudo-character $\psi_A+\psi_A(1)$ induces a map of local $\OO$-algebras $R^{\ps}\rightarrow R_{\bar{\psi}}$.
Let  $X^{\ps,12}_{\PP_{\max}}$ be the 
schematic image of $\Spec R_{\bar{\psi}}\rightarrow X^{\ps}$, 
induced by this map. 
Let 
$W:= X^{\gen,\chi}\times_{X^{\ps}} X^{\ps,12}_{\PP_{\max}}$. 
The generic fibre $W[1/p]$ contains all the singular closed points, and since $X^{\gen}[1/p]$ is Jacobson, $W[1/p]$ contains the singular locus of $X^{\gen}[1/p]$. The special fibre $\overline{W}$ is a union of $Z^{12}_{\PP_{\max}}$ and $Y$ (as underlying topological spaces). Thus $\dim \overline{W}\le 6$ as  
$\dim Y\le 4$ by Lemma \ref{bound_Y} and 
$\dim Z^{12}_{\PP_{\max}}\le 6$ by Proposition \ref{bound_again}.
Since $W$ is a 
$\GL_d$-invariant subscheme of $X^{\gen}$, Lemma \ref{dim_gen_sp} implies that 
$\dim W[1/p]\le 6$. 
 It follows from Corollary
\ref{ci_gen_chi} that $\dim X^{\gen, \chi}[1/p]=\dim \Xbar^{\gen, \chi}=8$. Thus 
the codimension of the singular locus in 
$X^{\gen, \chi}[1/p]$ is at least $2$. 
\end{proof}

\begin{cor}\label{Xgen_chi_normal} $X^{\gen, \chi}$ is normal.
\end{cor} 
\begin{proof} Since $A^{\gen, \chi}$ is $\OO$-torsion 
free by Corollary \ref{ci_gen_chi}, the map $\OO\rightarrow A^{\gen, \chi}$ is flat. We have shown 
in Propositions \ref{Xbar_normal} and  \ref{X1p_normal}
that the fibre rings $L\otimes_{\OO} A^{\gen, \chi}$
and $k\otimes_{\OO} A^{\gen, \chi}$ are normal. 
Since $\OO$ is a regular ring \cite[Corollary 2.2.23]{BH}
implies that $A^{\gen, \chi}$ is normal.
\end{proof}

\begin{cor}\label{chi_normal} For each character $\chi: \mu_{p^{\infty}}(F)\rightarrow \OO^{\times}$ and $\rhobar: G_F\rightarrow \GL_d(k)$ the ring 
$R^{\square,\chi}_{\rhobar}$ is a normal integral domain.
\end{cor}
\begin{proof} The proof is essentially the same as the proof of 
Corollary \ref{sp_normal}. To see this, note that Corollary \ref{Xgen_chi_normal} implies that $X^{\gen,\chi}$ is normal, and that the formal completion of $X^{\gen,\chi}$ along the maximal ideal corresponding to $\rhobar$ is $R^{\square,\chi}_{\rhobar}$ by the $\chi$-version of Proposition \ref{main_complete} as explained in the proof of Corollary \ref{ci_gen_chi}.
\end{proof}

\begin{lem}\label{Kummer_fix_char0} 

Let $W$ be a
closed subscheme of $\Spec R^{\square, \chi}_{\rhobar}[1/p]$ with the property that $H^2(G_F, \adz\rho_x)\neq 0$ for all closed points $x\in W$. Then $\dim R^{\square, \chi}_{\rhobar}[1/p]- \dim W\ge 2.$
\end{lem}
\begin{proof} Since in characteristic zero 
$\adz \rho_x$ is a direct summand of $\ad \rho_x$
we obtain that $H^2(G_F, \ad \rho_x)\neq 0$ 
for all $x\in W$. This implies that $W$ is contained in the singular locus of $\Spec R^{\square, \chi}_{\rhobar}[1/p]$. Since $R^{\square, \chi}_{\rhobar}[1/p]$ is normal, the singular locus has codimension of at least $2$.
\end{proof}

The next result answers affirmatively a question raised 
by GB--Juschka in \cite[Question 1.10]{bjpp}.

\begin{cor}\label{BJ_conjecture} The map $R_{\det \rhobar}\rightarrow R^{\square}_{\rhobar}$ induces a bijection between the sets of 
irreducible components. 
\end{cor} 
\begin{proof} Since $R^{\square}_{\rhobar}$ is $\OO$-torsion free by Corollary \ref{ci}, 
the irreducible components of  $R^{\square}_{\rhobar}$
and $R^{\square}_{\rhobar}[1/p]$ coincide. Since 
the algebra $\OO[\mu_{p^{\infty}}(F)] [1/p]$ is semi-simple, we have 
\begin{equation}\label{invert_p} R^{\square}_{\rhobar}[1/p]\cong \prod_{\chi: \mu_{p^{\infty}}(F)\rightarrow \OO^{\times}} R^{\square, \chi}_{\rhobar}[1/p].
\end{equation}
It follows from Corollaries \ref{ci_chi} and \ref{chi_normal} that $R^{\square, \chi}_{\rhobar}$ 
is an $\OO$-torsion free integral domain. We note that 
the special fibres of these rings are non-zero, thus the 
rings themselves are non-zero. Hence, the localization 
$R^{\square, \chi}_{\rhobar}[1/p]$ is  non-zero
and is an integral domain.
\end{proof} 

\begin{cor}\label{generic_fibre_normal}$R^{\square}_{\rhobar}[1/p]$ is normal and $R^{\square}_{\rhobar}$ is reduced. 
\end{cor} 
\begin{proof} The first assertion follows from \eqref{invert_p} 
and Corollary \ref{chi_normal}.  Since $R^{\square}_{\rhobar}$ is $\OO$-torsion free by Corollary \ref{ci}, it is a subring of $R^{\square}_{\rhobar}[1/p]$. This implies the second assertion as normal rings are reduced.
\end{proof} 

\begin{cor}\label{factorial} If either $d=2$ and $[F:\Qp]\ge 4$ or $d\ge 3$ and $[F:\Qp]\ge 3$ then for each character $\chi: \mu_{p^{\infty}}(F)\rightarrow \OO^{\times}$ and $\rhobar: G_F\rightarrow \GL_d(k)$ the rings 
$R^{\square,\chi}_{\rhobar}$, $R^{\square,\chi}_{\rhobar}/\varpi$  are regular in codimension $3$. In particular, $R^{\square,\chi}_{\rhobar}$ and $R^{\square, \chi}_{\rhobar}/\varpi$ are factorial. 
\end{cor}
\begin{proof} The assumptions together with the lower bound on the codimension of the Kummer-irreducible locus in Proposition \ref{nspcl} and the containment $V^{\Kirr} \subset V^{\nspcl}$ 
(resp.\ Proposition \ref{Virr})
 imply that the complement of the absolutely irreducible non-special locus in $\Xbar^{\gen, \chi}$ (resp.\ absolutely irreducible locus in $X^{\gen, \chi}[1/p]$) has codimension at least $4$. It follows from Proposition \ref{Vnspcl_regular} (resp.\ Proposition \ref{Virr_reg}) that 
it contains the singular locus in $\Xbar^{\gen,\chi}$ (resp.~$X^{\gen, \chi}[1/p]$).
Hence, $X^{\gen, \chi}$ and $\Xbar^{\gen, \chi}$ are regular in codimension $3$, which 
implies that $R^{\square, \chi}_{\rhobar}$ and $R^{\square, \chi}_{\rhobar}/\varpi$ are regular in codimension $3$. Since both rings are also complete intersection by Corollary \ref{ci_chi}, they are 
factorial by a theorem of Grothendieck, see \cite{R3_factorial} for a short proof. 
\end{proof}

\begin{remar} The assumptions in Corollary \ref{factorial} are not optimal as the next 
Corollary shows. To find the optimal assumptions
one would have to further study the reducible locus
and we don't want to pursue this here. 
We note that if $F=\Qp$, $p\ge 5$ and $\rhobar= 
\bigl(\begin{smallmatrix} 1 & \ast \\ 0 & \omega \end{smallmatrix}\bigr)$ is non-split, where $\omega$ 
is the cyclotomic character modulo $p$, then it follows from \cite[Corollary B.5]{image} that 
$R^{\square, \chi}_{\rhobar}\cong \OO\br{x_1, \ldots, x_9}/(x_1 x_2- x_3 x_4)$ and hence is not factorial. Hence some assumptions in Corollary \ref{factorial} 
have to be made. 
\end{remar}

\begin{cor}\label{factorial_irr} If $\rhobar$ is absolutely irreducible then $R^{\square, \chi}_{\rhobar}$ and $R^{\square, \chi}_{\rhobar}/\varpi$ are factorial, 
except in the case $d=2$, $F=\QQ_3$ and $\rhobar\cong \rhobar(1)$. 
\end{cor}
\begin{proof} Since $\rhobar$ is absolutely irreducible $X^{\gen, \chi}[1/p]$ is regular by Proposition \ref{Virr_reg},
and the singular locus of $\Xbar^{\gen, \chi}$ is
contained in $Z^{\spcl}$, which has codimension at least $\frac{1}{2}[F:\Qp]d^2$ by Lemma \ref{special}. Thus if either $d>2$ 
or $F\neq \Qp$ then we can conclude that 
$R^{\square,\chi}_{\rhobar}$ and $R^{\square,\chi}_{\rhobar}/\varpi$ are regular in codimension $3$ and hence factorial. 

If  $\rhobar\not\cong \rhobar(1)$ then $H^2(G_{\Qp}, \adz \rhobar)=0$ and 
it follows from Lemma \ref{formal_smooth} that
$R^{\square, \chi}_{\rhobar}$ and 
$R^{\square,\chi}_{\rhobar}/\varpi$ are
formally smooth, hence regular and hence factorial. 

If $d=2$ then $\rhobar\cong \rhobar(1)$ implies that
$\det \rhobar= (\det\rhobar) \omega^2$. This leaves us with two cases $F=\QQ_2$ or $F=\QQ_3$.
If $p=2$ then $R^{\square}_{\rhobar}$ is formally smooth over $\OO[\mu]$ by \cite[Proposition 4.5]{che_unpublished}, and thus $R^{\square, \chi}_{\rhobar}$ and $R^{\square,\chi}_{\rhobar}/\varpi$
are regular. 

We claim that if $F=\QQ_3$, $d=2$ and $\rhobar\cong \rhobar(1)$ then 
the ring $R^{\square, \chi}_{\rhobar}$ is not factorial. It follows from \cite[Theorem 5.1]{boeckle} that in this case
$R^{\square, \chi}_{\rhobar}$ is formally smooth over 
$\OO\br{b,c,d}/(r)$, where $r= (1+d)^6 (1 + bc u) - (1+bc v)$ and  $u, v$ are units in $\OO\br{b,c}$. The ideal 
$\pp=(b, d)$ is prime of height $1$. If $R^{\square, \chi}_{\rhobar}$ were factorial then $\pp$ would have to be principal, \cite[\href{https://stacks.math.columbia.edu/tag/0AFT}{Tag 0AFT}]{stacks-project}, and thus there would exist $\pi \in \OO\br{b,c,d}$ such that we have an equality of ideals 
$(b, d)= ( r, \pi)$
in $\OO\br{b,c,d}$. By considering this modulo $(\varpi, c)$, we would conclude that $(d^3- d^6, \bar{\pi})$ is
the maximal ideal in $k\br{b,d}$. Since 
$(d^3-d^6, \bar{\pi})\rightarrow (b,d)/(b,d)^2$ is not surjective, we obtain a contradiction.
The same argument shows that $R^{\square, \chi}_{\rhobar}/\varpi$ is also not factorial.  
\end{proof}

\begin{prop}\label{Agen_chi_domain}
For each character $\chi: \mu_{p^{\infty}}(F)\rightarrow \OO^{\times}$ the rings 
$A^{\gen,\chi}$ and $A^{\gen, \chi}/\varpi$  are integral domains.
\end{prop}
\begin{proof}Since $A^{\gen,\chi}$ is normal by Corollary \ref{Xgen_chi_normal}, it is a product of normal
domains $A^{\gen, \chi}\cong A_1\times \ldots \times A_m$. 
The action of $G$ on $X^{\gen, \chi}$ leaves the connected 
components invariant by Lemma \ref{IrredCompInvariant}. 
It follows from Lemma \ref{fix} that each $\Spec A_i$ contains 
a closed point over the closed point $X^{\ps}$. Thus 
$A_i\otimes_{R^{\ps}} k$ are non-zero for $1\le i\le m$. 
If $m>1$ then this would imply that the fibre at 
the closed point of $X^{\ps}$ is not connected contradicting 
Lemma \ref{closed_orbit}. The same proof 
works also for the special fibre. 
\end{proof} 

Define $R^{\ps, \chi}:= R^{\ps}\otimes_{\OO[\mu], \chi} \OO $ for a character $\chi:\mu\to \OO^\times$ and using the isomorphism from Lemma~\ref{detmapmu}. We let 
$X^{\ps,\chi}=\Spec R^{\ps, \chi}$ and let $\Xbar^{\ps, \chi}$ be its special fibre. 
\begin{cor}\label{Rps_normal} The rings $R^{\ps}[1/p]$, $R^{\ps, \chi}[1/p]$ and the 
rigid spaces $(\Spf R^{\ps})^{\rig}$, $(\Spf R^{\ps, \chi})^{\rig}$  are normal. Moreover, 
$R^{\ps, \chi}[1/p]$
is an integral domain and thus the map
$R_{\det\rhobar}[1/p]\rightarrow R^{\ps}[1/p]$
induces a bijection between the sets of irreducible components.
\end{cor} 
\begin{proof} 
The assertion 
follows from \cite[Theorem A.1]{finite} using Corollary \ref{generic_fibre_normal}. As part of the proof one obtains
$R^{\ps}[1/p]=( A^{\gen}[1/p])^G$. This yields $R^{\ps, \chi}[1/p]=
( A^{\gen, \chi}[1/p])^G$. Proposition \ref{Agen_chi_domain}
implies that $A^{\gen, \chi}[1/p]$ is an integral domain. Hence
$R^{\ps, \chi}[1/p]$ is an integral domain, and the assertion 
about irreducible components is proved in the same manner as 
Corollary \ref{BJ_conjecture}.
\end{proof}

\begin{cor}\label{no_name} The image of 
$R^{\ps}$ in $A^{\gen}$ is the maximal $\OO$-torsion free quotient of $R^{\ps}$ 
and is also  
the maximal reduced quotient
of $R^{\ps}$. In particular, 
the map $R_{\det \Dbar}\rightarrow R^{\ps}\rightarrow R^{\ps}[1/p]$ induces a 
bijection between the sets of irreducible components. Moreover, if $\Dbar$ is multiplicity free then $R^{\ps}$ is reduced 
and $\OO$-torsion free.
\end{cor}
\begin{proof} By \cite[Theorem 2.20]{WE_alg} the map $X^{\gen}\sslash\GL_d \rightarrow X^{\ps}$ is an adequate homeomorphism. It follows from 
\cite[Proposition 3.3.5]{alper} that the kernel of 
$R^{\ps}\rightarrow (A^{\gen})^{\GL_d}$ is 
nilpotent and vanishes after inverting $p$.
Since $A^{\gen}$ is $\OO$-torsion free and reduced, this implies that both quotients coincide and are equal to the image of $R^{\ps}$ in $A^{\gen}$. This together with the last part of 
Corollary \ref{Rps_normal}
implies the assertion 
about the irreducible components. 

If $\Dbar$ is multiplicity free then 
$E$ is a generalized matrix algebra by
\cite[Theorem 2.22]{che_durham},  and 
it follows from \cite[Theorem 3.8 (4)]{WE_alg} that
$R^{\ps}=(A^{\gen})^{\GL_d}$, 
and so $R^{\ps}$ is $\OO$-torsion free and reduced. 
\end{proof} 

\begin{cor}\label{no_name2} The image of 
$R^{\ps, \chi}/\varpi$ in 
$A^{\gen, \chi}/\varpi$ is 
the maximal reduced quotient of 
$R^{\ps, \chi}/\varpi$. 
The image of $R^{\ps, \chi}$
in $A^{\gen, \chi}$ is the maximal reduced 
quotient of $R^{\ps, \chi}$ and is also 
the maximal $\OO$-torsion free quotient of $R^{\ps, \chi}$. Moreover, if 
$\Dbar$ is 
multiplicity free then both $R^{\ps, \chi}/\varpi$ and $R^{\ps, \chi}$ are integral domains.
\end{cor}
\begin{proof} If we work with the algebra 
$E^{\chi}:=E\otimes_{\OO[\mu], \chi} \OO$ instead of $E$ then the argument in the proof of Corollary \ref{no_name} gives adequate homeomorphisms
$$X^{\gen,\chi}\sslash\GL_d \rightarrow X^{\ps,\chi}, \quad \Xbar^{\gen,\chi}\sslash\GL_d \rightarrow \Xbar^{\ps, \chi}.$$
In particular, the kernel of
$R^{\ps, \chi}/\varpi\rightarrow A^{\gen, \chi}/\varpi$ is nilpotent. Since $A^{\gen,\chi}/\varpi$ is an integral domain by Proposition \ref{Agen_chi_domain} we obtain the first assertion. The argument with $R^{\ps, \chi}$ is the same as in Corollary \ref{no_name} using 
that $A^{\gen, \chi}$ is an integral domain.

If $\Dbar$ is multiplicity free then $E^{\chi}$ and $E^{\chi}/\varpi$ are generalized matrix algebras, and the argument in Corollary \ref{no_name} carries over. 
\end{proof}

\begin{lem} If $R^{\ps, \chi}/\varpi$ satisfies Serre's 
condition (S1) then $R^{\ps, \chi}/\varpi$ and $R^{\ps, \chi}$ are integral domains.
\end{lem}
\begin{proof}
We first note that 
$R^{\ps, \chi}/\varpi$
satisfies Serre's condition
(R0). Since the underlying 
reduced subschemes of $\Xbar^{\ps}$ and $\Xbar^{\ps, \chi}$ coincide, 
Proposition \ref{existence_Kirr}
implies that the Kummer-irreducible locus
$(\Xbar^{\ps, \chi})^{\Kirr}$ in $\Xbar^{\ps, \chi}$ is open dense. If $x\in (\Xbar^{\ps, \chi})^{\Kirr}$ is a closed point then the pseudo-character $D_x$ is absolutely irreducible, and hence is associated to an absolutely irreducible representation  which we denote by $\rho_x$. Let $R_{\rho_x}$ be the universal deformation ring of $R_{\rho_x}$ and let $R^{\chi}_{\rho_x}$ be the quotient of $R_{\rho_x}$
parameterizing deformations such that the restriction of the determinant to $\Art_F(\mu)\subset G_F^{\ab}$ is equal to $\chi$.
Since $R^{\square, \chi}_{\rho_x}$ is formally smooth over $R_{\rho_x}^{\chi}$, the 
 Kummer-irreducibility of $x$ implies that 
$R_{\rho_x}^{\chi}$ is regular. The proof of 
\cite[Lemma 5.1.6]{bj} shows that $x$ is a regular 
point in $\Xbar^{\ps, \chi}$.
Hence $\Xbar^{\ps, \chi}$ contains an open dense regular 
subscheme, which implies that  
 $R^{\ps, \chi}/\varpi$ satisfies (R0). Since 
$R^{\ps, \chi}/\varpi$ satisfies (S1) by assumption we conclude that $R^{\ps, \chi}/\varpi$ is reduced. It follows from 
Lemma \ref{no_name2} and Proposition \ref{Agen_chi_domain}  that $R^{\ps, \chi}/\varpi$ is an integral domain.

Let $R^{\ps,\chi}\twoheadrightarrow R^{\ps,\chi}_{\mathrm{tf}}$ be the maximal 
$\OO$-torsion free quotient quotient, and let 
$\mathfrak a$ be the kernel of this map.
We have an exact sequence 
$0\rightarrow \mathfrak a/\varpi\rightarrow
R^{\ps,\chi}/\varpi \rightarrow
R^{\ps,\chi}_{\mathrm{tf}}/\varpi
\rightarrow 0$. It follows from Corollary \ref{no_name2} that $\mathfrak a$ is nilpotent. Since $R^{\ps,\chi}/\varpi$ is reduced we deduce from the exact sequence that $\mathfrak a/\varpi$ is zero. Nakayama's lemma implies that $\mathfrak a=0$. Thus $R^{\ps,\chi}$ 
is $\OO$-torsion free, and
hence is a subring of $A^{\gen, \chi}$ by Corollary \ref{no_name2}.
Since $A^{\gen, \chi}$ is 
domain we conclude that 
$R^{\ps,\chi}$ is an integral domain.
\end{proof}

\begin{remar}
 We expect that the rings $R^{\ps, \chi}$ and $R^{\ps, \chi}/\varpi$ are integral domains. Although we know the dimension of $R^{\ps, \chi}/\varpi$ by \cite[Theorem 5.5.1]{bj} 
 we cannot conclude that the ring is complete intersection (which would imply that (S1) holds) as we lack a presentation analogous to \eqref{present_square_x}.
  Since $A^{\gen, \chi}$ and $A^{\gen, \chi}/\varpi$ are 
 integral domains this question is closely related to the embedding problem discussed in \cite[Section 1.3.4]{bel_che}.
 \end{remar}

\section{Deformation rings with fixed determinant}\label{fixed_det}

Let $\rhobar: G_F\rightarrow \GL_d(k)$ be a representation with pseudo-character $\Dbar$ and let $\psi: G_F\rightarrow \OO^{\times}$
be a character lifting 
$\det \rhobar=\det \Dbar$.  Let
$$
R^{\square, \psi}_{\rhobar}:=
R^{\square}_{\rhobar}\otimes_{R_{\det \rhobar}, \psi} \OO.$$ 
Let $\mu:=\mu_{p^{\infty}}(F)$ and let $\chi:\mu\to \OO^\times$ be a character such that the restriction of $\psi$ to $\mu$ under the Artin map $\mu\to G_F^\ab$ from local class field theory is equal to $\chi$. Then $R^{\square, \psi}_{\rhobar}$ is a quotient of the ring $R^{\square, \chi}_{\rhobar}$ considered in the previous section. 
We let $X^{\square, \chi}=
\Spec R^{\square, \chi}_{\rhobar}$, $X^{\square, \psi}=
\Spec R^{\square, \psi}_{\rhobar}$
and denote by $\Xbar^{\square, \chi}$
and 
$\Xbar^{\square, \psi}$ their special fibres.

Let $\mathcal X: \mathfrak{A}_{\OO}\rightarrow \Sets$ be the functor, which sends $(A, \mm_A)$ to the group 
$\mathcal X(A)$ of continuous characters $\theta: G_F\rightarrow 1+\mm_A$ whose restriction to $\mu$ under the Artin map is trivial. It follows from Lemma \ref{detmapmu} that the functor 
$\mathcal X$ is pro-represented by 
\begin{equation}\label{OX1}
\OO(\mathcal X)\cong R_{\Eins}\otimes_{\OO[\mu]} \OO\cong \OO\br{y_1, \ldots, y_{[F:\Qp]+1}}.
\end{equation}

For $e\in\mathbb{N}$ 
let $\varphi_e:\mathcal X\to \mathcal X$ be the natural transformation that sends $\theta\in \mathcal X(A)$ to $\theta^e$. We also write $\varphi_e$ for the induced maps $\OO(\mathcal X)\to \OO(\mathcal X)$  
and $\Spec \OO(\mathcal X)\to \Spec \OO(\mathcal X)$. The natural transformation $D^{\square,\chi}_{\rhobar}\to \mathcal X$,  $\rho\mapsto (\det\rho) \psi^{-1}$ induces a homomorphism of local $\OO$-algebras $\OO(\mathcal X)\rightarrow R^{\square, \chi}_{\rhobar}$; we will consider
$R^{\square, \chi}_{\rhobar}$ as 
$\OO(\mathcal X)$-algebra via this map in the statements below. 

\begin{prop}\label{prop_TwistFunctors}
One has a natural isomorphism of functors
$$D^{\square,\chi}_{\rhobar} \times_{\mathcal X, \varphi_d} \mathcal X \cong  D^{\square,\psi}_{\rhobar}\times
\mathcal X.$$
\end{prop}
\begin{proof}
Let $(A,\mm_A)$ be in $\mathfrak{A}_\OO$. 
An element in
$(D^{\square,\chi}_{\rhobar} \times_{\mathcal X, \varphi_d} \mathcal X) (A)$ 
is a pair $(\rho, \theta)$ 
such that $\theta:G_F\to 1+\mm_A$ 
is a continuous homomorphism that is
trivial on $\mu$, $\rho:G_F\to \GL_d(A)$
is a continuous homomorphism such that
$\det\rho$ and $\chi$ agree when
restricted to $\mu$, and one has
$(\det\rho)\psi^{-1}=\theta^d$.
An element in $(D^{\square,\psi}_{\rhobar}\times \mathcal X)(A)$ is a pair $(\rho_1,\theta_1)$ where $\theta_1:G_F\to 1+\mm_A$ is a continuous homomorphism that is trivial on $\mu$ and $\rho_1:G_F\to \GL_d(A)$ is a continuous homomorphism such that $\det\rho_1=\psi$. One verifies that the map
$$ (\rho, \theta)\mapsto (\rho \otimes \theta^{-1}, \theta) $$ 
defines a bijection that is natural in $A$.
\end{proof}
\begin{cor}\label{cor_TwistedIsom}
Proposition~\ref{prop_TwistFunctors} induces a natural isomorphism 
\[  R_{\rhobar}^{\square,\chi}\otimes_{\OO(\mathcal X), \varphi_d}\OO(\mathcal X) \cong   R_{\rhobar}^{\square,\psi} \widehat\otimes_\OO\OO(\mathcal X). \]
\end{cor}

We now clarify some properties of the map $\varphi_d:\OO(\mathcal X)\to \OO(\mathcal X)$.

\begin{lem}\label{lem_R1Basics}
The map $\varphi_d$ is finite and flat and 
becomes \'etale after inverting $p$. Moreover, it induces a universal homeomorphism on the special fibres. 
\end{lem}\label{univ_homeo}
\begin{proof} We may write  $d= e p^m$, such that $p$ does not divide $e$. Then  $\varphi_d= \varphi_{p^m} \circ \varphi_e$.
Since  $e$ is prime to $p$, elements in $1+\mm_A$ for $(A,\mm_A)$ in 
$\mathfrak{A}_{\OO}$ possess 
a unique $e$-th root in $1+\mm_A$ 
by the binomial theorem, and
it follows that $\varphi_e$ 
is an isomorphism. We thus may 
assume that $d$ is a power of $p$.

The map 
$\varphi_d:\OO(\mathcal X)\to \OO(\mathcal X)$ sends $y_i$ to $(1+y_i)^d-1$. 
One checks that the monomials 
$\prod_{i=1}^{[F:\Qp]+1} y_i^{m_i}$ with 
$0\le m_i\le d-1$ form a basis of $\OO(\mathcal X)$ 
as $\OO(\mathcal X)$-module via $\varphi_d$, by checking the assertion modulo $\varpi$ and using 
Nakayama's lemma. A (standard) calculation
shows that the discriminant is a power of $p$ up to a sign. Thus $\varphi_d$ becomes \'etale after inverting $p$.

The map $\overline{\varphi}_d: \OO(\mathcal X)/\varpi\rightarrow \OO(\mathcal X)/\varpi$ is a power of 
the relative Frobenius of $\Spec (\OO(\mathcal X)/\varpi) /\Spec k$. The last assertion follows from 
\cite[\href{https://stacks.math.columbia.edu/tag/0CCB}{Tag 0CCB}]{stacks-project}.
\end{proof}

In the following results we deduce properties of the ring $R^{\square, \psi}_{\rhobar}$.
\begin{cor}\label{ci_psi} The following hold:
\begin{enumerate}
\item $R^{\square, \psi}_{\rhobar}$ is a local complete intersection, flat over $\OO$ and of relative dimension
$(d^2-1)([F:\Qp]+1)$.
\item $R^{\square, \psi}_{\rhobar}/\varpi$ is a local complete intersection of dimension $(d^2-1)([F:\Qp]+1)$.
\end{enumerate}
\end{cor}
\begin{proof}
The pushout of the isomorphism from Proposition~\ref{Prop-RelCI} under $R_{\det\rhobar}\to\OO$, which corresponds to $\psi$, gives an isomorphism 
    \[ \OO\br{x_1,\dots,x_r}/(f_1,\ldots,f_t)\xrightarrow{\simeq} R^{\square,\psi}_{\rhobar} \]
with $r-t=(d^2-1)([F:\Qp]+1)$. 
To prove (1) and (2) it thus suffices to show that the dimension of 
$R^{\square, \psi}_{\rhobar}/\varpi$ is at most  $(d^2-1)([F:\Qp]+1)$, or equivalently, see \eqref{OX1}, it suffices to show that 
\begin{equation}\label{bound_for_psi}
\dim \bigl((R^{\square, \psi}_{\rhobar}\wtimes_{\OO} \OO(\mathcal X))/\varpi\bigr) \le d^2([F:\Qp]+1).
\end{equation}
Let us write
$\overline{\mathcal X}:=\Spec \OO(\mathcal X)/\varpi$. 
Since 
$\overline{\varphi}_d: \overline{\mathcal X}\rightarrow \overline{\mathcal X}$
 is a universal homeomorphism 
by Lemma \ref{lem_R1Basics}
the map 
\begin{equation}\label{homeo1}
\Xbar^{\square, \chi}
\times_{\overline{\mathcal X}, \overline{\varphi}_d} \overline{\mathcal X}\rightarrow \Xbar^{\square, \chi}
\end{equation}
is a homeomorphism. In particular, the spaces have the same dimension, which is equal to $d^2([F:\Qp]+1)$ by Corollary \ref{ci_chi}.
We conclude using
Corollary~\ref{cor_TwistedIsom}
that \eqref{bound_for_psi} is an equality. 
\end{proof}

\begin{lem}\label{formal_smooth_psi} Let $\rho: G_F \rightarrow \GL_d(\kappa)$ be a representation, such that $\det \rho=\psi$, where $\kappa$ is either a finite or local field of characteristic $p$ or a finite extension of $L$. If $H^2(G_F, \adz\rho)=0$, then the ring $R^{\square, \psi}_{\rho}$ is formally smooth over $\Lambda$ with $\Lambda$ as in Subsection~\ref{sec_completions}. 
\end{lem}
\begin{proof} This is the same proof as the 
proof of Lemma \ref{formal_smooth}.
\end{proof} 

\begin{thm}\label{psi_normal} The rings $R^{\square, \psi}_{\rhobar}$ and $R^{\square, \psi}_{\rhobar}/\varpi$ are normal integral domains. 
\end{thm} 
\begin{proof} We will first prove that
$R^{\square, \psi}_{\rhobar}/\varpi$
is normal. 
Since $R^{\square,\psi}_{\rhobar}/\varpi$ 
is complete intersection by Corollary~\ref{ci_psi},
it suffices to show that
$R^{\square, \psi}_{\rhobar}/\varpi$ 
satisfies Serre's condition (R1). 
Let $\mathfrak p \in \Xbar^{\square,\psi}
:=\Spec R^{\square,\psi}_{\rhobar}/\varpi$ be a point of height at most $1$ and assume that 
the local ring at $\pp$
is not regular. Then by Lemma~\ref{formal_smooth_psi} there is a closed irreducible subset $Z$ of $\Xbar^{\square,\psi}$
of codimension at most $1$, 
the closure of $\mathfrak p$, 
such that for all $z\in Z$ with 
finite or local residue field the space 
$H^2(G_F, \adz\rho_z)$ is non-zero.
Using the explicit bijection from 
the proof of 
Proposition~\ref{prop_TwistFunctors}, 
and the isomorphism of 
Corollary \ref{cor_TwistedIsom} modulo $\varpi$
it follows that there is a closed irreducible subset 
$W\subset \Xbar^{\square, \chi}\times_{\overline{\mathcal X}, \overline{\varphi}_d} \overline{\mathcal X}$ 
of codimension at most $1$, such that for all
$w\in W$ with finite or local residue field 
the space $H^2(G_F, \adz\rho_w)$ is non-zero,
where, 
as in the proof 
of Proposition~\ref{prop_TwistFunctors},
the point $w$ corresponds to a pair $(\rho_w, \theta_w)$. 
Since the map \eqref{homeo1} is a homeomorphism 
and sends $(\rho_w, \theta_w)$ to $\rho_w$, the image 
of $W$ in $\Xbar^{\square, \chi}$, which we denote by $W'$, 
is closed irreducible of codimension
at most $1$ in 
$\Xbar^{\square, \chi}$ 
and all $x\in W'$
with finite or local residue field have non-vanishing
$H^2(G_F, \adz\rho_x)$. Lemma \ref{Kummer_fix} implies that the codimension 
of $W'$ is at least $2$ yielding a contradiction.

Let us prove that $R^{\square, \psi}_{\rhobar}$ is normal. Since 
$R^{\square, \psi}_{\rhobar}$ is $\OO$-torsion free by Corollary \ref{ci_psi} and we know that the special fibre is normal, 
it is enough to prove that 
$R^{\square, \psi}_{\rhobar}[1/p]$ is normal,
see  the proof of Proposition \ref{Xgen_chi_normal}. 
Lemma \ref{lem_R1Basics} implies 
that the map 
\begin{equation}\label{etale}
X^{\square, \chi}[1/p]
\times_{\mathcal X[1/p], \varphi_d}
\mathcal X[1/p]\rightarrow 
X^{\square, \chi}[1/p]
\end{equation}
is finite \'etale. We proceed exactly 
as in the proof for the special fibre,
using \eqref{etale} instead of \eqref{homeo1} and Lemma \ref{Kummer_fix_char0} instead of Lemma 
\ref{Kummer_fix}.
\end{proof}

\begin{cor} The absolutely irreducible locus
is dense in $\Spec R^{\square, \psi}_{\rhobar}[1/p]$ and the Kummer-irreducible locus is dense 
in $\Spec R^{\square, \psi}_{\rhobar}/\varpi$.
\end{cor}
\begin{proof} 
By Proposition~\ref{nspcl_formal} and Corollary~\ref{irr_gen_dense} the absolutely irreducible locus is dense open in $\Spec R^{\square,\chi}_{\rhobar}/\varpi$ and in $\Spec R^{\square,\chi}_{\rhobar}[1/p]$. Arguing as in the proof of Theorem \ref{psi_normal} one deduces that the absolutely irreducible locus is dense open in the spaces $\Spec R^{\square,\psi}_{\rhobar}/\varpi$ and $\Spec R^{\square,\psi}_{\rhobar}[1/p]$. For absolutely irreducible $x\in \Spec R^{\square,\chi}_{\rhobar}/\varpi$ Kummer-irreducibility implies $H^2(G_F,\adz\rho_x)=0$, so the assertion on the density of the Kummer-irreducible locus in $\Spec R^{\square,\chi}_{\rhobar}/\varpi$ follows from the proof of Theorem \ref{psi_normal}.
\end{proof}

As explained in Section~\ref{sec_irrcomp} both $R^{\ps}$ and $A^{\gen}$ are naturally $R_{\det \Dbar}$-algebras. Moreover, $\det \Dbar=\det\rhobar$. 
We let 
$$ R^{\ps, \psi}:=R^{\ps}\otimes_{R_{\det \Dbar}, \psi}\OO, \quad A^{\gen, \psi}:=A^{\gen}\otimes_{R_{\det \Dbar}, \psi}\OO.$$

\begin{cor}\label{ci_gen_psi} The following hold:
 \begin{enumerate}
 \item $A^{\gen, \psi}$ is $\OO$-flat, 
 equi-dimensional of dimension $1+ (d^2-1)([F:\Qp]+1)$, normal and is locally complete intersection;
 \item $A^{\gen, \psi}/\varpi$ is  equi-dimensional of dimension $(d^2-1)([F:\Qp]+1)$, normal, and is locally complete intersection. 
 \end{enumerate}
 \end{cor}
\begin{proof} The ring $A^{\gen}$ is excellent,
since it is a finitely generated over a complete local Noetherian ring. Thus its local rings are also excellent. An excellent local ring is normal if and only if its completion 
with respect to the maximal ideal is normal, 
\cite[Theorem 32.2 (i)]{matsumura}. Given this
the proof is the same as
the proof of Corollary \ref{ci_gen_chi} 
using Theorem  \ref{psi_normal}.
\end{proof} 

\begin{cor}\label{Agen_psi_domain} The rings $A^{\gen, \psi}$ and $A^{\gen,\psi}/\varpi$   are integral domains.
 \end{cor}
 \begin{proof} The proof is the same as the 
 proof of Proposition \ref{Agen_chi_domain}.
 \end{proof}
 
\begin{cor}\label{Rpspsi_normal} The ring $R^{\ps, \psi}[1/p]$ and 
the rigid space $(\Spf R^{\ps, \psi})^{\rig}$ are normal. 
The ring $R^{\ps, \psi}[1/p]$ is 
an integral domain. 
\end{cor} 
\begin{proof} This follows from \cite[Corollary A.10]{finite}. The last part is proved in the same way as Corollary \ref{Rps_normal} using Corollary \ref{Agen_psi_domain}.
\end{proof}
 
\begin{cor} The maximal reduced quotient of 
$R^{\ps, \psi}$ is equal to 
 the maximal $\OO$-torsion free quotient of
  $R^{\ps, \psi}$  and is an integral domain. 
  Moreover, if 
  $\Dbar$ is multiplicity free then $R^{\ps, \psi}$ is an 
$\OO$-torsion free integral domain. 
\end{cor} 
\begin{proof}This is proved in the same way as 
Corollary \ref{no_name}.
\end{proof}

\begin{prop}\label{prop_det_to_square} The map 
\begin{equation}\label{det_to_square}
R_{\det \rhobar}\rightarrow R^{\square}_{\rhobar}
\end{equation}
is flat.
\end{prop}
\begin{proof} Let $S:=\OO\br{z, y_1, \ldots, y_{1+[F:\Qp]}}$. By arguing  as in the proof of Proposition \ref{Prop-RelCI}  we may choose presentations 
$$R_{\det \rhobar}\cong S/((1+z)^m-1), \quad 
R^{\square}_{\rhobar}\cong S\br{x_1, \ldots, x_r}/( (1+z)^m-1, f_1, \ldots, f_t),$$
such that \eqref{det_to_square} is a map of $S$-algebras and $(1+z)^m-1, f_1, \ldots, f_t$ is a regular sequence in $S\br{x_1, \ldots, x_r}$. Let $S':= S\br{x_1, \ldots, x_r}/(f_1, \ldots, f_t)$. Then $S'$ is complete intersection, and hence Cohen--Macaulay, and the fibre ring $k\otimes_S S'$ is isomorphic to $R^{\square, \psi}_{\rhobar}/\varpi$, which 
has dimension equal to $\dim R^{\square}_{\rhobar} - \dim R_{\det \rho}=\dim S'- \dim S$, by Corollary \ref{ci_psi}. Since $S$ is regular, the 
fibre-wise criterion for flatness, \cite[Theorem 23.1]{matsumura}, implies that $S'$ is flat over $S$.
Hence, $R^{\square}_{\rhobar}\cong S'/((1+z)^m-1)$ is flat over
$R_{\det \rhobar}\cong S/((1+z)^m-1)$.
\end{proof}

 \section{Density of points with prescribed \texorpdfstring{$p$}{p}-adic Hodge theoretic properties}\label{sec_dense}
 
 We fix a continuous representation $\rhobar: G_F \rightarrow \GL_d(k)$, let $R^{\square}_{\rhobar}$ be its universal framed deformation ring and let $X^{\square}=\Spec R^{\square}_{\rhobar}$. If $x:R^{\square}_{\rhobar}\rightarrow \Qpbar$ is an $\OO$-algebra homomorphism then we 
 denote by $\rho^{\square}_x:G_F\rightarrow \GL_d(\Qpbar)$ the specialization of 
the universal framed deformation  $\rho^{\square}: G_F\rightarrow \GL_d(R^{\square}_{\rhobar})
 $ at $x$. 
 In this section we will study Zariski closures of  subsets $\Sigma \subset X^{\square}(\Qpbar)$, 
 such that $\rho^{\square}_x$ is potentially semi-stable for all $x\in \Sigma$ and satisfies additional conditions imposed on either  the Hodge--Tate weights or the inertial type of $\rho^{\square}_x$. Recall that the
 Hodge--Tate weights $\HT(\rho)$ of a potentially semi-stable representation $\rho$ is a collection $\underline{k}$ of $d$-tuples of integers
 $\underline{k}_{\sigma}=(k_{\sigma,1}\ge k_{\sigma,2}\ge\ldots \ge k_{\sigma, d})$ for each embedding $\sigma: F\hookrightarrow \Qpbar$ and we say that 
 $\underline{k}$ is \textit{regular} if all the inequalities are strict. Let $$\Sigma^{\cris}:=\{ x\in X^{\square}(\Qpbar): \rho_x^{\square} \text{ is crystalline with regular Hodge--Tate weights}\}.$$
 For a fixed regular Hodge--Tate weight $\underline{k}$ we let 
 $$\Sigma_{\underline{k}}:=\{x\in X^{\square}(\Qpbar): \rho_x^{\square} \text{ is potentially semi-stable with } \HT(\rho^{\square}_x)= \underline{k}\}.$$
 If $\rho$ is potentially semi-stable then to it we may attach a Weil--Deligne representation $\WD(\rho)$; we denote by $\WD(\rho)^{F-\mathrm{ss}}$ its Frobenius semi-simplification. We may attach a smooth 
 irreducible representation of $\GL_d(F)$, which we denote by $\LL(\WD(\rho))$, to $\WD(\rho)^{F-\mathrm{ss}}$  via the classical Langlands correspondence, see \cite[Section 1.8]{6auth} for more details
 and further references. 
 
 Let $\Sigma^{\prnc}_{\underline{k}}$ be the subset of  $\Sigma_{\underline{k}}$, 
 such that $x\in \Sigma_{\underline{k}}$ lies in
 $\Sigma^{\prnc}_{\underline{k}}$ if and only if
 $\LL(\WD(\rho_x^{\square}))$ is a principal series representation.
 In terms of the Galois side  $\Sigma^{\prnc}_{\underline{k}}$ may be characterised
 as the set of $x\in \Sigma_{\underline{k}}$ such that the restriction of $\rho^{\square}_x$ to the Galois group of some finite abelian extension of  $F$ is crystalline. 
 
 Let $\Sigma^{\spcd}_{\underline{k}}$ be the subset of $\Sigma_{\underline{k}}$ such that $x$ lies in $\Sigma^{\spcd}_{\underline{k}}$ if and only if $\WD(\rho_x)$ is irreducible as a representation of the Weil group $W_F$ of $F$, and 
 is induced from a $1$-dimensional representation of $W_E$, where $E$ 
 is an unramified extension of $F$ of degree $d$. In this case, $\LL(\WD(\rho_x))$ is a supercuspidal representation of $\GL_d(F)$.
 
 The goal of this section is the following theorem.
 \begin{thm}\label{density} Assume that $p\nmid 2d$. Let $\Sigma$ be any of the sets $\Sigma^{\cris}$, $\Sigma^{\prnc}_{\underline{k}}$, $\Sigma^{\spcd}_{\underline{k}}$ defined above.  Then $\Sigma$ is Zariski dense in
 $X^{\square}$.
 \end{thm}
 
 \begin{remar} We could additionally require 
 the representations in $\Sigma^{\cris}$ to be \textit{benign} in the sense of \cite[Definition 6.8]{EP} or instead of considering crystalline representations fix an inertial type. 
 
 One could also change the definition of $\Sigma^{\spcd}_{\underline{k}}$ to allow 
 $E$ to be a ramified extension of $F$, see \cite[Section 5.3]{EP}.
 \end{remar}
 The problem for $\Sigma^{\cris}$ has been studied 
 by Colmez \cite{colmez_tri}, 
 Kisin \cite{kisin_GQp_GL2Qp},
 Chenevier \cite{che_inf_fern},
 Nakamura \cite{nakamura1}, \cite{nakamura}. 
 Hellmann and Schraen have studied the problem for $\Sigma^{\prnc}_{\underline{k}}$ and $\Sigma^{\cris}$ in \cite{HS}. Emerton 
 and VP have studied the problem for $\Sigma^{\cris}$,
 $\Sigma^{\prnc}_{\underline{k}}$ 
 and $\Sigma^{\spcd}_{\underline{k}}$ in \cite{EP}. 
 A common feature of these papers is 
 that they show that 
 the closure of $\Sigma$ is a union of irreducible components of $X^{\square}$ and density is equivalent to 
 showing that $\Sigma$ meets each irreducible component. If one knows the irreducible components
 then one might hope to show density this way. This strategy has been carried out   for $\Sigma^{\cris}$ by Colmez--Dospinescu--VP in \cite{CDP2} for $p=d=2$ and $F=\Qp$ and by AI in \cite{Iy} for $p>d$ and $F$ arbitrary, when 
 $\rhobar$ is the trivial representation, 
 where after determining irreducible components
 one can write down the lifts explicitly. We note that using 
 Corollary \ref{BJ_conjecture} one may remove the assumption
 $p>d$ in \cite[Theorem 5.11]{Iy}. It seems impossible to carry this out for 
  arbitrary $\underline{k}$ and $\rhobar$  directly, even if 
 one knows that the irreducible components of 
 $X^{\square}$ are in bijection with irreducible 
 components of $\Spec R_{\det \rhobar}$. Instead 
 we combine our knowledge of irreducible components
 with results of \cite{EP}.
 
 The paper \cite{EP} builds on the global patching 
 arguments carried out in \cite{6auth}, which
 assumes that $p\nmid 2d$ and $\rhobar$ has 
 a potentially diagonalisable lift. This
 last condition can be easily verified if 
 $\rhobar$ is semi-simple, see 
 \cite[Lemma 2.2]{6auth}; it has been shown 
 to always be satisfied in 
 \cite[Theorem 1.2.2]{EG_stack}. The output of 
 \cite{6auth} is a complete local Noetherian $\OO$-algebra $R_{\infty}$ with residue field $k$ 
 and a linearly compact $R_{\infty}$-module
 $M_{\infty}$, which carries a continuous,
 $R_{\infty}$-linear action of $G:=\GL_d(F)$.
 Moreover, the action of $R_{\infty}[K]$ on
 $M_{\infty}$ extends (uniquely) to a continuous action 
 of the completed group algebra $R_{\infty}\br{K}$,
 where $K:=\GL_d(\OO_F)$, so that $M_{\infty}$ 
 is a finitely generated $R_{\infty}\br{K}$-module.
 
 \begin{lem}\label{components_Rinfty} We have an isomorphism of $R^{\square}_{\rhobar}$-algebras:
 $$ R_{\infty}\cong R^{\square}_{\rhobar}\wtimes_{\OO} A,$$
 where $A$ is a complete local Noetherian $\OO$-algebra, which is  $\OO$-torsion free, reduced and equi-dimensional. Thus the  ring $R_{\infty}$ is a reduced,
 $\OO$-torsion free and flat $R^{\square}_{\rhobar}$-algebra. 
 
 After replacing $L$ by a finite extension, the irreducible components of $\Spec R_{\infty}$
 are of the form $\Spec(R^{\square, \chi}_{\rhobar}\wtimes_{\OO} A/\pp)$, for a character $\chi:\mu_{p^{\infty}}(F)\rightarrow \OO^{\times}$
 and a minimal prime $\pp$ of $A$. Moreover, distinct pairs $(\chi, \pp)$ give rise to distinct irreducible components of $\Spec R_{\infty}$. 
 \end{lem}
 \begin{proof} The ring $R_{\infty}$ is defined 
 in \cite[Section 2.8]{6auth} and is formally 
 smooth over the ring denoted by
 $R^{\mathrm{loc}}$ in \cite[Section 2.6]{6auth}.
 The ring $R^{\mathrm{loc}}$ is a completed tensor
 product over $\OO$ of $R^{\square}_{\rhobar}$, the ring
 $R_{\tilde{v}_1}^{\square}$, which is formally 
 smooth over $\OO$ by \cite[Lemma 2.5]{6auth},
 and potentially semi-stable 
 rings at other places above $p$, denoted by $R^{\square, \xi, \tau}_{\tilde{v}}$ in \cite[Section 2.4]{6auth}. These are 
 $\OO$-torsion free, reduced and equi-dimensional by
 \cite[Theorem 3.3.8]{kisin_pst}. Thus $R_{\infty}\cong R^{\square}_{\rhobar}\wtimes A$, where $A$ is formally smooth over
 the ring 
 $\wtimes_{v\in S_p\setminus \pp} R^{\square,\xi,\tau}_{\tilde{v}}$ in the notation of 
 \cite{6auth}. Since the rings $R^{\square,\xi,\tau}_{\tilde{v}}$ are $\OO$-torsion free, reduced and equi-dimensional, so is the ring $A$ 
 by \cite[Corollary A.2]{6auth} and \cite[Lemma A.1]{HP}. Since $R^{\square}_{\rhobar}$ is 
 also $\OO$-torsion free, reduced and equi-dimensional we obtain that the same holds
 for $R_{\infty}$. Since $A$ is $\OO$-torsion free, $R_{\infty}$ is a flat $R^{\square}_{\rhobar}$-algebra. 
 
It follows from \cite[Lemma A.5]{HP} that after replacing $L$ with a finite extension, we may assume that for all minimal primes $\pp$ of $A$, the quotient $A/\pp$ is geometrically integral, by which we mean that 
$(A/\pp)\otimes_{\OO} \OO_{L'}$ is integral domain for all finite extensions $L'/L$.
If $\pp'$ is a minimal prime of $R^{\square}_{\rhobar}$ then $R^{\square}_{\rhobar}/\pp'= R^{\square, \chi}_{\rhobar}$ for a unique 
character $\chi: \mu_{p^{\infty}}(F)\rightarrow \OO^{\times}$ by Corollary \ref{BJ_conjecture}.  The moduli interpretation of 
$R^{\square, \chi}_{\rhobar}$ together with Corollary \ref{chi_normal} shows that the ring is geometrically integral. It follows from \cite[Lemma 3.3 (5)]{blght} that the minimal primes $\qq$ of $R_{\infty}$ are of the form $\pp'(R^{\square}_{\rhobar}\wtimes_{\OO} A)+\pp(R^{\square}_{\rhobar}\wtimes_{\OO} A)$, where $\pp'$ is the image of $\qq$ in $\Spec R^{\square}_{\rhobar}$ and 
$\pp$ is the image of $\qq$ in $\Spec A$. This implies the last assertion.
\end{proof}

In our arguments we will not invoke the assumption $p\nmid 2d$, since eventually this restriction used in construction of $M_{\infty}$ should become redundant. In particular,  the next two Lemmas  do not use this assumption.
 \begin{lem}\label{dth-root} Let $\psi:G_F\rightarrow \OO^{\times}$ be a character such 
 that $\psi$ is trivial on the torsion subgroup of $G_F^{\ab}$. Then after replacing $L$ by a finite extension we may find a character $\eta: G_F \rightarrow \OO^{\times}$ such that $\eta^d=\psi$. 
 \end{lem}
 \begin{proof} It follows from local class field theory that 
 the maximal torsion-free quotient of $G_F^{\ab}$ is isomorphic 
 to $\widehat{\mathbb Z}\times \Zp^{m}$, where $m=[F:\Qp]$. We choose topological generators $\gamma_1, \ldots, \gamma_{m+1}$, where $\gamma_1$ is a generator of $\widehat{\mathbb Z}$. Let $\overline{\psi(\gamma_1)}$ be the 
 image of 
 $\psi(\gamma_1)$ in $k$. If it is not equal to $1$ then 
  choose $\lambda\in \overline{k}$ such that $\lambda^d=
 \overline{\psi(\gamma_1)}$. We enlarge $L$, so that the residue field contains $\lambda$ and
 let $\mu: \widehat{\mathbb Z}\times \Zp^{m} \rightarrow \widehat{\mathbb{Z}}\rightarrow \OO^{\times}$ be the unramified character, such that $\mu(\gamma_1)$ is equal to 
 the Teichm\"uller lift of $\lambda$. 
 After replacing $\psi$ with $\psi \mu^{-d}$ we may 
 assume that $\psi(\gamma_1)\equiv 1\pmod{\varpi}$.
Thus we may view $\psi$ as a character on $\Zp^{m+1}$ and 
$\psi(\gamma_i)\equiv 1\pmod{\varpi}$ for $1\le i \le m+1$. 
After enlarging 
$L$ we may find $y_i \in (\varpi)$ such that $(1+y_i)^d= \psi(\gamma_i)$. 
Since the series $(1+y_i)^x:=\sum_{n=0}^{\infty} \begin{pmatrix}x\\ n\end{pmatrix} y_i^n$ converges for all $x\in \Zp$, we may define $\eta$ on $\Zp^{m+1}$ by sending $\gamma_i$ to $1+y_i$ and then inflate it to $G_F$. 
\end{proof}

\begin{lem}\label{cris_char} Let $\kappa: G_F\rightarrow \OO^{\times}$ be a character. Then 
there is a crystalline character $\psi: G_F\rightarrow \OO^{\times}$ such that $\psi\kappa^{-1}$ is trivial on the torsion part of $G_F^{\ab}$. 

In particular, given characters $\overline{\kappa}: G_F\rightarrow k^{\times}$ and $\chi: \mu_{p^{\infty}}(F)\rightarrow \OO^{\times}$ there exists 
a crystalline character $\psi: G_F\rightarrow \OO^{\times}$ lifting $\overline{\kappa}$ such that $\psi(\Art_F(z))= \chi(z)$ for all $z\in \mu_{p^{\infty}}(F)$.
\end{lem} 
\begin{proof} The Artin map $\Art_F: F^{\times}\rightarrow G_F^{\ab}$ of local class field theory allows us to identify characters 
$\psi: G_F \rightarrow \OO^{\times}$ with characters $\psi: F^{\times}\rightarrow \OO^{\times}$. Under this identification $\psi$ is crystalline if and only if $\psi(x)= \prod_{\sigma: F\hookrightarrow L} \sigma(x)^{n_{\sigma}}$ for some integers $n_{\sigma}$ and for all $x\in \OO_F^{\times}$ by \cite[Proposition B.4]{conrad_lift}.

Let $\zeta$ be a generator of the torsion subgroup of $F^{\times}$ and 
let $m$ be the multiplicative order of $\zeta$. Let $\xi$ be a primitive 
$m$-th root of unity in $L$. Then $\kappa(\zeta)= \xi^a$ for some integer 
$a$. Let $\sigma: F\hookrightarrow L$ be an embedding such that $\sigma(\zeta)=\xi$. Let $\psi: F^{\times}\rightarrow \OO^{\times}$ 
be the character $\psi(x)= \sigma( x \varpi_F^{-v(x)})^a$ for all $x\in F^{\times}$, where $v$
is a valuation on $F$ normalized so that $v(\varpi_F)=1$. Then $\psi \kappa^{-1}(\zeta)=1$ and hence $\psi \kappa^{-1}$ is trivial on the torsion subgroup of $F^{\times}$. Moreover, $\psi$ is crystalline by the above. 
Note that $\psi\kappa^{-1}\equiv 1\pmod{\varpi}$. 

For the last part, we may choose any 
$\kappa: G_F\rightarrow \OO^{\times}$ lifting
$\overline{\kappa}$ and satisfying $\kappa(\Art_F(z))= \chi(z)$
for all $z\in \mu_{p^{\infty}}(F)$ and apply the previous
part. 
 \end{proof} 
 
 \begin{lem}\label{delta} Let $\psi: G_F\rightarrow \OO^{\times}$ be a character lifting $\det \rhobar$ and let 
 $x:R_{\det \rhobar}\rightarrow \OO$ be the corresponding $\OO$-algebra homomorphism. Then 
 the centre $Z$ of $G$ acts on $M_{\infty}\otimes_{R_{\det \rhobar}, x}\OO$ via the character $\delta^{-1}$, where $\delta: Z\rightarrow \OO^{\times}$ is the composition 
 $$ Z\overset{\cong}{\longrightarrow} F^{\times}\overset{\Art_F}{\longrightarrow} G_F^{\ab}\overset{\varepsilon^{d(d-1)/2} \psi}{\longrightarrow}\OO^{\times},$$
 where $\varepsilon$ is the $p$-adic cyclotomic character. 
 
 Moreover, 
 $M_{\infty}\otimes_{R_{\det \rhobar}, x}\OO$ 
 is non-zero and projective in the category of 
 linearly compact $\OO\br{K}$-modules 
 on which $Z\cap K$ acts by $\delta^{-1}$.
 
 Further, if $\psi$ is crystalline then there is an algebraic character $\theta: \Res^F_{\Qp} \Gm \rightarrow \Gm$ defined over $L$ such that $\delta|_{K\cap Z}$ is equal to the composition 
 $$ \OO_F^{\times} \hookrightarrow (\Res^F_{\Qp} \Gm)(\Qp)\rightarrow (\Res^F_{\Qp} \Gm)(L)\overset{\theta}{\longrightarrow}\Gm(L),$$
 where $\Res^F_{\Qp}$ denotes the restriction of scalars. 
 \end{lem} 
 \begin{proof} It follows from the discussion at the beginning of \cite[Section 4.22]{6auth} that 
 $Z$ acts via $\delta$ on the Pontryagin dual of  $M_{\infty}\otimes_{R_{\det \rhobar}, x}\OO$. 
 Hence it acts on $M_{\infty}\otimes_{R_{\det \rhobar}, x}\OO$ via $\delta^{-1}$. 
 The second part follows from \cite[Corollary 4.26]{6auth}. The last part follows from 
 \cite[Proposition B.4]{conrad_lift} as explained in the proof of Lemma \ref{cris_char}.
 \end{proof}
 If $V$ is a continuous
 representation of $K$ on a finite
 dimensional $L$-vector space then we define a
 finitely generated $R_{\infty}[1/p]$-module
 $M_{\infty}(V)$ as follows. Since $K$ is compact 
 it stabilizes a bounded $\OO$-lattice $\Theta$ in $V$. Let
 $$M_{\infty}(\Theta):=
 (\Hom^{\cont}_{\OO\br{K}}(M_{\infty},
 \Theta^d))^d,$$
 where $(\cdot)^d:= \Hom^{\cont}_{\OO}
 (\cdot, \OO)$. Then 
 $M_{\infty}(\Theta)$ is a finitely
 generated $R_{\infty}$-module. The module 
 $M_{\infty}(V):=M_{\infty}(\Theta)\otimes_{\OO}L$
 does not depend on the choice of a lattice
 $\Theta$. 
 
 We will denote by $\Irr(G)$ the set
 of equivalence classes of irreducible algebraic
 representation of $(\Res^F_{\Qp} \GL_d)_L$ defined over $L$.  If $\xi\in \Irr(G)$ then
 we will consider it as a representation of $K$ 
 by evaluating at $L$ and letting $K$ act 
 via the composition
 $$K\hookrightarrow (\Res^F_{\Qp} \GL_d)(\Qp)
 \rightarrow (\Res^F_{\Qp} \GL_d)(L).$$
If $M$ is a compact $\OO$-module then we define an $L$-Banach space 
 $$\Pi(M):=\Hom^{\cont}_{\OO}(M, L),$$
equipped with supremum norm. If $M$ is a compact 
$\OO\br{K}$-module, then the action of $K$ on $M$ makes $\Pi(M)$ into  a unitary $L$-Banach space representation of $K$. For example, the map $K\rightarrow \OO\br{K}$ induces an isomorphism of unitary $L$-Banach space representations 
$\Pi(\OO\br{K})\cong \mathcal C(K, L)$, the space 
of continuous functions from $K$ to $L$, with $K$-action given by left translations, \cite[Corollary 2.2]{st_iw}.  

\begin{lem}\label{proj} Let $\theta: \Res^F_{\Qp} \Gm \rightarrow \Gm$ be an algebraic character defined over $L$, and let $\delta: Z\cap K\rightarrow \OO^{\times}$ be the character associated to $\theta$ in Lemma \ref{delta}.
Let $M$  be non-zero and  projective in the category of 
 linearly compact $\OO\br{K}$-modules 
 on which $Z\cap K$ acts by $\delta^{-1}$. Then there is $\xi\in \Irr(G)$ such that $\Hom^{\cont}_K(M, \xi^*)\neq 0$.
 \end{lem}
 \begin{proof}
 We may assume that $M$ is a direct summand of 
 $\OO\br{K}\wtimes_{\OO\br{K\cap Z}, \delta^{-1}}\OO$, 
 since an arbitrary projective module is isomorphic to a product of 
 indecomposable projectives, and these are direct summands of $\OO\br{K}\wtimes_{\OO\br{K\cap Z}, \delta^{-1}}\OO$. Then the Banach space $\Pi(M)$ is 
 a non-zero direct summand $\mathcal C_{\delta}(K, L)$, 
 the subspace of $\mathcal C(K, L)$ on which $Z\cap K$ 
 acts by $\delta$.

Using the theory of highest weight we may find $\tau\in \Irr(G)$ such that the central character of $\tau$ is 
 equal to $\theta$. It follows from \cite[Corollary 7.8]{DPS} that the evaluation map 
 $$\bigoplus_{\xi' \in \Irr(G/Z)} \tau\otimes \xi'\otimes \Hom_K^{\cont}(\tau\otimes \xi', \Pi(M))\rightarrow \Pi(M)$$
 has dense image. Thus there is $\xi'\in \Irr(G/Z)$ 
 and an irreducible summand $\xi$ of $\tau\otimes \xi'$
 such that $\Hom_K^{\cont}(\xi, \Pi(M))\neq 0$. Dually, 
 this implies $\Hom_K^{\cont}(M, \xi^*)\neq 0$. 
 \end{proof} 
 
\begin{thm}\label{faithful_square} The action of $R^{\square}_{\rhobar}$ on $M_{\infty}$ is faithful. 
 \end{thm}
 \begin{proof} Let $\pp$ be a minimal prime 
 of $R^{\square}_{\rhobar}$. We have shown in Corollary \ref{BJ_conjecture} that 
 there is a character
 $\chi: \mu_{p^{\infty}}(F)\rightarrow L^{\times}$
 such that $R^{\square}_{\rhobar}/\pp = R^{\square, \chi}_{\rhobar}$.
 It follows from Lemma \ref{cris_char} that there is a
 crystalline character 
 $\psi: G_F\rightarrow \OO^{\times}$
 lifting $\det \rhobar$ such that 
 $\psi(\Art_F(z))= \chi(z)$ 
 for all $z\in \mu_{p^{\infty}}(F)$. Let 
 $x:R_{\det\rhobar}\rightarrow \OO$ 
 be the corresponding $\OO$-algebra
 homomorphism. It follows from Lemmas 
 \ref{delta}, \ref{proj} that there is 
 $\xi \in \Irr(G)$ such that 
 $$\Hom_K^{\cont}(M_{\infty}\otimes_{R_{\det \rhobar},x}\OO, \xi^*)\neq 0.$$
This implies that 
$M_{\infty}(\xi)\otimes_{R_{\det \rhobar}, x} \OO\neq 0.$
 
 Let $\mathfrak a$ be the $R_{\infty}$ annihilator of
 $M_{\infty}$. In \cite[Theorem 6.12]{EP} it is shown, following the approach of Chenevier \cite{che_inf_fern} and Nakamura \cite{nakamura}, that the closure in $\Spec R_{\infty}$ of the union 
 of the supports of  $M_{\infty}(\xi')$ for all $\xi'\in \Irr(G)$ is a union 
 of irreducible components of $\Spec R_{\infty}$. Thus 
 there is a minimal prime $\qq$ of $R_{\infty}$
 such that $\Supp M_{\infty}(\xi)\subset V(\qq)\subset V(\mathfrak a)$.
 
 Since $M_{\infty}(\xi)\otimes_{R_{\det \rhobar}, x} \OO\neq 0$, Lemma \ref{components_Rinfty} implies that the image of 
 $\qq$ in $\Spec R^{\square}_{\rhobar}$ is equal to $\pp$. Thus
 $\pp$ contains $\mathfrak a \cap R^{\square}_{\rhobar}$, which 
 is the $R^{\square}_{\rhobar}$-annihilator of $M_{\infty}$. Since 
 $R^{\square}_{\rhobar}$ is reduced by Corollary \ref{generic_fibre_normal}, the intersection of all minimal
 prime ideals is zero and hence $R^{\square}_{\rhobar}$ acts
 faithfully on $M_{\infty}$.
 \end{proof}

\begin{proof}[Proof of Theorem \ref{density}] This is proved in the same way as 
\cite[Theorems 5.1, 5.3]{EP}. 
Let us sketch the proof in 
the case of $\Sigma^{\cris}$ for the convenience of the reader. For each $\xi\in \Irr(G)$ let 
$\mathfrak a_{\xi}$ be the 
$R^{\square}_{\rhobar}$-annihilator of $M_{\infty}(\xi)$. 
It follows from \cite[Lemma 4.18]{6auth} that 
$R^{\square}_{\rhobar}/\mathfrak a_{\xi}$ is a quotient 
of the crystalline deformation ring of $\rhobar$ 
with Hodge--Tate weights corresponding to the highest weight of $\xi$, see \cite[Section 1.8]{6auth}, \cite[Remark 5.14]{DPS}. Moreover, $R^{\square}_{\rhobar}/\mathfrak a_{\xi}$ is a union of irreducible components of that ring. This implies that $R^{\square}_{\rhobar}/\mathfrak a_{\xi}$ is reduced and $\OO$-torsion free. The set $\Sigma^{\cris}$ contains the set of maximal ideals of $(R^{\square}_{\rhobar}/\mathfrak a_{\xi})[1/p]$. Since $(R^{\square}_{\rhobar}/\mathfrak a_{\xi})[1/p]$ is Jacobson, if $a\in R^{\square}_{\rhobar}$ is 
contained in the intersection of all maximal ideals 
in $\Sigma^{\cris}$ then $a$ will 
annihilate $M_{\infty}(\xi)$ for all $\xi \in \Irr(G)$. 
The continuous $L$-linear dual of   $M_{\infty}(\xi)$
can be identified with $\Hom_K(\xi, \Pi(M_{\infty}))$. 
The key point is that the image of the evaluation map 
\begin{equation}\label{evaluation}
\bigoplus_{\xi\in \Irr(G)} \xi\otimes_L \Hom_K(\xi, \Pi(M_{\infty}))\rightarrow \Pi(M_{\infty})
\end{equation}
is dense. Thus $a$ will annihilate the left hand side of \eqref{evaluation}, and by density it will annihilate $\Pi(M_{\infty})$. The continuous $L$-linear dual of $\Pi(M_{\infty})$ can 
be identified with $M_{\infty}[1/p]$. Since $R^{\square}_{\rhobar}$ is $\OO$-torsion free and $R^{\square}_{\rhobar}$ acts faithfully on $M_{\infty}$ by Theorem \ref{faithful_square} we deduce that $a=0$. 

If $\Sigma=\Sigma^{\prnc}_{\underline{k}}$ or $\Sigma^{\spcd}_{\underline{k}}$ then the argument is the same, except that instead of considering 
all $\xi\in \Irr(G)$ one fixes $\xi\in \Irr(G)$,
such that the highest weight of $\xi$ corresponds to the Hodge--Tate weights $\underline{k}$ and one considers the family $\xi\otimes_L V$, where $V$ are principal series or appropriate supercuspidal types, see the proof of Theorems 5.1, 5.3 in \cite{EP} for more details.   
\end{proof}
\begin{remar} It is natural to ask whether the ring $R_{\infty}$ acts faithfully on $M_{\infty}$. We cannot answer this question in general, since it boils down to whether every 
irreducible component of the potentially semi-stable rings
$R^{\square, \xi, \tau}_{\tilde{v}}$ (see the proof of Lemma \ref{components_Rinfty}, where $v\in S_p\setminus \pp$ is a place above
$p$, different from the place at which the patching construction is
carried out) has a point corresponding to an automorphic Galois
representation. These questions are connected with modularity lifting theorems and the Fontaine--Mazur conjecture, see \cite[Remark 4.20]{6auth}.

However, if all $R^{\square, \xi, \tau}_{\tilde{v}}$ were integral domains then the ring $A$ in Lemma \ref{components_Rinfty} would also be an integral domain, and we would deduce from 
the proof of Theorem \ref{faithful_square} that $R_{\infty}$ acts faithfully on $M_{\infty}$. A further possibility is to avoid the 
modularity lifting related issues  by carrying out the patching construction of \cite{6auth} 
at all places above $p$ at once. Then the proof of Theorem \ref{faithful_square} 
would carry over in this new setting to show that $R_{\infty}$ acts faithfully on $M_{\infty}$. 
\end{remar} 

\begin{remar}\label{rel_to_BIP2} In \cite{BIP2}, which is a sequel to this paper, we have proved the density 
of $\Sigma^{\cris}$ in the rigid analytic space $(\Spf R^{\square}_{\rhobar})^{\rig}$ 
associated to the formal scheme $\Spf R^{\square}_{\rhobar}$ by making a strong use of 
\cite{EG_stack} to show that $\Sigma^{\cris}$ is non-empty. This in turn implies Theorem 
\ref{density} for $\Sigma^{\cris}$ without any restrictions on the prime $p$, see \cite[Corollary 5.2]{BIP2}. However, the sets $\Sigma^{\prnc}_{\underline{k}}$ and $\Sigma^{\spcd}_{\underline{k}}$ for a fixed regular $\underline{k}$ are not Zariski dense
in $(\Spf R^{\square}_{\rhobar})^{\rig}$ and the argument explained in this section is the only 
known method to approach the density result in Theorem \ref{density} in these cases. We also find Theorem \ref{faithful_square} to be  an interesting result in its own right: if one believes the expectation in \cite[Section 6]{6auth} that 
$M_{\infty}$ should realize the conjectural $p$-adic Langlands correspondence then
Theorem \ref{faithful_square} has to hold. 
\end{remar} 

\appendix
\section{Kummer-irreducible points}\label{appendix}

The purpose of the appendix is to slightly generalize the notion of non-special points in 
$\Xbar^{\ps}= \Spec R^{\ps}/\varpi$ in \cite[Definition 5.1.2]{bj}. 
We use the notation of the main text. In particular, $\zeta_p$ is a primitive $p$-th root of unity in a fixed algebraic closure $\overline{F}$ of $F$. If $x\in \Xbar^{\ps}$ then we let $D_x= D^u\otimes_{R^{\ps}}\overline{\kappa(x)}$, where
$\overline{\kappa(x)}$ is an algebraic closure of the 
residue field at $x$, and we let $\rho_x:G_F\to \GL_d(\overline{\kappa(x)})$ be the semisimple representation whose pseudo-character is~$D_x$.

\begin{defi} We say that $x\in P_1(R^{\ps}/\varpi)$ is \textit{Kummer-irreducible} if the restriction of 
$D_x$ to $G_{F'}$ is absolutely irreducible for all degree $p$ Galois extensions $F'$ of $F(\zeta_p)$. Otherwise, we say that $x$ 
is \textit{Kummer-reducible}.
\end{defi} 
Thus $x$ is Kummer-irreducible if and only 
if $\rho_x|_{G_{F(\zeta_p)}}$ is non-special in the sense of \cite[Definition 5.1.2]{bj}. In particular, if $\zeta_p\in F$ then both notions coincide. 
Our main interest in this notion is the following Lemma. 

\begin{lem}\label{adz0} If $x$ is Kummer-irreducible  then
$H^2(G_F, \adz \rho_x)=0$.
\end{lem} 
\begin{proof} Since the order of $\Gal(F(\zeta_p)/F)$ is prime to $p$ we have 
$$H^2(G_F, \adz \rho_x)\cong H^2(G_{F(\zeta_p)}, \adz \rho_x)^{\Gal(F(\zeta_p)/F)}.$$
Since $x$ is Kummer-irreducible,  the restriction 
of $\rho_x$ to $G_{F(\zeta_p)}$ is non-special, 
and it follows from \cite[Lemma 5.1.1]{bj} that 
$H^2(G_{F(\zeta_p)}, \adz \rho_x)=0$. 
\end{proof}

If $E\subset \overline F$ is a finite extension of $F$ then we denote by 
$R^{\ps}_{E}$ the universal ring for deformations of the pseudo-character 
$\Dbar|_{G_{E}}$. We let $\Xbar^{\ps}_E=\Spec R^{\ps}_{E}$, $U^{\irr}_E$ 
the absolute irreducible locus in $\Xbar^{\ps}_E$
and $U^{\nspcl}_E$ the non-special locus in $U^{\irr}_E$. 
These are open subschemes of $\Xbar^{\ps}_E$. Let $U^{\spcl}_E$ be the complement of $U^{\nspcl}_E$ of $U^{\irr}_E$.
We drop the subscript $E$, when $E=F$.

\begin{lem}\label{lem_PDring-Restr}
If $E\subset \overline F$ is a finite extension of $F$ then the morphism $r: X^{\ps} \rightarrow X^{\ps}_E$, 
induced by restriction of pseudo-characters of $G_F$ to  $G_{E}$, is finite.
\end{lem}
\begin{proof}
The proof is similar to the proof of Proposition \ref{iotaP_finite}. The map $R^{\ps}_E\rightarrow R^{\ps}$ is a local homomorphism of complete local rings with residue field $k$. 
Topological Nakayama's lemma implies that it is enough to show that the fibre ring $S:=R^\ps/\mm_{R^\ps_{E}}R^\ps$ is a
finite dimensional $k$-vector space, which amounts to showing $\Spec S=\{\mm_S\}$. We note that $S$ represents the functor of deformations $D_A:A[G_F]\to A$ of $\Dbar$ to local artinian $k$-algebras $A$ such that $D_A|_{G_{E}}=\Dbar|_{G_{E}}\otimes_k A$.

Let $y$ be any point of $\Spec S$ with associated pseudo-character $D_y$ and semisimple representation $\rho_y: G_F\to \GL_d(\overline{\kappa(y)})$. The restriction $\rho_y|_{G_{E}}$ is semisimple, cf.~\cite[Lemma~2.1.4]{bj}, and its associated pseudo-character is $\Dbar|_{G_{E}}\otimes_k\overline{\kappa(y)}$, so that $\rho_y(G_{E})$ is finite. Hence $\rho_y(G_F)$ is finite, and therefore $D_y$ is defined over a finite field $k'\supset k$. This shows that the corresponding ring map $S\to \kappa(y)$ factors via $k'$, and thus its kernel $y$ is the maximal ideal~$\mm_S$.
\end{proof}

We define the \textit{Kummer-reducible} locus in $U^{\irr}$ as 
$$U^{\Kred}:= U^{\irr} \cap \bigl(\bigcup_{F'} r^{-1} (\Xbar^{\ps}_{F'} \setminus U^{\irr}_{F'}) \bigr),$$
where the union is taken over all degree $p$ Galois extensions $F'$ of $F(\zeta_p)$. Since
there are only finitely many such extensions, $U^{\Kred}$ is closed in $U^{\irr}$. We define 
the \textit{cyclotomic-reducible} locus in $U^{\irr}$ as 
$$U^{\cred}:= U^{\irr}\cap r^{-1}( \Xbar^{\ps}_{F(\zeta_p)} \setminus U^{\irr}_{F(\zeta_p)}).$$
This is also a closed subset of $U^{\irr}$, and is contained in $U^{\Kred}$. If $F$ does not 
contain a primitive $p$-th root of unity then 
$U^{\cred}=U^{\spcl}$, and $U^{\cred} = \emptyset$ otherwise.

\begin{lem} We have $U^{\spcl} \subset U^{\Kred}$. Moreover, the inclusion is an equality if $F$
contains a primitive $p$-th root of unity.
\end{lem}
\begin{proof} If $\zeta_p\in F$ then the definitions of $U^{\Kred}$ and $U^{\spcl}$ coincide. 
If $\zeta_p\not\in F$ then $y\in U^{\spcl}$ if and only if $D_y$ is irreducible and the restriction of $D_y$ to $G_{F(\zeta_p)}$ is reducible. If we further restrict 
$D_y$ to $G_{F'}$, where $F'$ is any degree $p$ Galois extension of $F(\zeta_p)$, then 
the pseudocharacter remains reducible. Thus $y\in U^{\Kred}$.
\end{proof}

\begin{lem}\label{jakob_cl_dim} Let $T$ be a locally closed subset of $U^{\irr}$, let $\overline{T}$ be its closure in $U^{\irr}$ and let $Z$ be its closure in $\Xbar^{\ps}$.
Then $\dim T = \dim \overline{T}$ and $\dim Z= \dim T+1$.
\end{lem} 
\begin{proof} Since $U^{\irr}$ is open 
in $\Xbar^{\ps}$, $T$ is locally closed in  $\Xbar^{\ps}$. Thus $T$ is open in $Z$.
Lemma \ref{jacobson} (5) applied with 
$\Spec R=\Spec S= Z$ and $U=T$ implies that $\dim Z = \dim T+1$. Since $\overline{T}$ is contained in $U^{\irr}$, it does not contain the closed point of $\Xbar^{\ps}$. 
Thus $\overline{T}\subset Z\setminus \{\mm_{R^{\ps}}\}$. Since $Z$ is the spectrum of 
a local ring, $\dim ( Z\setminus \{\mm_{R^{\ps}}\}) = \dim Z -1$. We conclude that 
$\dim \overline{T} \le \dim Z -1 = \dim T$. Since $\overline{T}$ contains $T$, $\dim T\le\dim \overline{T}$. 
\end{proof}
\begin{remar} The equality $\dim T= \dim \overline{T}$ in Lemma \ref{jakob_cl_dim} may also be deduced from \cite[\href{https://stacks.math.columbia.edu/tag/0DRT}{Tag 0DRT}]{stacks-project}, which applies in a more general context.
\end{remar} 

\begin{lem}\label{zeta_p}
We have 
$$\dim U^{\irr}- \dim U^{\cred} \ge \frac{1}{2} d^2[F:\Qp]\ge 2.$$
\end{lem}
\begin{proof} It follows from \cite[Theorem 5.5.1]{bj} and Lemma \ref{jakob_cl_dim} 
that 
\begin{equation}\label{dim_Uirr}
\dim U^{\irr} = d^2[F:\Qp].
\end{equation}
If $\zeta_p\in F$  then $U^{\cred}$ is empty and the required bound follows. If $\zeta_p\not\in F$ then $U^{\cred}=U^{\spcl}$, and it follows from 
\cite[Theorem 5.4.1 (a)]{bj} and Lemma \ref{jakob_cl_dim} that $\dim U^{\spcl} \le \frac{1}{2} d^2[F:\Qp]$.
\end{proof}

\begin{lem}\label{Xprime} We have 
$$ \dim U^{\irr} - \dim U^{\Kred} \ge d[F:\Qp]\ge 2.$$
\end{lem}
\begin{proof} It follows from \cite[Lemma 5.1.1]{bj} that the preimage of 
$ U^{\spcl}_{F(\zeta_p)}$ in $\Xbar^{\ps}$ under the morphism 
$r: \Xbar^{\ps}\rightarrow \Xbar^{\ps}_{F(\zeta_p)}$ from Lemma~\ref{lem_PDring-Restr} with $E=F(\zeta_p)$ is equal to $U^{\Kred}\setminus U^{\cred}$. Thus the induced
morphism $r: U^{\Kred}\setminus U^{\cred} \rightarrow U^{\spcl}_{F(\zeta_p)}$ 
is also finite.  We deduce
$$\dim (  U^{\Kred}\setminus U^{\cred}) \le \dim U^{\spcl}_{F(\zeta_p)}$$
from \cite[\href{https://stacks.math.columbia.edu/tag/01WG}{Tag 01WG}]{stacks-project}.

Since $F(\zeta_p)$ contains a primitive $p$-th root of unity, \cite[Lemma 5.1.1]{bj} implies that if 
$p$ does not divide $d$ then $U^{\spcl}_{F(\zeta_p)}$ is empty, thus $U^{\Kred}=U^{\cred}$ 
and the required bound follows from Lemma \ref{zeta_p}. 

Let us assume 
that $p$ divides $d$. Part (a) of \cite[Theorem 5.4.1]{bj} applied with $K=F(\zeta_p)$ bounds the dimension of 
$U^{\spcl}_{F(\zeta_p)}$ by $\frac{1}{2}d^2
[F(\zeta_p):\Qp]$ from above. The $\frac{1}{2}$ in this estimate appears by estimating $[K':K]\ge 2$ (see the proof of \cite[Theorem 5.4.1]{bj} for the notation; $K'$ corresponds to our $F'$).  If $K$ contains a $p$-th root of unity then it follows from Case II of \cite[Lemma 5.1.1]{bj} that $[K':K]=p$. 
Since $F(\zeta_p)$ contains a $p$-th root of unity,  the argument in the proof of \cite[Theorem 5.4.1]{bj} gives us
$$\dim U^{\spcl}_{F(\zeta_p)}  \le  (d/p)^2[F':\Qp] = \frac{[F(\zeta_p):F]}{p} d^2[F:\Qp].$$
Since $[F(\zeta_p):F]\le p-1$ we conclude that 
$$\dim ( U^{\Kred}\setminus U^{\cred}) \le \frac{p-1}{p} d^2[F:\Qp] .$$
Lemma \ref{jakob_cl_dim} implies that the same bound holds for the dimension of 
the closure of $U^{\Kred}\setminus U^{\cred}$ in $U^{\irr}$. 
Lemma \ref{zeta_p} gives the bound 
$$\dim U^{\cred}\le \frac{1}{2} d^2[F:\Qp] .$$
Thus 
$$\dim U^{\Kred} \le \frac{p-1}{p} d^2[F:\Qp] .$$
Since $\dim U^{\irr} =d^2 [F:\Qp]$ by \eqref{dim_Uirr}, we obtain 
$$\dim U^{\irr}-\dim U^{\Kred} \ge
\frac{d}{p} d [F:\Qp]\ge d[F:\Qp]\ge 2,$$
where we have used that $p$ divides $d$ in the second inequality.
\end{proof} 

\begin{prop}\label{existence_Kirr} There exists an open dense
subscheme $U^{\mathrm{Kirr}}\subset U^{\irr}$ such that 
$x\in P_1(R^{\ps}/\varpi)$ is Kummer-irreducible if and only if
$x$ is a closed point in $U^{\mathrm{Kirr}}$. Moreover, 
$\dim U^{\irr} -\dim (U^{\irr}\setminus U^{\mathrm{Kirr}})\ge d[F:\Qp]\ge 2$.
\end{prop}
\begin{proof} Let $U^{\mathrm{Kirr}}$ be the complement of $U^{\Kred}$ in $U^{\irr}$. 
Since $U^{\Kred}$ is closed in $U^{\irr}$, $U^{\mathrm{Kirr}}$ is open in $U^{\irr}$. If $y\in U^{\irr}$ then $y$ lies in $U^{\Kred}$ if and only 
if there exists a degree $p$ Galois extension $F'$ of $F(\zeta_p)$ such that 
$D_y|_{G_{F'}}$ is reducible. If $y\in \Xbar^{\ps}$ is not the closed point and 
$D_y |_{G_{F'}}$ is irreducible for all such $F'$, then $D_y$ is irreducible and 
hence $y\in U^{\Kirr}$. It follows from Lemma \ref{jacobson} (4) that
$U^{\Kirr}\cap P_1(R^{\ps}/\varpi)$ is the set of closed points in $U^{\Kirr}$. 
The bound for the difference of dimensions follows from Lemma \ref{Xprime}. 
Since $U^{\irr}$ is equi-dimensional by \cite[Theorem 5.5.1]{bj} this implies density.
\end{proof}

\bibliographystyle{plain}
\bibliography{Ref}

\begin{thebibliography}{10}

\bibitem{alper}
Jarod Alper.
\newblock Adequate moduli spaces and geometrically reductive group schemes.
\newblock {\em Algebr. Geom.}, 1(4):489--531, 2014.

\bibitem{unique_top}
V.~I. Arnautov and M.~I. Ursul.
\newblock On the uniqueness of topologies for some constructions of rings and
  modules.
\newblock {\em Sibirsk. Mat. Zh.}, 36(4):735--751, i, 1995.

\bibitem{babnik}
Maurice Babnik.
\newblock Irreduzible {K}omponenten von 2-adischen {D}eformationsr\"{a}umen.
\newblock {\em J. Number Theory}, 203:118--138, 2019.
\newblock Text in English and German.

\bibitem{blght}
Tom Barnet-Lamb, David Geraghty, Michael Harris, and Richard Taylor.
\newblock A family of {C}alabi-{Y}au varieties and potential automorphy {II}.
\newblock {\em Publ. Res. Inst. Math. Sci.}, 47(1):29--98, 2011.

\bibitem{bel_che}
Jo\"{e}l Bella\"{\i}che and Ga\"{e}tan Chenevier.
\newblock Families of {G}alois representations and {S}elmer groups.
\newblock {\em Ast\'{e}risque}, (324):xii+314, 2009.

\bibitem{boeckle}
Gebhard B\"{o}ckle.
\newblock Deformation rings for some mod 3 {G}alois representations of the
  absolute {G}alois group of {$\mathbf{Q}_3$}.
\newblock {\em Ast\'{e}risque}, (330):529--542, 2010.

\bibitem{BIP2}
Gebhard B\"ockle, Ashwin Iyengar, and Vytautas Pa\v{s}k\={u}nas.
\newblock Zariski density of crystalline points.
\newblock {\em Proc. Natl. Acad. Sci. USA}, 120(13):Paper No. e2221042120, 7
  pp, 2023.

\bibitem{bjpp}
Gebhard B\"{o}ckle and Ann-Kristin Juschka.
\newblock Irreducibility of versal deformation rings in the {$(p,p)$}-case for
  2-dimensional representations.
\newblock {\em J. Algebra}, 444:81--123, 2015.

\bibitem{bj}
Gebhard B\"ockle and Ann-Kristin Juschka.
\newblock Equidimensionality of universal pseudodeformation rings in
  characteristic $p$ for absolute {G}alois groups of $p$-adic fields.
\newblock 2023.
\newblock
  \url{https://typo.iwr.uni-heidelberg.de/fileadmin/groups/arithgeo/templates/data/Gebhard_Boeckle/Boeckle-Juschka-Pseudo-20230701.pdf}.

\bibitem{Bourbaki_ac8et9}
N.~Bourbaki.
\newblock {\em \'{E}l\'{e}ments de math\'{e}matique. {A}lg\`ebre commutative.
  {C}hapitres 8 et 9}.
\newblock Springer, Berlin, 2006.
\newblock Reprint of the 1983 original.

\bibitem{BH}
Winfried Bruns and J\"{u}rgen Herzog.
\newblock {\em Cohen--{M}acaulay rings}, volume~39 of {\em Cambridge Studies in
  Advanced Mathematics}.
\newblock Cambridge University Press, Cambridge, 1993.

\bibitem{DDR_Cai}
Yichang Cai.
\newblock Derived deformation rings allowing congruences.
\newblock 2021.
\newblock \url{https://arxiv.org/pdf/2108.13135.pdf}.

\bibitem{R3_factorial}
Frederick Call and Gennady Lyubeznik.
\newblock A simple proof of {G}rothendieck's theorem on the parafactoriality of
  local rings.
\newblock In {\em Commutative algebra: syzygies, multiplicities, and birational
  algebra ({S}outh {H}adley, {MA}, 1992)}, volume 159 of {\em Contemp. Math.},
  pages 15--18. Amer. Math. Soc., Providence, RI, 1994.

\bibitem{6auth}
Ana Caraiani, Matthew Emerton, Toby Gee, David Geraghty, Vytautas
  Pa\v{s}k\={u}nas, and Sug~Woo Shin.
\newblock Patching and the {$p$}-adic local {L}anglands correspondence.
\newblock {\em Camb. J. Math.}, 4(2):197--287, 2016.

\bibitem{CE}
Henri Cartan and Samuel Eilenberg.
\newblock {\em Homological algebra}.
\newblock Princeton Landmarks in Mathematics. Princeton University Press,
  Princeton, NJ, 1999.
\newblock With an appendix by David A. Buchsbaum, Reprint of the 1956 original.

\bibitem{che_unpublished}
Ga\"{e}tan Chenevier.
\newblock Sur la vari\'et\'e des caract\`eres $p$-adique du groupe de galois
  absolu de $\mathbb{Q}_p$.
\newblock
  \url{http://gaetan.chenevier.perso.math.cnrs.fr/articles/lieugalois.pdf},
  2010.

\bibitem{che_inf_fern}
Ga\"{e}tan Chenevier.
\newblock Sur la densit\'{e} des repr\'{e}sentations cristallines de
  {$\text{Gal}(\overline{\mathbb Q}_p/\mathbb Q_p)$}.
\newblock {\em Math. Ann.}, 355(4):1469--1525, 2013.

\bibitem{che_durham}
Ga\"{e}tan Chenevier.
\newblock The {$p$}-adic analytic space of pseudocharacters of a profinite
  group and pseudorepresentations over arbitrary rings.
\newblock In {\em Automorphic forms and {G}alois representations. {V}ol. 1},
  volume 414 of {\em London Math. Soc. Lecture Note Ser.}, pages 221--285.
  Cambridge Univ. Press, Cambridge, 2014.

\bibitem{colmez_tri}
Pierre Colmez.
\newblock Repr\'{e}sentations triangulines de dimension 2.
\newblock Number 319, pages 213--258. 2008.
\newblock Repr\'{e}sentations $p$-adiques de groupes $p$-adiques. I.
  Repr\'{e}sentations galoisiennes et $(\phi,\Gamma)$-modules.

\bibitem{CDP2}
Pierre Colmez, Gabriel Dospinescu, and Vytautas Pa\v{s}k\={u}nas.
\newblock Irreducible components of deformation spaces: wild $2$-adic
  exercises.
\newblock {\em Int. Math. Res. Not. IMRN}, (14):5333--5356, 2015.

\bibitem{mod_lift}
Brian Conrad.
\newblock Modularity lifting seminar webpage, 2010.
\newblock \url{http://virtualmath1.stanford.edu/~conrad/modseminar/}.

\bibitem{conrad_lift}
Brian Conrad.
\newblock Lifting global {G}alois representations with local properties.
\newblock 2011.
\newblock \url{http://math.stanford.edu/~conrad/papers/locchar.pdf}.

\bibitem{DPS}
Gabriel {Dospinescu}, Vytautas {Pa{\v{s}}k{\={u}}nas}, and Benjamin {Schraen}.
\newblock {Infinitesimal characters in arithmetic families}.
\newblock {\em arXiv e-prints}, page arXiv:2012.01041, December 2020.

\bibitem{EG_stack}
Matthew Emerton and Toby Gee.
\newblock {\em Moduli stacks of \'{e}tale ({$\varphi, \Gamma$})-modules and the
  existence of crystalline lifts}, volume 215 of {\em Annals of Mathematics
  Studies}.
\newblock Princeton University Press, Princeton, NJ, [2023] \copyright 2023.

\bibitem{EP}
Matthew Emerton and Vytautas Pa\v{s}k\={u}nas.
\newblock On the density of supercuspidal points of fixed regular weight in
  local deformation rings and global {H}ecke algebras.
\newblock {\em J. \'{E}c. polytech. Math.}, 7:337--371, 2020.

\bibitem{DDR_GV}
S.~Galatius and A.~Venkatesh.
\newblock Derived {G}alois deformation rings.
\newblock {\em Adv. Math.}, 327:470--623, 2018.

\bibitem{Gouvea}
Fernando~Q. Gouv\^{e}a.
\newblock Deformations of {G}alois representations.
\newblock In {\em Arithmetic algebraic geometry ({P}ark {C}ity, {UT}, 1999)},
  volume~9 of {\em IAS/Park City Math. Ser.}, pages 233--406. Amer. Math. Soc.,
  Providence, RI, 2001.
\newblock Appendix 1 by Mark Dickinson, Appendix 2 by Tom Weston and Appendix 3
  by Matthew Emerton.

\bibitem{HS}
Eugen Hellmann and Benjamin Schraen.
\newblock Density of potentially crystalline representations of fixed weight.
\newblock {\em Compos. Math.}, 152(8):1609--1647, 2016.

\bibitem{HP}
Yongquan Hu and Vytautas Pa\v{s}k\={u}nas.
\newblock On crystabelline deformation rings of {${\rm Gal}(\overline{\mathbb
  Q}_p/\mathbb Q_p)$}.
\newblock {\em Math. Ann.}, 373(1-2):421--487, 2019.
\newblock With an appendix by Jack Shotton.

\bibitem{Iy}
Ashwin Iyengar.
\newblock Deformation theory of the trivial mod {$p$} {G}alois representation
  for {${\rm GL}_n$}.
\newblock {\em Int. Math. Res. Not. IMRN}, (22):8896--8935, 2020.

\bibitem{kisin_over}
Mark Kisin.
\newblock Overconvergent modular forms and the {F}ontaine-{M}azur conjecture.
\newblock {\em Invent. Math.}, 153(2):373--454, 2003.

\bibitem{KisinCurrentDevelopments}
Mark Kisin.
\newblock Modularity of 2-dimensional {G}alois representations.
\newblock In {\em Current developments in mathematics, 2005}, pages 191--230.
  Int. Press, Somerville, MA, 2007.

\bibitem{kisin_pst}
Mark Kisin.
\newblock Potentially semi-stable deformation rings.
\newblock {\em J. Amer. Math. Soc.}, 21(2):513--546, 2008.

\bibitem{kisin_GQp_GL2Qp}
Mark Kisin.
\newblock Deformations of {$G_{\mathbb Q_p}$} and {${\rm GL}_2(\mathbb Q_p)$}
  representations.
\newblock {\em Ast\'{e}risque}, (330):511--528, 2010.

\bibitem{matsumura_alg}
Hideyuki Matsumura.
\newblock {\em Commutative algebra}, volume~56 of {\em Mathematics Lecture Note
  Series}.
\newblock Benjamin/Cummings Publishing Co., Inc., Reading, Mass., second
  edition, 1980.

\bibitem{matsumura}
Hideyuki Matsumura.
\newblock {\em Commutative ring theory}, volume~8 of {\em Cambridge Studies in
  Advanced Mathematics}.
\newblock Cambridge University Press, Cambridge, second edition, 1989.
\newblock Translated from the Japanese by M. Reid.

\bibitem{Mazur_GQ}
B.~Mazur.
\newblock Deforming {G}alois representations.
\newblock In {\em Galois groups over {${\bf Q}$} ({B}erkeley, {CA}, 1987)},
  volume~16 of {\em Math. Sci. Res. Inst. Publ.}, pages 385--437. Springer, New
  York, 1989.

\bibitem{Mazur}
Barry Mazur.
\newblock An introduction to the deformation theory of {G}alois
  representations.
\newblock In {\em Modular forms and {F}ermat's last theorem ({B}oston, {MA},
  1995)}, pages 243--311. Springer, New York, 1997.

\bibitem{nakamura1}
Kentaro Nakamura.
\newblock Deformations of trianguline {$B$}-pairs and {Z}ariski density of two
  dimensional crystalline representations.
\newblock {\em J. Math. Sci. Univ. Tokyo}, 20(4):461--568, 2013.

\bibitem{nakamura}
Kentaro Nakamura.
\newblock Zariski density of crystalline representations for any {$p$}-adic
  field.
\newblock {\em J. Math. Sci. Univ. Tokyo}, 21(1):79--127, 2014.

\bibitem{nekovar}
Jan Nekov\'{a}\v{r}.
\newblock Selmer complexes.
\newblock {\em Ast\'{e}risque}, (310):viii+559, 2006.

\bibitem{image}
Vytautas Pa{\v{s}}k\=unas.
\newblock The image of {C}olmez's {M}ontreal functor.
\newblock {\em Publ. Math. Inst. Hautes \'Etudes Sci.}, 118:1--191, 2013.

\bibitem{finite}
Vytautas Pa\v{s}k\={u}nas and Shen-Ning Tung.
\newblock Finiteness properties of the category of {${\rm mod}\,p$}
  representations of {${\rm GL}_2 (\Bbb Q_p)$}.
\newblock {\em Forum Math. Sigma}, 9:Paper No. e80, 39, 2021.

\bibitem{Pro87}
Claudio Procesi.
\newblock A formal inverse to the {C}ayley--{H}amilton theorem.
\newblock {\em J. Algebra}, 107(1):63--74, 1987.

\bibitem{RZ}
Luis Ribes and Pavel Zalesskii.
\newblock {\em Profinite groups}, volume~40 of {\em Ergebnisse der Mathematik
  und ihrer Grenzgebiete. 3. Folge. A Series of Modern Surveys in Mathematics
  [Results in Mathematics and Related Areas. 3rd Series. A Series of Modern
  Surveys in Mathematics]}.
\newblock Springer-Verlag, Berlin, second edition, 2010.

\bibitem{st_iw}
P.~Schneider and J.~Teitelbaum.
\newblock {B}anach space representations and {I}wasawa theory.
\newblock {\em Israel J. Math.}, 127:359--380, 2002.

\bibitem{seshadri}
C.~S. Seshadri.
\newblock Geometric reductivity over arbitrary base.
\newblock {\em Advances in Math.}, 26(3):225--274, 1977.

\bibitem{stacks-project}
The {Stacks Project Authors}.
\newblock \textit{Stacks Project}.
\newblock \url{https://stacks.math.columbia.edu}, 2018.

\bibitem{WE_thesis}
Carl Wang-Erickson.
\newblock Moduli of {G}alois {R}epresentations.
\newblock 2013.
\newblock ProQuest LLC, Ann Arbor, MI, Thesis (Ph.D.)--Harvard University.

\bibitem{WE_alg}
Carl Wang-Erickson.
\newblock Algebraic families of {G}alois representations and potentially
  semi-stable pseudodeformation rings.
\newblock {\em Math. Ann.}, 371(3-4):1615--1681, 2018.

\end{thebibliography}

\end{document}